\documentclass[11pt]{article}
\usepackage[utf8]{inputenc}

\usepackage{mathtools}
\usepackage{amsfonts}
\usepackage{amsmath}
\usepackage{amsthm}
\usepackage{amssymb}
\usepackage{graphicx}
\usepackage{hyperref}
\usepackage{arxiv}

\usepackage{float}
\usepackage{subcaption}

\DeclarePairedDelimiter\floor{\lfloor}{\rfloor}
\newcommand{\oN}{{\mathbb N}}
\newcommand{\oR}{{\mathbb R}}

\newcommand{\ignore}[1]{}

\usepackage{pgf,tikz,pgfplots}
\newcommand\bovermat[2]{%
  \makebox[0pt][l]{$\smash{\overbrace{\phantom{%
    \begin{matrix}#2\end{matrix}}}^{\text{#1}}}$}#2}    
\usepackage{mathrsfs}

\usetikzlibrary{arrows}

\providecommand{\subclass}[1]{\textbf{AMS subject classification} #1}

\title{Exactness of Parrilo's conic approximations for copositive matrices and associated low order  bounds for the stability number of a graph }

\author {Monique Laurent
\thanks{Centrum Wiskunde \& Informatica (CWI), Amsterdam, and Tilburg University. \url{monique.laurent@cwi.nl} }
\And 
Luis Felipe Vargas
\thanks{Centrum Wiskunde \& Informatica (CWI), Amsterdam. \url{luis.vargas@cwi.nl}
\newline
This work is supported by the European Union's Framework Programme for Research and Innovation Horizon
2020 under the Marie Skłodowska-Curie Actions Grant Agreement No. 813211  (POEMA).
} 
}

\newtheorem{theorem}{Theorem}[section]
\newtheorem{coro}[theorem]{Corollary}
\newtheorem{lemma}[theorem]{Lemma}
\newtheorem{conjecture}[theorem]{Conjecture}
\newtheorem{question}[theorem]{Question}
\newtheorem{proposition}[theorem]{Proposition}
\newtheorem{definition}[theorem]{Definition}

\newtheorem{remark}[theorem]{Remark}
\newtheorem{example}[theorem]{Example}
\newcommand{\kzero}{\mathcal{K}_n^{(0)}}

\newcommand{\kzeroc}{\MK^{(0)}\text{-certificate}}
\newcommand{\konec}{\MK^{(1)}\text{-certificate}}

\usepackage{xcolor}
\makeatletter
\def\tagform@#1{\maketag@@@{$($#1$)$\@@italiccorr}}
\makeatother
\numberwithin{equation}{section}

\makeatletter
\renewcommand*\env@matrix[1][*\c@MaxMatrixCols c]{%
	\hskip -\arraycolsep
	\let\@ifnextchar\new@ifnextchar
	\array{#1}}
\makeatother

\newcommand{\supp}{\text{\rm Supp}}
\newcommand{\COP}{{\text{\rm COP}}}
\newcommand{\MK}{{\mathcal K}}
\newcommand{\MS}{{\mathcal S}}

\newcommand{\rrank}{\vartheta\text{\rm -rank}}

\newcommand{\MoL}{\textcolor{red}}

\begin{document}
\maketitle

\begin{abstract}
De Klerk and Pasechnik~(2002) introduced the bounds $\vartheta^{(r)}(G)$ ($r\in \oN$) for  the stability number $\alpha(G)$ of a graph $G$ and conjectured exactness at order $\alpha(G)-1$:  $\vartheta^{(\alpha(G)-1)}(G)=\alpha(G)$.  These bounds rely on the conic approximations  $\MK_n^{(r)}$  by Parrilo~(2000) for the copositive cone $\COP_n$.
A difficulty in the convergence analysis of  $\vartheta^{(r)}$ is the bad behaviour of the cones $\MK_n^{(r)}$ under adding a zero row/column: when applied to a matrix not in $\MK^{(0)}_n$ this gives a matrix not in any $\MK^{(r)}_{n+1}$, thereby showing strict inclusion $\bigcup_{r\ge 0}\MK^{(r)}_n\subset \COP_n$  for $n\ge 6$.
We investigate the graphs with  $\vartheta^{(r)}(G)=\alpha(G)$ for $r=0,1$: we  algorithmically reduce testing exactness of $\vartheta^{(0)}$  to  acritical graphs, we characterize  critical graphs with $\vartheta^{(0)}$ exact, and we  exhibit   graphs for which  exactness of  $\vartheta^{(1)}$ is not preserved under  adding  an isolated node. This disproves a conjecture by Gvozdenovi\'c and Laurent~(2007) which, if true, would have implied   the above conjecture by de Klerk and Pasechnik.
\end{abstract}

\keywords{stable set problem \and $\alpha$-critical graph \and sum-of-squares polynomial \and copositive matrix \and semidefinite programming \and Shor relaxation}
\subclass{ 05Cxx; 90C22; 90C26; 90C27; 90C30; 11E25}

\section{Introduction}
The problem of computing the {\em stability number} $\alpha(G)$ of a graph $G=(V=[n],E)$, defined as the maximum cardinality of a stable set in $G$, is a central problem in combinatorial optimization with a wide range of applications (e.g., to scheduling, social networks analysis, genetics and chemistry, see \cite{Bomze-survey}, \cite{WuHao}, \cite{Hossain} and references therein).
This problem is well-known to be NP-hard \cite{Karp}, which motivates the study of tractable approximations obtained by means of linear or semidefinite relaxations.  In this paper we investigate some semidefinite bounds $\vartheta^{(r)}(G)$ ($r\in \oN$) that were introduced in \cite{dKP2002}, with a special focus on the question of understanding for which graphs the bounds are exact, especially for low order $r=0$ and $r=1$.  Exactness of the bounds is closely related to the question whether certain associated graph matrices $M_G$ admit copositivity certificates of semidefinite type or, equivalently, 
whether certain associated graph polynomials $F_G$ admit nonnegativity  certificates in terms of sums of squares.

The starting point to define these notions is the following copositive reformulation from \cite{dKP2002} for the stability number:
\begin{equation}\label{copoalpha}
\alpha(G)=\min\{t: t(I+A_G)-J \in \COP_n\}.
\end{equation}
Here, $A_G$, $I$ and $J$ denote, respectively, the adjacency matrix of $G$,  the identity matrix and  the all-ones matrix, and $\COP_n$ is the  cone of copositive matrices defined as
$$\COP_n=\{M\in \mathcal S^n: (x^{\circ 2})^TMx^{\circ 2}\ge 0\ \text{ for all } x\in \oR^n\},
$$ setting 
$x^{\circ 2}=(x_1^2,\ldots,x_n^2)$.
Since the minimum is attained in program (\ref{copoalpha})   the following {\em graph matrix }
\begin{equation}\label{graph-matrix}
M_G:=\alpha(G)(A_G+I)-J
\end{equation}
is copositive or, 
 equivalently,  the following {\em graph polynomial }
\begin{equation}\label{graph-pol}
F_G(x):=(x^{\circ2})^TM_Gx^{\circ2}
\end{equation}
is nonnegative on $ \mathbb{R}^n$. 
A natural question is whether there exist certificates for copositivity of $M_G$ based on semidefinite programming and whether there exist certificates of nonnegativity for $F_G$ based on sums of squares of polynomials. 
Such certificates can be designed using the hierarchy of inner approximations for the copositive cone $\COP_n$ proposed by Parrilo \cite{Parrilo-thesis-2000}, and defined by
\begin{equation}\label{coneKr}
\MK_n^{(r)}=\{M \in \mathcal S^n:  \big(\sum_{i=1}^nx_i^2\big)^r (x^{\circ2})^TMx^{\circ 2} \in \Sigma\}\quad \text {for } r\in \oN,
\end{equation}
 where $\Sigma$ denotes the cone of sums of squares of polynomials.
These cones satisfy 
$\MK_n^{(r)}\subseteq \MK_n^{(r+1)}\subseteq \COP_n$ and they cover the interior of the copositive cone:
\begin{equation}\label{inclusion}
\text{int}(\COP_n) \subseteq \bigcup_{r\geq0}\MK_n^{(r)}\subseteq \COP_n.
\end{equation}
Starting from the copositive formulation (\ref{copoalpha})  and using the cones $\mathcal K^{(r)}_n$, de Klerk and Pasechnik \cite{dKP2002} introduced the following hierarchy of approximations for $\alpha(G)$:
\begin{equation}\label{eqthetar}
\vartheta^{(r)}(G)=\min\{t : t(A_G+I)-J \in \MK_n^{(r)}\},
\end{equation} 
which satisfy  $\alpha(G)\leq \vartheta^{(r+1)}(G) \leq \vartheta^{(r)}(G)$ for all $r\in \mathbb{N}$ and $\lim_{r\to \infty}\vartheta^{(r)}(G)=\alpha(G)$. Note  the minimum is indeed attained in program (\ref{eqthetar}). As sums of squares of polynomials can be modelled using semidefinite programming each bound $\vartheta^{(r)}(G)$ is defined via a semidefinite program. The bound is said to be {\em exact at order} $r$ if $\vartheta^{(r)}(G)=\alpha(G)$.

Yet another useful notion is the parameter $\rrank(G)$, called the {\em $\vartheta$-rank} of $G$, which is defined in \cite{LV2021} as the smallest integer $r$ for which $\vartheta^{(r)}(G)=\alpha(G)$, setting $\rrank(G)=\infty$ if no such $r$ exists.

For clarity let us  summarize  the following links 
between the above notions: for any integer $r\in \oN$ we have
\begin{equation}\label{equiv}
M_G\in \mathcal K^{(r)}_n \Longleftrightarrow \big(\sum_{i=1} ^n x_i^2\big)^r F_G \in \Sigma\Longleftrightarrow \vartheta^{(r)}(G)=\alpha(G)\Longleftrightarrow \rrank(G)\le r.
\end{equation}
De Klerk and Pasechnik \cite{dKP2002} conjectured that the hierarchy $\vartheta^{(r)}(G)$ converges to $\alpha(G)$ in at most $\alpha(G)-1$ steps, which would show that this continuous copositive-based hierarchy has the same convergence behaviour as the Lasserre hierarchy based on discrete formulations of $\alpha(G)$ \cite{Lasserre2001b,Laurent2003}.
 In view of (\ref{equiv}) this can be reformulated as follows.

\begin{conjecture}[\cite{dKP2002}] \label{conj1}
For a graph  $G$, any of the following equivalent claims holds: (i) $M_G\in \MK_n^{(\alpha(G)-1)}$, \\ (ii) $(\sum_{i=1}^n x_i^2)^{\alpha(G)-1}F_G\in \Sigma,$\quad  (iii)  $\vartheta^{(\alpha(G)-1)}(G)=\alpha(G)$,\quad (iv) $\rrank(G)\leq \alpha(G)-1$.
\end{conjecture}
The weaker conjecture asking  whether finite convergence holds at {\em some} order $r\in \oN$ is also open.
\begin{conjecture}[\cite{LV2021}]\label{conj2}
For a graph  $G$, any of the following equivalent claims holds: (i) $M_G\in\bigcup_{r\in \oN}\mathcal K^{(r)}_n$,
\\ (ii)  $(\sum_{i=1}^n x_i^2)^{r}F_G\in \Sigma$ for some $r\in\oN$,\quad
(iii) $\vartheta^{(r)}(G)=\alpha(G)$ for some $r\in\oN$,\quad (iv)   $\rrank(G)< \infty$.
  \end{conjecture}

Let us recap some of the main known results about these conjectures. In \cite{GL2007} Conjecture 1 was shown to hold for all graphs with $\alpha(G)\leq 8$ (see also  \cite{PVZ2007} for the case $\alpha(G)\le 6$).
In \cite{LV2021} it was observed that it suffices to prove both Conjectures~1 and  2 for the class of critical graphs, i.e., for the graphs $G$ satisfying $\alpha(G\setminus e)=\alpha(G)+1$ for all edges $e$ of $G$.
 In addition, it is shown in \cite{LV2021}  that Conjecture 2 holds for acritical graphs, i.e., for the graphs $G$ satisfying  $\alpha(G\setminus e)=\alpha(G)$ for all edges. 

\paragraph*{Some possible directions for resolving Conjectures \ref{conj1} and \ref{conj2}.}
In what follows we mention some possible strategies that could be followed to attack the above two conjectures along with their pitfalls.

A first idea is to investigate whether one can exploit the fact that any graph matrix $M_G$ has its diagonal entries that all take the same value (equal to $\alpha(G)-1$). 
Indeed  it is conjectured in \cite{DDGH} that any copositive matrix with diagonal entries 0 or 1 belongs to some cone $\MK_n^{(r)}$ and it is shown that this is true for matrix size $n=5$ (with $r=1$ in that case).
Hence a positive answer to this conjecture would immediately imply that $M_G $ belongs to some cone $\mathcal K^{(r)}_n$ and thus settle Conjecture~\ref{conj2}.  However, we will disprove the above conjecture from \cite{DDGH} for  matrix size $n\geq 6$ (see Section~\ref{parrilo-cones}).  In particular, this shows that the inclusion $\bigcup_{r\geq0}\MK_n^{(r)}\subseteq \COP_n$ in (\ref{inclusion}) is strict for any $n\geq 6$.

A second possible strategy is to consider the impact of adding an isolated node. 
Let $G\oplus i_0$ denote the graph obtained by adding $i_0$ as an isolated node to $G$.
Consider the following two conjectures. 

\begin{conjecture}[\cite{GL2007}]\label{conjisolated}
For any graph $G$, we have $\rrank(G\oplus i_0)\le \rrank (G)$.
\end{conjecture}
\begin{conjecture}\label{conjfinite}
For any graph $G$, $\rrank(G)<\infty$  implies  $\rrank(G\oplus i_0)<\infty$.
\end{conjecture}
Conjecture \ref{conjisolated}  is  in fact posed in \cite{GL2007} in a more general form (see \cite[Conjecture 4]{GL2007}.
 In addition,  it is shown in \cite{GL2007}  that Conjecture \ref{conjisolated} implies Conjecture \ref{conj1}.
We will show that Conjecture \ref{conjfinite} is in fact equivalent to  Conjecture \ref{conj2} (see Proposition \ref{adding-iso-rk-f}). 
In an attempt to relate  $\rrank(G\oplus i_0)$ and $\rrank(G)$ let us consider the following decomposition of the graph matrices,   proposed  in \cite{GL2007}, where we set  $\alpha:=\alpha(G)$ so that $\alpha(G\oplus i_0)=\alpha+1$:
\begin{equation}\label{M_G-isolated0}
M_{G\oplus i_0}
=\left(\begin{matrix} \alpha  & -1 \cr -1 & (\alpha+1) (I+A_{G})-J\end{matrix}\right)
= 
\left(\begin{matrix} \alpha  & -1 \cr -1 & {1\over \alpha}J\end{matrix}\right)
+
{\alpha+1\over \alpha}
\left(\begin{matrix} 0 & 0 \cr 0 & \alpha(I+A_{G})-J\end{matrix}\right).
\end{equation}
If the operation of adding a zero row/column  preserves  membership in the cones $\MK^{(r)}$ then, in view of  (\ref{M_G-isolated0}), it would immediately follow that 
$M_G\in \mathcal K^{(r)}_n$ implies $M_{G\oplus i_0}\in \mathcal K^{(r)}_{n+1}$, which would show Conjecture \ref{conjisolated} (and thus also Conjecture \ref{conj1}). In addition, if adding a zero row/column  preserves  membership in the union $\bigcup_{r\in\oN}\mathcal K^{(r)}$, then again in view of (\ref{M_G-isolated0}), Conjecture \ref{conjfinite} would be true and thus Conjecture \ref{conj2} too.
However, adding
 a zero row/column does {\em not} in general preserve membership in the cones $ \mathcal K^{(r)}$ for a given order $r\ge 1$ (while this is clearly true for order $r=0$); this was observed (numerically) for order $r=1$ using the graph matrix $M_{C_5}\in \mathcal K^{(1)}_5$ of the $5$-cycle (see \cite{dKP2002}).
We will show that also the second property fails:  adding a zero row/column to a matrix $M\in \bigcup_{r\in \oN}\mathcal K^{(r)}_n\setminus  \MK_n^{(0)}$ produces a matrix that does not belong to the union $\bigcup_{r\in \oN} \mathcal K^{(r)}_{n+1}$ (see Theorem \ref{theo-cop}).

Motivated by the above observations, our focus in this paper is to  investigate the following topics: the impact of adding a zero row/column to a matrix in $\bigcup_{r\in\oN}\mathcal K^{(r)}_n\setminus \mathcal K^{(0)}_n$ (in Section \ref{parrilo-cones}), the behaviour of the $\rrank$ under  some simple graph operations in relation to Conjectures \ref{conj1} and \ref{conj2} (in Section \ref{pre-iso-critical}),  structural properties of the graphs with  $\rrank(G)=0$ (in Section \ref{sec-rank-0}), and   the impact of adding an isolated node to a graph $G$ with $\rrank (G)=1$ (in Section~\ref{sec-k1}). We now give some more details about the last two topics. 

\paragraph*{Graphs with small $\rrank$ 0 or 1.}
In order to investigate the graphs with small $\rrank (G)=0$ or 1 we will use the explicit characterizations of the cones $\mathcal K^{(0)}_n$ and $\mathcal K^{(1)}_n$ provided by Parrilo \cite{Parrilo-thesis-2000}.
There it is shown that a matrix $M\in \mathcal S^n$ belongs to $\MK^{(0)}_n$ if and only if $M$ admits a decomposition $M=P+N$ with  $P\succeq 0$, $N\geq 0$ and $N_{ii}=0$ for all $i\in [n]$; we call such matrix $P$ a {\em $\kzeroc$ for $M$}.
This in particular permits to show that the bound $\vartheta^{(0)}(G)$ coincides with the bound $\vartheta'(G)$, which is the Lovasz' theta number strenghtened by adding a nonnegativity constraint (see \cite{dKP2002}).   Parrilo \cite{Parrilo-thesis-2000} also showed that $M\in \MK_n^{(1)}$ if and only if there exist positive semidefinite matrices $P(1), P(2), \dots, P(n)$ satisfying certain linear constraints (see Lemma \ref{lemK1}); we say that such matrices form a {\em $\konec$ for $M$}. We exploit the structure of the zeros of the quadratic form $x^TMx$ to obtain information about the kernels of  $\MK^{(0)}$- and $\konec$s for $M$. This information plays a crucial  role 
in our study of the graphs with $\rrank(G)=0$ or 1, i.e, for which $M_G$ belongs to $\MK_n^{(0)}$ or $\MK_n^{(1)}$. In some cases it permits to show uniqueness of the certificates, a useful property for the study of the $\rrank$. As an example,  the graph matrix $M_{C_5}$ of the 5-cycle has a unique $\mathcal K^{(1)}$-certificate and this uniqueness  property  permits to characterize the diagonal scalings of $M_{C_5}$ that belong to $\mathcal K^{(1)}_5$ (see Section~\ref{secscalingC5}).

Our main results are as follows. We characterize the critical graphs with $\rrank$ 0 as the disjoint unions of cliques,  and we reduce the problem of deciding whether a graph has $\rrank$ 0 to the same problem for the class of acritical graphs (see Section~\ref{sec-rank-0}). This reduction can be done in polynomial time for the class of graphs $G$ with fixed value of $\alpha(G)$. In addition we show  that adding an isolated node to a graph with $\rrank$ 1 may produce a graph with $\rrank$ at least 2, thus disproving Conjecture \ref{conjisolated} above. 
We also  characterize the maximum number of isolated nodes that can be added to some graphs with $\rrank$ 1 (such as odd cycles and their complements) while preserving the $\rrank$ 1 property (see Section~\ref{sec-k1}). For example, for the graph $C_5$  this maximum number of nodes is shown to be equal to 8.
Here too we will exploit uniqueness properties of some of the matrices arising in $\mathcal K^{(1)}$-certificates.

The study of the graphs with  $\rrank$ 0 is also relevant to the question of understanding when the basic semidefinite relaxation (also known as the Shor relaxation) of  a quadratic (or, more generally, polynomial) optimization problem is exact. This question has received increased attention in the recent years. We refer, e.g., to the works \cite{BY2020,GY2021,WK2021} (and references therein), which investigate this question for various classes of quadratic problems, such as random instances in \cite{BY2020} and standard quadratic programs in \cite{GY2021}.
In fact, thanks to a reformulation of $\alpha(G)$  as the optimum value of a suitable  polynomial optimization problem (involving degree $2r+2$ forms), it turns out that the parameter $\vartheta^{(r)}(G)$ can also be viewed as the optimum value of the Shor relaxation of this polynomial optimization problem (see \cite[Section 6.3]{GL2007}). Hence, also Conjectures \ref{conj1} and  \ref{conj2} can be seen in the light of understanding exactness of  Shor relaxations.
 
Yet another  motivation for the study of the graphs with $\rrank$ 0 comes from its relevance to fundamental questions in complexity theory.
Deciding whether a graph $G$ has $\rrank(G)=0$  indeed amounts to deciding whether the polynomial $F_G(x)=(x^{\circ2})^TM_Gx^{\circ2}$ is a sum of squares, i.e, whether  an associated  semidefinite program is feasible. Equivalently, as mentioned above, $\rrank(G)=0$ if and only if  there exists a positive semidefinite matrix $P\in \mathcal S^n$ satisfying the linear constraints: $P_{ii}=\alpha(G)-1$ for $i\in V$ 
and $P_{i,j}\leq -1$ for $\{i,j\}\notin E$, which thus again asks about the feasibility of a semidefinite program. Recall that  the complexity status of deciding feasibility of a semidefinite program is still unknown. On the positive side it was shown in \cite{PorkolabKhachiyan} that one can test feasibility of a semidefinite program involving matrices of size $n$ and with $m$ linear constraints  in polynomial time when $n$ or $m$ is fixed.
In addition,   it was shown in \cite{Ramana} that this problem  belongs to the class NP if and only if it belongs to co-NP.  Understanding the complexity status for the class of semidefinite programs related to the question of testing whether $\rrank(G)=0$ offers a rich playground to be explored later.

\paragraph*{Organization of the paper.}
The paper is organized as follows. In Section \ref{preliminaries} we group some preliminary results. In particular, we recall the characterization of the cones $\MK_n^{(0)}$ and $\MK_n^{(1)}$ from \cite{Parrilo-thesis-2000} and we give some structural properties of the matrices  arising in $\mathcal K^{(0)}$- and $\mathcal K^{(1)}$-certificates for  membership in these cones. 
We also  recall a characterization for the minimizers of the Motzkin-Straus formulation  (\ref{motzkin-form}) for $\alpha(G)$.
In Section \ref{parrilo-cones}, we  provide explicit constructions showing that adding a zero row/column to a matrix in $\bigcup_{r\ge 0}\MK^{(r)}_n\setminus \mathcal K^{(0)}_n$ may produce a matrix in $\COP_{n+1}\setminus \bigcup_{r\ge 0} \MK^{(r)}_{n+1}$, thereby showing strict  inclusion $\bigcup_{r\ge 0}\MK^{(r)}_n \subset \COP_n$ for any $n\ge 6$.
We also construct copositive matrices with an all-ones diagonal that do not belong to any cone $\MK^{(r)}_n$ for $n\ge 7$, thereby disproving a conjecture from \cite{DDGH}. Exploiting the fact that the graph matrix $M_{C_5}$ admits a unique $\mathcal K^{(1)}$-certificate, we can characterize the diagonal scalings of  $M_{C_5}$ that still belong to $\mathcal K^{(1)}_5$.
 In Section \ref{pre-iso-critical} we present some known and new results dealing with the behavior on the $\rrank$ under  simple graph operations like adding an isolated node and deleting an acritical edge, and we investigate their relevance for Conjectures \ref{conj1} and  \ref{conj2}.
 In Section \ref{sec-rank-0} we discuss the role of critical edges in the study of the graphs with $\rrank$ 0.
  In particular, we characterize the critical graphs with $\rrank$ 0 and we give an algorithmic procedure that reduces the problem of deciding whether a graph has $\rrank$ 0 to the same problem restricted to graphs with no critical edges. In Section \ref{sec-k1} we develop some tools using criticality (as well as symmetry and kernel properties) to study the impact of adding isolated nodes to graphs with $\rrank$ 1. As an application we can characterize how many isolated nodes can be added to an odd cycle (or its complement) while preserving the $\rrank$ 1 property. As a byproduct, we show that adding an isolated node can increase the $\rrank$, thereby refuting Conjecture~\ref{conjisolated}. 

\paragraph*{Notation.}
Given a graph $G=(V,E)$, a set $S\subseteq V$ is  {stable} (aka {independent}) if $S$ does not contain any  edge of $G$. Then, $\alpha(G)$ denotes the maximum cardinality of a stable set, called  the {\em stability number} of $G$. For a subset $U\subseteq V$,  $G[U]$ denotes the induced subgraph of $G$, with vertex set $U$ and edge set $\{\{i,j\}\in E: i,j\in U\}$ and, given an edge $e\in E$, $G\setminus e=(V,E\setminus \{e\})$ is the subgraph obtained by deleting the edge $e$.
An edge $e\in E$ is {\em critical} if $\alpha(G\setminus e)=\alpha(G)+1$  and $e$ is called  {\em acritical} otherwise.
We say that $G$ is {\em critical} if all its edges are critical and that $G$ is {\em acritical} if it has no critical edges. A set $C\subseteq V$ is a clique if $\{i, j\}\in E$ for all $i\ne j\in C$ and the maximum cardinality of a clique is  $\omega(G)=\alpha(\overline G)$. Then $\chi(G)$ (resp., $\chi(\overline G)$) denotes the minimum number of stable sets (resp., cliques) whose  union is $V$. For convenience we also set $\overline \chi(G)=\chi(\overline G)$. Clearly one has $\omega( G)\le \chi(G)$ and $\alpha(G)\le \chi(\overline G)$. 
Recall that a graph $G$ is called {\em perfect} if $\chi(H)=\omega(H)$ for every induced subgraph $H$ of $G$. The celebrated {\em strong perfect graph theorem} of Chudnovsky et al. \cite{SPGT} shows that $G$ is perfect if and only if $G$ does not contain an odd cycle $C_{2n+1}$ or its complement $\overline{C_{2n+1}}$ ($n\ge 2$) as an induced subgraph. For a node $i\in V$,   $N(i)$ denotes the set of nodes $j\in V$ that are adjacent to $i$ and   $i^\perp:= \{i\}\cup N(i)$ is the closed neighborhood of $i$; then $i$ is called an isolated node if $N(i)=\emptyset$. For a subset $S\subseteq V$  set $N_S(i)=N(i)\cap S$. For a graph $G$ and a node $i_0\not\in V$, $G\oplus i_0=(V\cup\{i_0\},E)$ denotes the graph obtained by adding the isolated node $i_0$ to $G$. In general, given two graphs $G$ and $H$,  the graph $G\oplus H=(V(G)\cup V(H), E(G)\cup E(H))$ denotes the disjoint union of  $G$ and $H$.  

We let $\mathcal S^n$ denote the set of $n\times n$ symmetric matrices. For a matrix $M\in \mathcal{S}^n$, we write $M\succeq 0$ if it is positive semidefinite (i.e., $x^TMx\ge 0$ for all $x\in\oR^n$) and $M\geq 0$ if all its entries are nonnegative. For a set $S\subseteq [n]$,  $M[S]$ denotes the principal submatrix of $M$ whose rows and columns are indexed by $S$. Throughout  $J_n$, $I_n$ denote the all-ones matrix and the identity matrix of size $n$ and we may omit the subscript $n$ when the size is not important or clear from the context. For integers $m,n\ge 1$,  $J_{m,n}$ denotes the $m\times n$ all-ones matrix. Throughout  $e$ denotes the all-ones vector (of appropriate size).
For a vector $x\in \mathbb{R}^n$, $\supp(x)=\{i\in [n] : x_i\ne 0\}$ denotes its support.  
The adjacency matrix $A_G\in \mathcal S^n$ of a graph $G=(V=[n],E)$ 
has entries $(A_G)_{ij}=1$ if $\{i,j\}\in E$ and zero otherwise.

Throughout $\oR[x]=\oR[x_1,\ldots,x_n]$ denotes the set of $n$-variate polynomials and $\Sigma$ is the set of sums of squares of polynomials, i.e.,  of the form $p_1^2+\ldots+p_m^2$ for some $m\in \oN$ and $p_1,\ldots,p_m\in \oR[x]$.
The degree of a polynomial $f\in \oR[x]$ is the largest degree $d$ of its terms and $f$ is said to be homogeneous of degree $d$ if all its terms have degree $d$.

\section{Preliminaries on the cones $\MK_n^{(r)}$}\label{preliminaries}


Recall that the cone $\MK_n^{(r)}$ consists of the matrices  $M\in \MS^n$ for which the polynomial $(\sum_{i=1}^{n}x_i^2)^r((x^{\circ2})^TMx^{\circ2})$ is a sum of squares of polynomials. A useful characterization for  matrices in $\MK_n^{(r)}$ is given by the following general result.  

\begin{theorem}[Pe\~{n}a et al. \cite{PVZ2007}]\label{theo-squares-vera}
Let $q\in\oR[x]$ be a homogeneous polynomial of degree $d$ and define the degree $2d$ polynomial $Q(x):=q(x^{\circ 2})=q(x_1^2,\ldots,x_n^2)$. Then, $Q\in \Sigma$  if and only if $q$ can be decomposed as
\begin{align}
    q(x)=\sum_{\substack{I\subseteq[n] \\
    |I|\leq d, |I|\equiv d \ (\text{\rm mod } 2)}}\sigma_I(x)\prod_{i\in I}x_i,
\end{align}
where $\sigma_I$ is a homogeneous polynomial with  degree  $d-|I|$ and $\sigma_I\in \Sigma$.
\end{theorem}
As an application, $M\in \MK_n^{(0)}$ if and only if there exist a matrix $P\succeq 0$ and scalars $c_{ij}\geq 0$ for $1\leq i < j \leq n$ such that 
\begin{equation}\label{eq-k0}
x^TMx = x^TPx + \sum_{0\leq i< j \leq n}c_{ij}x_ix_j.
\end{equation}
This corresponds to the characterization of the cone $\MK_n^{(0)}$ given by Parrilo in \cite{Parrilo-thesis-2000}, which reads
\begin{equation}\label{kzero-p+n}
\MK_n^{(0)}=\{P+N \text{ : } P\succeq 0, N\geq 0\}.
\end{equation}
Note that in (\ref{kzero-p+n}) we can indeed assume, without loss of generality, that $N_{ii}=0$ for all $i\in[n]$. We say that $P$ is a {\em $\kzeroc$}  for $M$ if  $P\succeq 0$, $P\leq M$ and $P_{ii}=M_{ii}$ for all $i\in[n]$. In other words, $P$ is a $\MK^{(0)}$-certificate for $M$ if there exist scalars $c_{ij}\ge 0$ for $1\leq i<j\leq n$ for which Eq. (\ref{eq-k0}) holds.

Similarly, using  Theorem \ref{theo-squares-vera}, $M\in \MK_n^{(1)}$ if and only if there exist matrices $P(i)\succeq 0$ for $i\in [n]$ and scalars $c_{ijk}\geq 0$ for distinct $i,j,k\in [n]$  such that 
\begin{equation}\label{eqpM}
\Big(\sum_{i=1}^{n} x_i\Big)x^TMx= \sum_{i=1}^{n} x_i x^TP(i) x +\sum_{1\leq i<j<k\leq n} c_{ijk} x_ix_jx_k.
\end{equation}
From this, we get  the characterization of the cone $\MK^{(1)}_n$ from Parrilo \cite{Parrilo-thesis-2000} (see also \cite{dKP2002}).

\begin{lemma}\label{lemK1}
A matrix $M$ belongs to the cone $\MK^{(1)}_n$ if and only if there exist matrices $P(i)\succeq 0$ for $i\in [n]$ and scalars $c_{ijk}\geq 0 $ for $1\leq i< j< k\leq n$ satisfying Equation (\ref{eqpM}). Equivalently, there exist matrices $P(i)\in \MS^n$ for $i\in [n]$ satisfying the following conditions:
\begin{description}
\item[(i)] $P(i)\succeq 0$ \  for all $i\in [n]$,
\item[(ii)] $P(i)_{ii}=M_{ii}$  \  for all $i\in [n]$,
\item[(iii)] $2P(i)_{ij}+P(j)_{ii}= 2M_{ij}+M_{ii}$  \ for all $i\ne j\in [n]$,
\item[(iv)] $P(i)_{jk}+P(j)_{ik}+P(k)_{ij} \le M_{ij}+M_{ik}+M_{jk}$ \  for all distinct $i,j,k\in [n]$.
\end{description}
\end{lemma}

\begin{proof}
As observed above, $M\in \MK_n^{(1)}$ if and only if there exist matrices $P(i)\succeq 0$ for $i\in[n]$ and scalars $c_{ijk}\geq 0$ satisfying Eq.(\ref{eqpM}). We now obtain the  conditions (ii)-(iv) by comparing coefficients at both sides of (\ref{eqpM}).  We give the details since they will be useful later. First, we start with the left hand side in (\ref{eqpM}): 
\begin{equation}\label{M_i}
\Big(\sum_{i=1}^{n} x_i\Big)x^TMx=\sum_{i=1}^{n}M_{ii}x_i^3 + \sum_{i\neq j\in[n]}x_i^2x_j(M_{ii}+2M_{ij})+\sum_{1\leq i<i<j<k\leq n}x_ix_jx_k(M_{ij} + M_{jk} + M_{ik}).
\end{equation}
Now we expand the right hand side in (\ref{eqpM}):
 \begin{equation}\label{P_i}
\begin{array}{l}
 \sum_{i=1}^{n}x_ix^TP(i)x +  \sum_{1\leq i<j<k\leq n}c_{ijk}x_ix_jx_k \\
 = \sum_{i=1}^{n} x_i^3 P(i)_{ii} + \sum_{i\neq j\in [n]} x_i^2x_j(P(j)_{ii}+2P(i)_{ij}) + \sum_{1\leq i<j<k\leq n}x_ix_jx_k(P(i)_{jk}+P(j)_{ik}+P(k)_{ij} + c_{ijk}).
\end{array} \end{equation}
 Comparing coefficients at both sides we obtain the desired result.
 \end{proof}
 
 \begin{remark}\label{remK0}
 Observe that Lemma \ref{lemK1} remains valid if in (i) we replace the condition $P(i)\succeq 0$ by the weaker condition $P(i)\in \mathcal K^{(0)}_n$. 
Indeed, as  $\MK_n^{(0)}=\mathcal {S}^n_+ +\oR^{n\times n}_+$, the `only if' part is clear  since $\mathcal{S}^n_+\subseteq \MK_n^{(0)}$, 
 and the `if part' follows easily from the fact that  $(x^{\circ 2})^TNx^{\circ 2}\in \Sigma$ for any $N\in\oR^{n\times n}_+$.
 \end{remark}
 
 We say that the matrices $P(1), P(2),\dots P(n)$ are a {\em $\konec$}  for $M$ if they satisfy the conditions (i)-(iv) of Lemma \ref{lemK1}.  In other words, the matrices $P(1), \dots, P(n)$ are a $\konec$ of $M$ if they are positive semidefinite 
 and there exist scalars $c_{ijk}\geq 0$ for $1\leq i < j < k\leq n$ satisfying Eq. (\ref{eqpM}).

Now we give some 
easy, but crucial properties of $\MK^{(0)}$- and $\MK^{(1)}$-certificates, involving their kernel, that will be repeated used in the paper.

\begin{lemma}\label{kernel-k0}
Let $M\in \MK_n^{(0)}$ and let $P$ be a $\kzeroc$ of $M$. If $x\in \mathbb{R}^n_+$ and $x^TMx=0$, then $Px=0$ and $P[S]=M[S]$, where $S=\{i\in[n] \text{ : } x_i>0\}$ is the support of $x$.
\end{lemma}
\begin{proof}
Since $P$ is a $\kzeroc$  there exists a matrix $N\geq 0$ such that $M=P+N$. Hence, $0=x^TMx=x^TPx +x^TNx$. Then  $x^TPx=0=x^TNx$ as $P\succeq 0$ and $N\geq 0$. This implies $Px=0$ since $P\succeq 0$. On the other hand, since $x^TNx=0$ and $N\geq 0$, we get $N_{ij}=0$ for $i,j\in S$. Hence, $M[S]=P[S]$, as $M=P+N$. 
\end{proof}

\begin{lemma}\label{kernel-k1}
Let $M\in\MK_n^{(1)}$ and let $P(1), \dots, P(n)$ be a $\konec$ of $M$. Let $x\in \mathbb{R}^{n}_+$ such that $x^TMx=0$. Then the following holds:
\begin{description}
\item[(i)] If $x_i>0$ then $P(i)x=0$.
\item[(ii)] If $x_i,x_j,x_k>0$ then $M_{ij}+M_{jk}+M_{ik}=P(i)_{ij}+P(j)_{ik} +P(k)_{ij}$.
\end{description}
\end{lemma}

\begin{proof}
By 
evaluating Eq. (\ref{eqpM}) at $x$, we get that the left hand side is zero while all terms in the right hand side are nonnegative, so all of them  vanish. Hence, if $x_i>0$ then $x^TP(i)x=0$, which implies $P(i)x=0$ as $P(i)\succeq 0$. On the other hand, if $x_ix_jx_k>0$ then $c_{ijk}=0$, which implies the desired identity (see Eq. (\ref{M_i}) and Eq. (\ref{P_i})).
\end{proof}

\begin{example}\label{example-5-cycle}
Consider the $5$-cycle $C_5$ shown in Fig. \ref{FigC51}
and the associated graph matrix $M_{C_5}=2(A_{C_5}+I)-J$,  also known as the Horn matrix and denoted by $H$.
\\
\begin{minipage}[t]{0.5\linewidth}
 \begin{figure}[H]
  \centering
\definecolor{uuuuuu}{rgb}{0.26666666666666666,0.26666666666666666,0.26666666666666666}
\definecolor{uququq}{rgb}{0.25098039215686274,0.25098039215686274,0.25098039215686274}
\begin{tikzpicture}[line cap=round,line join=round,>=triangle 45,x=.3cm,y=.3cm]
\draw [line width=1pt] (-13.578509963461387,8.326106157908189)-- (-11.19999999999867,10.054194799648329);
\draw [line width=1pt] (-11.19999999999867,10.054194799648329)-- (-8.821490036535954,8.326106157908187);
\draw [line width=1pt] (-8.821490036535954,8.326106157908187)-- (-9.72999999999844,5.53);
\draw [line width=1pt] (-12.6699999999989,5.53)-- (-9.72999999999844,5.53);
\draw [line width=1pt] (-12.6699999999989,5.53)-- (-13.578509963461387,8.326106157908189);
\begin{scriptsize}
\draw [fill=uququq] (-12.6699999999989,5.53) circle (2.5pt);
\draw[color=uququq] (-13.66, 5.52) node {$4$};
\draw [fill=uququq] (-9.72999999999844,5.53) circle (2.5pt);
\draw[color=uququq] (-8.92, 5.52) node {$3$};
\draw [fill=uuuuuu] (-8.821490036535954,8.326106157908187) circle (2.5pt);
\draw[color=uququq] (-8.0, 8.52) node {$2$};
\draw [fill=uuuuuu] (-11.19999999999867,10.054194799648329) circle (2.5pt);
\draw[color=uququq] (-12, 10.5) node {$1$};
\draw [fill=uuuuuu] (-13.578509963461387,8.326106157908189) circle (2.5pt);
\draw[color=uququq] (-14.3, 8.9) node {$5$};
\end{scriptsize}
\end{tikzpicture}
 \caption{Graph $C_5$}
 \label{FigC51}
  \end{figure}
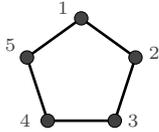
\end{minipage}
\begin{minipage}[t]{0.4\linewidth}

$$H= M_{C_5}=
\begin{pmatrix}
1 & 1 & -1 & -1 & 1 \\
1 & 1 & 1 & -1 & -1  \\
-1 & 1 & 1 & 1 & -1  \\
-1 & -1 & 1 & 1 & 1  \\
1 & -1 & -1 & 1 & 1  \\
\end{pmatrix}$$
 \hspace{2cm}The Horn matrix
\end{minipage}

 The Horn matrix $H$ is known to belong to $\MK_n^{(1)}$ \cite{Parrilo-thesis-2000}.  As we now show, it admits a unique $\konec$, where the matrices  $P(1),\ldots,P(5)$ are of the form shown below:
\begin{equation}\label{eqPiC5}
P(1)=
\begin{pmatrix}
1 & 1 & -1 & -1 & 1 \\
1 & 1 & -1 & -1 & 1  \\
-1 & -1 & 1 & 1 & -1  \\
-1 & -1 & 1 & 1 & -1  \\
1 & 1 & -1 & -1 & 1  \\
\end{pmatrix}, \quad P(i)=
\begin{pmatrix}
\bovermat{$i^{\perp}$}{1 & 1 & 1 &} \bovermat{$V\setminus i^{\perp}$}{-1 & -1} \\
1 & 1 & 1 & -1 & -1  \\
1 & 1 & 1 & -1 & -1  \\
-1 & -1 & -1 & 1 & 1  \\
-1 & -1 & -1 & 1 & 1  \\
\end{pmatrix} \ \ \text{ for } i\in [5].
\end{equation}
Up to symmetry it suffices to show that $P(1)$ 
has the above  shape. Let $C_1, C_2, C_3, C_4, C_5$ denote its  columns. Since the vectors $(1,0,1,0,0), (1,0,0,1,0), (1,1,0,2,0), (1,0,2,0,1)$ are zeros of the form  $x^THx$, by Lemma \ref{kernel-k1} (i), we obtain $C_1=-C_3$, $C_1=-C_4$, $C_1+C_2+2C_4=0$ and $C_1+C_5+2C_3=0$. Hence, $C_1=C_2=C_5=-C_3=-C_4$.  Since $P(1)_{11}=1$ the above conditions determine the first row and column and therefore the rest of the matrix $P(1)$, which thus has the desired shape. 
\end{example}

As shown in the previous lemmas, the zeros of the quadratic form $x^TMx$ give us information about the kernel of  $\MK^{(0)}$- and $\MK^{(1)}$-certificates for $M$. For the case of the graph matrices $M_G=\alpha(G)(A_G+I)-J$ there is a full characterization of the zeros of this quadratic form  in  $\Delta_n$ (and thus in $\oR^n_+$). First, observe that, for $x\in \Delta_n$, we have $x^TM_Gx=0$ if and only if $x$ is an optimal solution of the following program
\begin{equation}\label{motzkin-form}
{1\over \alpha(G)}=\min\{x^T (I+A_G)x: x\in \Delta_n\}.
\end{equation}
 Indeed we have
\begin{equation}\label{eq-mini}
x^TM_Gx=0 \Longleftrightarrow \alpha(G)x^T(A_G+I)x -x^TJx=0 \Longleftrightarrow x^T(A_G+I)x =\frac{1}{\alpha(G)}.
\end{equation}
The  formulation of $\alpha(G)$ in (\ref{motzkin-form}) is due to Motzkin and Straus \cite{motzkin} and underlies its copositive formulation in (\ref{copoalpha}).

We conclude with recalling the characterization of the minimizers of problem (\ref{motzkin-form}), following \cite[Corollary 4.4]{LV2021} (see also \cite{GHPR}).

\begin{theorem}\label{minimizers-LV}
Let $x\in \Delta_n$ 
with support $S=\{i\in [n]: x_i>0\}$, and let $V_1, V_2, \dots, V_k$ denote the connected components of the graph $G[S]$. Then $x$ is an optimal solution of (M-S) if and only if $k=\alpha(G)$, $V_i$ is a clique and $\sum_{j\in V_i}x_j={1\over \alpha(G)}$ for all $i\in [k]$. In that case all edges in $G[S]$ are critical edges of $G$.
\end{theorem}



\ignore{We  finish with the following result, which shows that duplicating rows/columns preserves the cones $\MK_n^{(r)}$.
\begin{lemma}[\cite{GL2007}]\label{twin-col}
Consider the matrices
$$M = \begin{pmatrix}
A & b   \\
b^T & c 
\end{pmatrix}\in \MS^n \quad \text{ and } \quad 
M'=\begin{pmatrix}
A & b & b  \\
b^T & c & c\\
b^T & c & c
\end{pmatrix}\in \MS^{n+1},$$
where $A\in \mathcal S^{n-1}$, $b\in \oR^{n-1}$ and $c\in \oR$. Then, $M\in \MK_n^{(r)}$ if and only if $M'\in \MK_{n+1}^{(r)}$.
\end{lemma}
}

\section{On the exactness of the approximation of $\COP_n$ by the Parrilo cones $\MK^{(r)}_n$}\label{parrilo-cones}

\ignore{Many combinatorial and quadratic programs can be formulated as a copositive program. That is, a conic linear programming over the cone of copositive matrices $\COP_n=\{M\in \mathcal{S}^n \text{ : } x^TMx\geq 0 \text{ for all } x\geq 0 \}$. Optimizing over $\COP_n$ is difficult in general since many hard combinatorial problems such as the Stable Set Problem  can be formulated as an instance of it. Moreover, determining whether a matrix is copositive is hard \cite{ }.  In \cite{Parrilo-thesis-2000}, Parrilo introduced a hierarchy of approximations for  $\COP_n$ based on sum of squares of polynomials. 
$$\MK^{(r)}_n=\Big\{M\in \MS^n: \Big(\sum_{i=1}^nx_i^2\Big)^r P_M(x) \in\Sigma_{}\Big\},$$
where $P_M(x)=(x^{\circ 2})^TMx^{\circ 2}$ and $\Sigma$ is the cone of sum of squares of polynomials. It is easy to observe that $\MK_n^{(0)}\subseteq \MK_n^{(1)}\subseteq \dots \subseteq \COP_n$. Moreover, By Reznick \cite{ }, $cl\Big(\bigcup\limits_{r\in \mathbb{N}}\MK_n^{(r)}\Big)=\COP_n$. 
}

In this section we  investigate 
 the cones $\MK_n^{(r)}$, which were introduced by Parrilo \cite{Parrilo-thesis-2000} as inner approximations of  the coositive cone $\COP_n$ and satisfy
$$\text{int}(\COP_n)\subseteq \bigcup_{r\ge 0} \MK^{(r)}_n\subseteq \COP_n.$$
 As pointed out  in \cite{ dKP2002,GL2007}, one difficulty for the understanding of the cones $\MK_n^{(r)}$ is that they are not closed under adding a zero row/column when $r\geq 1$. In addition, while $\COP_4=\MK^{(0)}_4$,  it is shown in \cite{DDGH} that for any $n\ge 5$ the copositive cone $\COP_n$ is not contained in  a single cone  $\MK^{(r)}_n$ for any $r\in \mathbb{N}$.  Here we prove that the situation is even worse: for $n\ge 6$, the cone $\COP_n$ is not even contained in the union of the cones $\MK^{(r)}_n$. For this, we show that if a copositive matrix does not belong to the cone $ \MK_n^{(0)}$ then after adding to it a zero row/column the resulting matrix  does not belong to any of the cones $\MK_{n+1}^{(r)}$ ($r\ge 0$). The question of whether the union of the cones $\MK^{(r)}_5$ covers the full copositive cone $\COP_5$ remains open. 
Motivated by this question one may ask whether any diagonal scaling of the Horn matrix $H=M_{C_5}$ lies in some cone $\MK^{(r)}_5$. We will characterize the diagonal scalings of  $H$ that belong to the cone $\MK^{(1)}_5$, which crucially relies on the fact  that $H$  admits a unique $\MK^{(1)}$-certificate.

\subsection{Constructing copositive matrices not belonging to any Parrilo cone }

Dickinson et al.  \cite{DDGH}  conjectured that for any integer $n\geq 1$ there exists an integer $r\ge 0$ such that any copositive matrix of size $n$ with $0,1$-valued diagonal entries lies in the cone $\MK^{(r)}_n$. The conjecture holds  for $n\leq 4$ with $r=0$ since 
$\COP_4=\MK_4^{(0)}$.
For $n=5$ it is shown in \cite{DDGH} that the conjecture  holds with $r=1$. 
Here we will show that this conjecture does not hold for $n\ge 6$. Even more we give an example of copositive matrix with an all-ones diagonal that does not belong to any of the cones $\MK^{(r)}_n$.
For this, we consider the following construction. Given two copositive matrices $M_1\in \COP_n$ and $M_2\in \COP_m$, we consider their direct sum
\begin{equation}
M_1\oplus M_2 := \left(\begin{array}{c|c}
M_1 & 0\\
\hline
0 & M_2
\end{array}\right),
\end{equation}
which is clearly copositive.
We will show below that, under some conditions on $M_1,M_2$, the matrix $M_1\oplus M_2$ does not belong to any of the cones $\MK_{n+m}^{(r)}$. We start with a preliminary result on sums of squares of polynomials. 

\begin{lemma}\label{least-degree} Let $f$ be a polynomial of degree $2d$ in $n$ variables. Write  $f=f_r+f_{r+1}+\dots f_{2d}$ where $f_r\ne 0$ and, for $r\le j\le 2d$, each $f_j$ is a homogeneous polynomial with degree $j$.  If $f$ is a sum of squares then $f_r$ is a sum of squares.
\end{lemma}

\begin{proof}
Since $f$ is a sum of squares we have $f=\sum_{i=1}^{m}q_i^2$ for some $q_i\in \mathbb{R}[x]$ wtih $\deg(q_i)\leq d$ for all $i\in [m]$. Then each $q_i$ has the form 
$q_i=\sum_{j=0}^{d}a_i^{(j)}$, where each nonzero $a_i^{(j)}$ is a homogeneous polynomial of degree $j$. For $i\in [m]$ set $L_i=\min\{ j: a_i^{(j)} \neq 0\} $ and set $L=\min\{L_i:  i\in [m]\}$. Notice that there is no monomial with degree less that $2L$  in $\sum_iq_i^2=f$ and 
$f_{2L} = \sum_{i=1}^{m}(a_i^{(L)})^2 \neq 0 $. Hence it follows that $f_r=f_{2L}$ is a sum of squares. 
\end{proof}

\begin{theorem}\label{theo-cop}
Let $M_1\in \COP_n$ and $M_2\in \COP_m$ be two copositive matrices. Assume that $M_1\notin \MK_n^{(0)}$ and that there exists $0\neq z\in \mathbb{R}^m_+$ such that $z^TM_2z=0$. Then  we have
\begin{equation}
\left(
\begin{array}{c|c}
M_1 & 0\\
\hline
0 & M_2
\end{array}
\right) \in \COP_{n+m} \setminus \bigcup\limits_{r\in \mathbb{N}}\MK_{n+m}^{(r)}.
\end{equation}
\end{theorem} 
\begin{proof}
Assume for contradiction $M_1\oplus M_2\in\MK^{(r)}_n$, i.e., the polynomial $(p_{M_1}(x) + p_{M_2}(y))(\sum_{i=1}^{n}x_i^2 + \sum_{j=1}^{m}y_j^2)^r$ is a sum of squares. Here, for convenience, we denote the $n+m$ variables as $x_i\ (i\in [n])$ and $y_j\ (j\in [m])$ and we set $p_{M_1}(x)= (x^{\circ 2})^T M_1x^{\circ 2}$ and $p_{M_2}(y)=(y^{\circ 2})^T M_2y^{\circ 2}$.
Write $z=y^{\circ2}$ for some $y\in \oR^m$, so that $p_{M_2}(y)=0$, and  $c:=\sum_{j=1}^m y_j^2>0$.
 Then  the polynomial $f(x):=p_{M_1}(x)(\sum_{i=1}^{n}x_i^2 + c)^r$ is a sum of squares.
 By decomposing $f$ as a sum of  homogeneous polynomials we see that its least degree homogeneous part is the polynomial $cp_{M_1}(x)$, with  degree 4. By Lemma \ref{least-degree} we obtain that $cp_{M_1}(x)$ is a sum of squares, i.e, $M_1\in \MK_n^{(0)}$, yielding a contradiction.
\end{proof}

 \ignore{
\begin{theorem}
Let $M$ be a matrix and let $M'$ be the matrtix obtained by adding a zero row-column. If $M' \in \MK_{n+1}^{(r)}$ then $M\in \MK_n^{(0)}$. In particular, if $M_1\notin \MK_n^{(0)}$ then $M' \notin \MK_{n+1}^{(r)}$ for all $r\in \mathbb{N}$
\end{theorem}
\begin{proof}
By construction $P_M(x)=P_{M'}(x)$. Assume that $M' \in \MK_{n+1}^{(r)}$, then the poynomial $p(x_1, x_2, \dots, x_n, x_{n+1})=P_M(x)(\sum_{i=1}^{n+1}x_i^2)^r$ is a sum of squares, then $q:=p(x_1, \dots, x_n, 1)= P_M(x)(\sum_{i=1}^{n}x_i^2 + 1)^r$ is also a sum of squares. By decomposing $q$ in homogeneous polynomials we obtain that the least degree homogeneous part is of degree 4 and is precisely $P_M(x)$. By Lemma \ref{least-degree} we obtain that $P_M(x)$ is a sum of squares, i.e, $M\in \MK_n^{(0)}$.
\end{proof}
}

We now use Theorem \ref{theo-cop} to give some classes of copositive matrices that do not belong to $\MK^{(r)}$ for any $r\in \mathbb{N}$.  As a first application we obtain 
\begin{equation}\label{zero-col}
M\in \COP_n\setminus \MK_n^{(0)} \ \Longrightarrow \ \left(\begin{array}{c|c}
M & 0\\
\hline
0 & 0
\end{array}
\right) \in \COP_{n+1}\setminus \bigcup\limits_{r\in \mathbb{N}} \MK_{n+1}^{(r)}.
\end{equation}
Since the inclusion $\MK^{(0)}_5\subset \COP_5$ is strict this shows that also the inclusion $ \bigcup_{r\in\mathbb{N}}\MK_n^{(r)}\subset \COP_n$ is strict for any $n\ge 6$. Hence the cone  $\bigcup_{r\in\mathbb{N}}\MK_n^{(r)}$ is not a closed set for $n\ge 6$. On the other hand,  we have $\COP_n=\MK_n^{(0)}$ for $n\leq 4$ \cite{Diananda}. The situation for the case of $5\times 5$ matrices remains open.

\begin{question}
Does equality $\COP_5 = \bigcup\limits_{r\ge 0}\MK_5^{(r)}$ hold?
\end{question}

Dickinson et al. \cite{DDGH} proved that any $5\times 5$ copostive matrix with 0,1-valued  diagonal entries belongs to $\MK_5^{(1)}$. They  conjectured that for any integer $n\ge 6$ there exists an integer $r\ge 0$ such that any $n\times n$ copositive matrix with $0,1$-valued  diagonal entries  belongs to $\MK_n^{(r)}$ (see \cite[Conjecture 1]{DDGH}).  Using Theorem~\ref{theo-cop} we can disprove this conjecture.

\begin{example}
Let  $M_1:=M_{C_5}=H$ be the Horn matrix, known   to be copositive with $H\notin \MK_n^{(0)}$. For the matrix $M_2$ we first consider the $1\times 1$ matrix $M_2=0$ and as a second example we consider $ M_2=\begin{pmatrix} 1 & -1 \\ -1 & 1 \end{pmatrix}\in \COP_2$. Then,  as an application of Theorem \ref{theo-cop},  we obtain 
\begin{equation}\label{eqexmat}
\left(
\begin{array}{c|c}
H & 0\\
\hline
0 & 0
\end{array}
\right) \in \COP_6 \setminus \bigcup\limits_{r\in\mathbb{N}}\MK_{6}^{(r)}, \quad  \left(
\begin{array}{c|c}
H & 0\\
\hline
0 & {\begin{array}{cc}1 & -1 \\ -1 & 1
\end{array}}
\end{array}
\right)\in \COP_7\setminus \bigcup\limits_{r\in \mathbb{N}}\MK_7^{(r)}.
 \end{equation}
\ignore{$$
M'=
\begin{pmatrix}
1 & -1 & 1 & 1 & -1 & 0\\
-1 & 1 & -1 & 1 & 1 & 0 \\
1 & -1 & 1 & -1 & 1 & 0 \\
1 & 1 & -1 & 1 & -1 & 0 \\
-1 & 1 & 1 & -1 & 1 & 0 \\
0 & 0 & 0& 0 & 0 & 0
\end{pmatrix} \quad M'' = \begin{pmatrix}
1 & -1 & 1 & 1 & -1 & 0 & 0\\
-1 & 1 & -1 & 1 & 1 & 0  & 0\\
1 & -1 & 1 & -1 & 1 & 0 & 0\\
1 & 1 & -1 & 1 & -1 & 0 & 0\\
-1 & 1 & 1 & -1 & 1 & 0 & 0\\
0 & 0 & 0& 0 & 0 & 1 & -1 \\
0 & 0 & 0 & 0 & 0 & -1 & 1 
\end{pmatrix} $$}

The left most matrix in (\ref{eqexmat}) is  copositive, has all its diagonal entries equal to $0,1$ and does not belong to any of the cones $\MK^{(r)}_6$; selecting for $M_2$ the zero matrix of  size $m\ge 1$ gives  a matrix in $\COP_n\setminus \bigcup_{r\ge 0}\MK^{(r)}_n$ for any size $n\ge 6$.
The right most matrix in (\ref{eqexmat}) is  copositive,  has all its diagonal entries equal to 1 and  does not lie in any of the cones $\MK_7^{(r)}$.
More generally, if we  select the matrix $M_2= {1\over m-1}(mI_m-J_m)$, which is positive semidefinite with $e^TM_2e=0$, then we obtain a matrix in $\COP_n\setminus \bigcup_{r\ge 0} \MK^{(r)}_n$ with diagonal entries equal to 1, for any size $n\ge 7$.
In contrast, as mentioned above, Dickinson et al. \cite{DDGH} proved that any copositive $5\times 5$ matrix with an all-ones diagonal  belongs to $\MK_5^{(1)}$. 
The situation for the case of $6\times 6$ copositive matrices remains open.

\end{example}
\begin{question}
Is it true that any $6\times 6$ copositive matrix with an all-ones diagonal belongs to $\MK^{(r)}_6$ for some $r\in \oN$?
\end{question}
 
 We conclude with an observation on the number of zeros in the simplex $\Delta_n$ of the quadratic form $x^TMx$ when $M$ is a copositive matrix.
For the class of copositive matrices arising from the graph matrices  $M_G =\alpha(G)(A_{G}+I)-J$ it is proved in   \cite{LV2021} that the number of such zeros is finite if and only if the graph $G$ is acritical, in which case the matrix $M_G$ belongs to some cone $\MK^{(r)}$. 
We now show that the property of having finitely many zeros  in the simplex for the quadratic form $x^TMx$  is in general not sufficient to ensure membership of $M$ in some cone $\MK^{(r)}$.  
 Specifically, we give a class  of copositive matrices $M\not\in \cup_{r} \MK^{(r)}$ for which the quadratic form $x^TMx$ has a unique zero in $\Delta_n$.

\ignore{
\begin{example}
Let $M$ be an $n\times n$ strictly copositive matrix such that $M\notin \MK_n^{(0)}$. Consider the matrix $M_{G_9} = 4(A_{G_9}+I)-J$, where $G_9$ is the acritical graph shown in Figure \ref{graph_G_9}. Then we have 
\begin{equation}
A:=\left(
\begin{array}{c|c}
M & 0\\
\hline
0 & M_{G_9}
\end{array}
\right) \in \COP\setminus \bigcup_{r\ge 0}\MK_{n+9}^{(r)}.
\end{equation}
Now we prove that the quadratic form $x^TAx$ has finitely many zeros over the simplex. Let $x\in \Delta_{n+9}$ such that $x^TAx=0$. Since $M$ is strictly copositive then $x_i=0$ for $i=1, \dots , n$ and $(x_{n+1}, x_{n+2}, \dots, x_{n+9})$ correspond to  a zero of the quadratic form $x^TM_{G_9}x$ over the simplex. Since $G_9$ is acritical then the quadratic form $x^TM_{G_9}x$ has finitely many zeros over the simplex (corresponding to the $\alpha$-stable sets of $G_9$, see Theorem \ref{minimizers-LV}). This completes the proof.
\end{example}
}

\begin{example}
Let  $M_1\in \COP_n$  be a strictly copositive matrix such that $M_1\not\in \MK_n^{(0)}$. For instance, one can  take $M_1=t(I+A_G)-J$, where $G$ is a graph with $\rrank(G)\ge 1$ and $\alpha(G)<t< \vartheta^{(0)}(G)$.
By Theorem \ref{theo-cop} we have 
\begin{equation}
M:=\left(
\begin{array}{c|c}
M_1 & 0\\
\hline
0 & \begin{array}{cc}
1 & -1 \\
-1 & 1
\end{array}
\end{array}
\right) \in \COP_{n+2}\setminus \bigcup_{r\ge 0}\MK_{n+2}^{(r)}.
\end{equation}
Now we prove that the quadratic form $x^TMx$ has a unique zero in the simplex. For this let $x\in \Delta_{n+2}$ such that $x^TMx=0$. Since $M_1$ is strictly copositive and $y:=(x_1,\ldots,x_n)$ is a zero of the quadratic form $y^TM_1y$ it follows that $x_1=\ldots=x_n=0$. Hence  $(x_{n+1}, x_{n+2})$ is  a zero of the quadratic form $x_{n+1}^2 - 2x_{n+1}x_{n+2}+x_{n+2}^2$ in the simplex $\Delta_2$ and thus $x_{n+1}=x_{n+2}=1/2$. This shows that  the only zero of the quadratic form $x^TMx$ in the simplex is $(0,0,\dots, 0, \frac{1}{2},\frac{1}{2})$, as desired.
\end{example}
 
\subsection{Characterizing the diagonal scalings of the Horn matrix in $\MK^{(1)}$}\label{secscalingC5}
As mentioned above, it is not known whether the union of the cones $\MK^{(r)}_5$ covers the full cone $\COP_5$, but  any  matrix  in $\COP_5$ with $0,1$-valued diagonal entries  lies in the cone $\MK^{(1)}_5$ \cite{DDGH}. One of the key ingredients for this result is the complete characterization of the extreme rays of the cone $\COP_5$ by Hildebrand \cite{Hildebrand}. In particular the Horn matrix $H$ and its positive  diagonal scalings define a class of extreme rays of $\COP_5$, so the question arises whether all of them lie in some cone $\MK^{(r)}_5$.  Here, a positive diagonal scaling of a matrix $M$ is a matrix of the form $DMD$, where $D=\text{\rm diag}(d_1,\ldots,d_5)$ with $d_1,\ldots,d_5>0$.

 \begin{question} Is it true that  every positive diagonal scaling  of the Horn matrix belongs to  $ \MK_5^{(r)}$ for some $r$?
  \end{question} 
 
 As a first partial step we   characterize the diagonal scalings of the Horn matrix that lie in $\MK^{(1)}_5$.
A key ingredient for this is the fact that the Horn matrix admits a unique $\MK^{(1)}$-certificate, as was observed in Example \ref{example-5-cycle}.

\begin{theorem}
Let $D=\text{\rm diag}(d_1,d_2,d_3,d_4,d_5)$ with $d_1,\ldots,d_5> 0$ and let $H$ be the Horn matrix. Then, $DHD$ belongs to $ \MK_5^{(1)}$ if and only if $d_1,\ldots,d_5$ satisfy the following inequalities
\begin{align}\label{eqd}
d_{i-1}d_i+d_id_{i+1}\ge d_{i-1}d_{i+1}\ \text{ for } i\in [5]\ (\text{indices taken modulo 5}).
\end{align}
\end{theorem}

\begin{proof}
Set $M:=DHD$. First we show the `if part'. Assume $d_1,\ldots,d_5$ satisfy  conditions (\ref{eqd}); we show $M\in \MK^{(1)}_5$. For this consider the matrices $Q(i):=DP(i)D$, where the matrices $P(i)$ are the $\MK^{(1)}$-certificate for $H$ from (\ref{eqPiC5}); we show that the matrices $Q(i)$ form a $\MK^{(1)}$-certificate for $M$, i.e., satisfy the conditions (i)-(iv) from Lemma \ref{lemK1}. Clearly $Q(i)\succeq 0$ and $Q(i)_{ii}=d_i^2$ for all $i\in [5]$, so (i), (ii) hold.
Also, $2Q(i)_{ij}+Q(j)_{ii}=2d_id_j P(i)_{ij}+d_i^2 P(j)_{ii}= 2M_{ij}+M_{ii}$ since $P(i)_{ij}=H_{ij}$ and $P(j)_{ii}=H_{ii}$, so (iii) holds.
We now check (iv), i.e., $Q(i)_{jk}+Q(j)_{ik}+Q(k)_{ij}\le M_{ij}+M_{jk}+M_{ik}$ for any distinct $i,j,k\in [5]$. There are two possible patterns (up to symmetry): $(i,j,k)= (1,2,4)$ and  $(i,j,k)= (5,1,2)$. For the first pattern we get
$$Q(1)_{24}+Q(2)_{14}+Q(4)_{12}= d_2d_4 P(1)_{24}+ d_1d_4P(2)_{14} + d_1d_2 P(4)_{12}= M_{24}+M_{14}+M_{12}.$$ For the second pattern we get
$$\begin{array}{ll}
M_{12}+M_{25}+M_{15}-(Q(5)_{12}+Q(1)_{25}+Q(2)_{15}) \\
 = d_1d_2-d_2d_5+d_1d_5 -( d_1d_2 P(5)_{12}+d_2d_5 P(1)_{25}+d_1d_5 P(2)_{15}) \\
  = d_1d_2-d_2d_5+d_1d_5 - ( -d_1d_2+d_2d_5 -d_1d_5)\\
=  2(d_1d_2-d_2d_5+d_1d_5),
\end{array}
$$
which is nonnegative if and only if (\ref{eqd}) holds. Hence the conditions (\ref{eqd}) indeed imply that the condition (iii) of Lemma \ref{lemK1} holds for the matrices $Q(i)$ and thus they form a $\MK^{(1)}$-certificate for $M$, as desired.

Conversely, assume $M=DHD\in \MK^{(1)}_5$ and let $Q(i)$ ($i\in [5]$) be a $\MK^{(1)}$-certificate for $M$; we show $Q(i)=DP(i)D$ for $i\in [5]$, where the matrices $P(i)$ are the unique $\MK^{(1)}$-certificate for $H$ from (\ref{eqPiC5}). In view of the above this implies  that the $d_i$'s satisfy the conditions (\ref{eqd}), as desired. 
Up to symmetry it suffices  to show $Q(1)=DP(1)D$. For this  note that if 
$z^THz=0$ for $z\in \oR^n_+$, then $y^TMy=0$ for $y:=D^{-1}z\in \oR^n_+$ 
 and thus,   by Lemma \ref{kernel-k1},  $Q(i)y=0$ whenever $y_i>0$. Consider the vectors $z_1=(1,0,1,0,0)$, $z_2=(1,0,0,1,0)$, $z_3=(1,1,0,2,0)$, $z_4=(1,0,2,0,1)$, which are zeros of $x^THx$, and the corresponding vectors $y_i=D^{-1}z_i$ for $i=1,2,3,4$, which are zeros of $x^TMx$. Let $C_1, \dots, C_5$ denote the columns of $Q(1)$. Then, using the zeros $y_1,\ldots,y_5$ of $x^TMx$   we obtain the relations
$$\frac{C_1}{d_1}+\frac{C_3}{d_3}=0, \quad
\frac{C_1}{d_1}+\frac{C_4}{d_4}= 0, \quad
\frac{C_1}{d_1}+\frac{C_2}{d_2} + 2\frac{C_4}{d_4} =0, \quad
\frac{C_1}{d_1} +2\frac{C_3}{d_3} +\frac{C_5}{d_5} =0,$$
which imply  $\frac{C_1}{d_1}=\frac{C_2}{d_2}=\frac{C_5}{d_5}=-\frac{C_3}{d_3}=-\frac{C_4}{d_4}$. As $Q(1)_{11}=d_1^2$ one  easily deduces   $Q(1)=DP(1)D$, as desired.
\end{proof}


\section{Behavior of the $\rrank$ under simple graph operations}\label{pre-iso-critical}

Recall that the $\rrank$ of $G$ is the minimum integer $r$ such that $\vartheta^{(r)}(G)=\alpha(G)$. In this section, we present some useful ideas  for bounding the $\rrank$ based on simple graph operations.
 Namely, we investigate the role of isolated nodes and of critical edges, and their impact on Conjectures \ref{conj1} and \ref{conj2}. In particular, we will show that it suffices to show Conjectures \ref{conj1} and \ref{conj2}  for the class of critical graphs and that Conjecture \ref{conj2} holds if the $\rrank$ remains finite under the operation of adding isolated nodes.
 \\
 
 We start with a lemma  relating  the $\rrank$ of a graph and that of its  induced subgraphs with the same stability number, which we will use later on.
 
\begin{lemma}\label{indu-same-alpha}
Let $G=(V,E)$ be a graph and let $H$ be an induced subgraph of $G$ such that $\alpha(G)=\alpha(H)$. Then, $\rrank(H)\leq \rrank(G)$. 
\end{lemma}
\begin{proof}
As $\alpha(G)=\alpha(H)=:\alpha$ we have $M_G=\alpha(A_G+I)-J$ and $M_H=\alpha(A_H+I)-J$. As $H$ is an induced subgraph of $G$,  $M_H$ is a principal submatrix of $M_G$ and thus  $M_G\in \MK^{(r)}$ implies $M_H\in \MK^{(r)}$.
\end{proof}

\begin{remark}
Let $G$ be the graph obtained by adding a pendant edge to $C_5$ (see the left most graph in Fig. \ref{fig-GammaGGc}), so that $\alpha(G)=3=\alpha(C_5)+1$. Then, $G$ has $\rrank$ 0 as it can be covered by $\alpha(G)=3$ cliques. However, $C_5$ is an induced subgraph of $G$ and has $\rrank$ 1 (see Example \ref{example-5-cycle}). 
This shows that the condition of having the same stability number in Lemma \ref{indu-same-alpha} cannot be dropped.  
\end{remark}

\subsection{Role of isolated nodes}
\ignore{
 Two nodes $u$ and $v$ in a graph $G$ are called \textit{twin} if they   have the same closed neighborhood:  $u^\perp=v^\perp$ (which implies they are adjacent). 
The next result  follows from the fact that deleting one copy out of two identical rows/columns preserves the cone $\mathcal K^{(r)}$. 

\begin{lemma}[\cite{GL2007}] \label{twin_n}\footnote{This lemma goes only in the arXiv version}
Assume $u,v\in V$ are  two twin nodes in a graph $G$ and let $G\setminus v$ be the graph obtained by deleting node $v$ in $G$.
Then $M_G\in \MK_n^{(r)}$ if and only if $M_{G\setminus v}\in\MK_{n-1}^{(r)}$. That is, $\rrank(G)=\rrank(G\setminus v)$.
\end{lemma}
}

\ignore{Given a graph $G=(V=[n],E)$ and a node $i\in V$, let $i^\perp=\{i\}\cup N_G(i)$ denote the closed neighborhood of $i$, and let
\begin{equation}\label{eqPi}
G_i= G\setminus i^\perp \oplus K_{i^\perp}
\end{equation}
be the graph obtained from $G$ by taking the disjoint sum of $G\setminus i^\perp$ and the complete graph on $i^\perp$.}

\ignore{Hence we have $\rrank(G\oplus K_m)=\rrank(G\oplus i_0)$ for any $m\ge 1$. From this we obtain the following result of \cite{GL2007}, which permits to bound the $\rrank$. \MoL{We give a sketch of proof, which we believe provides useful insight on the behavior of the $\rrank.$ }
}

We recall a result  from \cite{GL2007}, which is useful  for bounding the $\rrank$ of a graph in terms of the $\rrank$ of certain subgraphs with an added isolated node.
\begin{proposition}[\cite{GL2007}]\label{propGL}
For any graph  $G=(V,E)$ we have:
\begin{equation}
\rrank(G)\leq 1 + \max_{i\in V} \rrank((G\setminus i^\perp) \oplus i).
\end{equation}
\end{proposition}
\ignore{
\begin{proof}\footnote{The proof will only be given in the arXiv version}
Set $G_i:= G\setminus i^\perp \oplus K_{i^\perp}$. By Lemma \ref{twin_n}, we have $\rrank(G_i) = \rrank(G\setminus i^\perp \oplus i)$. 
Assume that $M_{G_i}\in \MK_n^{(r-1)}$ for all $i\in V$, we need to prove that $M_G\in \MK^{(r)}_n$.
Note that each  matrix $$P(i):=M_{G_i} + (\alpha(G)-\alpha(G_i)) (I+A_{G_i}) = \alpha(G)(I+A_{G_i}) -J$$ belongs to $\MK_{n}^{(r-1)}$, since it is the sum of two matrices in $\MK_{n}^{(r-1)}$ as  $\alpha(G)-\alpha(G_i)\ge 0$. 
We have 
$$\begin{array}{l}
\displaystyle \Big(\sum_i x_i^2\Big)^r (x^{\circ 2})^T M_Gx^{\circ 2} \\
= \displaystyle \Big(\sum_i x_i^2\Big)^{r-1}\Big( \sum_i x_i^2 (x^{\circ 2})^T P(i) x^{\circ 2}
+\sum_i x_i^2 (x^{\circ 2})^T (M_G-P(i)) x^{\circ 2}\Big)\\
=\displaystyle  \sum_i x_i^2 \Big(\underbrace{(\sum_i x_i^2)^{r-1} (x^{\circ 2})^T P(i) x^{\circ 2}}_{=\sigma_1} \Big)
+ \Big(\sum_i x_i^2\Big)^{r-1}\Big( \underbrace{\sum_i x_i^2 (x^{\circ 2})^T (M_G-P(i)) x^{\circ 2}}_{=\sigma_2}\Big).
\end{array}
$$
We show that this polynomial is a sum of squares, thus showing $M_G\in\MK^{(r)}_n$. Indeed,   $\sigma_1\in \Sigma$ since each $P(i)$ belongs to $\MK^{(r-1)}_n$.
In addition, one can check that the matrices $P(i)$ satisfy the conditions (ii)-(iv) of Lemma \ref{lemK1}. Then, using the identity (\ref{eqpM})  we obtain that $\sigma_2$ has nonnegative coefficients, and thus $\sigma_2\in\Sigma$, which
 concludes the proof.
\end{proof}
}
In view of Proposition \ref{propGL},  understanding how adding isolated nodes changes the $\rrank$ is crucial for Conjectures~\ref{conj1} and \ref{conj2}. On the one hand, it was shown in \cite{GL2007} that if adding an isolated node does not increase the $\rrank$ then Conjecture \ref{conj1} holds.

\begin{proposition}[\cite{GL2007}]\label{adding-iso-rk}
Assume $\rrank (G\oplus i_0)\le \rrank (G)$ for any graph $G$. Then Conjecture \ref{conj1} holds.
\end{proposition}
\ignore{
\begin{proof}
We show $\rrank(G)\le \alpha(G)-1$ by induction on the number of nodes.
Note $\alpha(G\setminus i^\perp) \le \alpha(G)-1$ for any $i\in V$. By the induction assumption we have that  $\rrank(G\setminus i^\perp)\le \alpha(G\setminus i^\perp)-1\le \alpha(G)-2$.
Using the assumption of the lemma, we have $\rrank(G\setminus i^\perp \oplus i) \le \rrank(G\setminus i^\perp) \le \alpha(G)-2$. We now use Proposition \ref{propGL} to conclude that $\rrank(G)\le \alpha(G)-1$. 
\end{proof}
}
As we now show,  if after adding an isolated node the  $\rrank$ can increase by at most  an absolute constant $a\in\oN$, 
then we can bound $\rrank(G)$ in terms of $\alpha(G)$. In particular, when $a=0$ we recover Proposition \ref{adding-iso-rk}.
\begin{proposition}
Let $a\in \mathbb{N}$ be a nonnegative number. Assume that $\rrank(G\oplus i_0)\leq \rrank(G)+a$ for all graphs $G$. Then $\rrank(G)\leq (a+1)\alpha(G)-1$ for all graphs $G$.
\end{proposition}
\begin{proof}
We proceed by induction on $\alpha(G)$. If $\alpha(G)=1$ then $\rrank(G)=0\leq a$. Assume now $\alpha(G)\ge 2$. Using  Proposition \ref{propGL} and the assumption we get $\rrank(G)\leq a+1 + \max_{i\in V}\rrank(G\setminus i^\perp)$. Since $\alpha(G\setminus i^\perp)\leq \alpha(G)-1$, we can apply  the induction assumption  to $G
\setminus i^\perp$ and  obtain $\rrank(G\setminus i^\perp)\leq (a+1)(\alpha(G)-1)-1$. This gives $\rrank(G)\leq a+1 + (a+1)(\alpha(G)-1)-1 = (a+1)\alpha(G)-1$.
\end{proof}
On the other hand, as we now show, Conjecture \ref{conj2} holds if and only if the $\rrank$ remains finite after adding isolated nodes to finite $\rrank$ graphs.
\begin{proposition}\label{adding-iso-rk-f}
Conjecture \ref{conj2} holds if and only if $\rrank(G)<\infty$ implies $\rrank(G\oplus i_0)<\infty$.
\end{proposition}

\begin{proof}
The `only if' part is clear. We show the `if' part by contradiction. So assume that  $\rrank(G)<\infty$ implies $\rrank(G\oplus i_0)<\infty$. Assume also  that  Conjecture \ref{conj2} does not hold and let $G=(V,E)$ be a counterexample with the minimum number of nodes, so  $\rrank(G)=\infty$. By Proposition \ref{propGL}, we obtain that $\rrank(G\setminus i ^\perp \oplus i)=\infty$ for some $i\in V$. If $i$ is not isolated in $G$, then $G\setminus i^\perp \oplus i$ would be a counterexample with less nodes than $G$,  contradicting the minimality of $G$. Hence $i$ is isolated in $G$, and thus we have $G = (G\setminus i)\oplus i$. Using again the minimality assumption, we know that $\rrank(G\setminus i)<\infty$, which implies   $\rrank(G)=\rrank((G\setminus i)\oplus i)<\infty$, thus yielding a contradiction.
\end{proof}

Clearly,  if $G$ has an isolated node $i_0$,  then $G\setminus {i_0} \oplus i_0=G$ and thus the above result in Proposition \ref{propGL} is of no use to derive information about the $\rrank$ of $G$ from the $\rrank$ of the graphs $G\setminus i^\perp \oplus i $. This observation (already made in \cite{GL2007}) points out to the difficulty of analysing the $\rrank$ of graphs with isolated nodes.  We will investigate this question in Section \ref{sec-isolated} below.

On the other hand, adding an isolated node to  a graph with $\rrank=0$  preserves the property of having $\rrank=0$.  To see this, consider a graph $G$ and set $\alpha(G)=\alpha$, so that $\alpha(G\oplus i_0)=\alpha+1$. 
Then, in view of (\ref{M_G-isolated0}),  the matrix $M_{G\oplus i_0}$ belongs to $\MK_{n+1}^{(0)}$ if $M_G\in \MK_n^{(0)}$. Indeed, the first matrix in the sum in (\ref{M_G-isolated0}) is positive semidefinite and the second one belongs to $\MK_{n+1}^{(0)}$ because adding a zero row/column preserves the cone $\MK^{(0)}$. 
Since adding an isolated node preserves the $\rrank$ 0 property, the next result follows as a direct application of Proposition~\ref{propGL}.

\begin{lemma}
[\cite{dKP2002}]\label{G-i-2}
If $\rrank(G\setminus i^\perp)=0$ for all $i\in V$ then $\rrank (G)\le 1$.
\end{lemma}

\begin{example} \label{odd-cycles}
As an application of Lemma \ref{G-i-2} we obtain  that $\rrank(C_{2n+1})\leq1$ and $\rrank(\overline{C_{2n+1}})\leq1$. Moreover, if  $G$ is a graph with $\alpha(G)=2$ then, for all nodes $i\in V$,  the graph     $G\setminus {i^{\perp}}$ is a  clique and thus has $\rrank$ 0. Hence, by Lemma \ref{G-i-2},  $\rrank(G)\le 1$ and thus Conjecture \ref{conj1} holds for  graphs with $\alpha(G)=2$ (as shown in \cite{dKP2002}).

Let $G=C_5 \oplus i_0$ be the graph obtained by adding one isolated node to the 5-cycle. As shown in \cite{dKP2002} 
$G$  has $\rrank$ 1   and the graph $G\setminus i_0^{\perp}$ is the 5-cycle which also has $\rrank$ 1. This shows that Lemma \ref{G-i-2} does not permit to characterize, in general,  graphs with $\rrank$ 1. For details on  the impact of adding isolated nodes to $C_5$ see Corollary \ref{cor-cycle-isolated}.
\end{example}

\subsection{Role of critical edges}

We finish this section with two results that are useful for bounding the $\rrank$ and show the role of critical edges in this context. On the one hand, deleting non-critical edges can only increase the $\rrank$.  On the other hand, we can strengthen a result from \cite{GL2007} for the class of acritical graphs.
\begin{lemma}[\cite{LV2021}]\label{del_acritical}
Let $G=(V,E)$ be a graph and let $e \in E$. If $e$ is not a critical edge, i.e., $\alpha(G)=\alpha(G\setminus{e})$, then $\rrank(G)\leq \rrank(G\setminus{e})$. Hence, it suffices to show Conjectures \ref{conj1} and \ref{conj2} for the class of critical graphs.
\end{lemma}

 \begin{remark}\label{rem-critical}
Let  $G=(V,E)$ be a graph. Then one can find a subgraph $H=(V,F)$ of $G$ (with $F\subseteq E$), which is critical and has the same stability number: $\alpha (G)=\alpha (H)$. Indeed to get such a graph $H$ it suffices to delete successively  any non-critical edge until getting a subgraph where all edges are critical.
 Then, by Lemma \ref{del_acritical},  for any such $H$ we have
 \begin{equation}\label{eqGH}
 \rrank(G)\le \rrank(H).
 \end{equation}
 \end{remark}
 \ignore{
\begin{lemma}\label{indu-same-alpha}
Let $G=(V,E)$ be a graph and let $H$ be an induced subgraph of $G$ such that $\alpha(G)=\alpha(H)=\alpha$. Then, $\rrank(H)\leq \rrank(G)$. 
\end{lemma}
\begin{proof}
Let $M_G=\alpha(A_G+I)-J$ and $M_H=\alpha(A_H+I)-J$. Since $H$ is an induced subgraph of $G$ then $M_H$ is a principal submatrix of $M_G$. Then $M_G\in \MK^{(r)}$ implies $M_H\in \MK^{(r)}$, hence $\rrank(H)\leq \rrank(G)$. 
\end{proof}
}

In the above lemma it was observed that critical edges play a role in the study of the $\rrank $, namely it would suffice to bound the $\rrank$ of critical graphs.  On the other hand, we now prove a stronger version of Conjecture \ref{conj1} for acritical graphs with $\alpha(G)\leq 8$. 
In \cite{GL2007} the authors proposed the following conjecture and proved that it  implies Conjecture~\ref{conj1}.

\begin{conjecture}[\cite{GL2007}]\label{conjecture-monique}	
	For any $r\geq 1$, we have
	\begin{equation}\label{eq-monique}
	\vartheta^{(r)}(G)\leq r + \max\limits_{S\subseteq V, S \text{\rm stable },|S|=r} \vartheta^{(0)}(G\setminus S^\perp).
	\end{equation}
\end{conjecture}

\begin{theorem}[\cite{GL2007}]\label{theoGLstrong}
Conjecture \ref{conjecture-monique} holds for $r\leq \min(6,\alpha(G)-1)$ and for $r=7=\alpha(G)-1$. In particular, 
Conjecture \ref{conj1} holds for graphs with $\alpha(G)\leq 8$.
\end{theorem}

 In the case of acritical graphs we can show a stronger  bound on the $\rrank$ for graphs with $\alpha(G)\le 8$.

\begin{proposition}
Let $G$ be an acritical graph with $\alpha(G)\leq 8$. Then $\rrank(G)\leq \alpha(G)-2$.
\end{proposition}

\begin{proof}
It suffices to show that $\vartheta^{(0)}(G\setminus S^\perp)\leq 2$ if $S$ is stable of size $\alpha(G)-2$ since then the result follows from Eq. (\ref{eq-monique}). Let $S=\{i_1, i_2, \dots, i_{\alpha(G)-2}\}$ be a stable set of size $\alpha(G)-2$ in $G$, so that $\alpha(G\setminus S^\perp)\leq 2$. If $\alpha(G\setminus S^\perp)=1$ then $\vartheta^{(0)}(G\setminus S^\perp)=1$ and we are done. 
So assume that $\alpha(G\setminus S^\perp)=2$. Then the graph $H:=(G\setminus S^\perp)\oplus S$ is an induced subgraph of $G$ with $\alpha(H)=\alpha(G)$. We claim that $H$ is acritical. This follows from the fact that any critical edge of $H$ should also be a critical edge of $G$. Indeed,  if $e$ is critical in $H$ then there exists a stable set  in $H\setminus e$ of size $\alpha(H)+1=\alpha(G)+1$, which is then also stable in $G\setminus e$ as $H$ is an induced subgraph of $G$, so that $e$ is critical in $G$. As $H$ is acritical also  the graph $G\setminus S^\perp$ is acritical. 
We claim that $G\setminus S^\perp$ is perfect. For if not then, by the strong perfect graph theorem (\cite{SPGT}),  $G\setminus S^\perp$   contains $C_{5}$ or $\overline{C_{2n+1}}$ ($n\ge 2$)  as an induced subgraph. Since these graphs have stability number equal to $\alpha(G\setminus S^\perp)=2$ they must be acritical graphs by the above argument. Thus we reach a contradiction since $C_5$ and $\overline{C_{2n+1}}$ have critical edges.
Hence  $G\setminus S^\perp$ is perfect and thus we have $\vartheta^{(0)}(G\setminus S^\perp)=\alpha(G\setminus S^\perp)=2$, which completes the proof.
\end{proof}

\section{Towards characterizing graphs with $\rrank$ 0} \label{sec-rank-0}

In this section we investigate the graphs $G$ with $\rrank$ 0, i.e,  such that $\vartheta^{(0)}(G)=\alpha(G)$ or, equivalently, $M_G\in \mathcal K^{(0)}_n$. Recall the well-known `sandwich inequality' from \cite{Lo79}:
\begin{equation}\label{eqsandwich} 
\alpha(G)\le \vartheta'(G)=\vartheta^{(0)}(G)\le \vartheta(G)\le \overline \chi(G).
\end{equation}
In view of (\ref{eqsandwich}), if $G$ can be covered by $\alpha(G)$ cliques then $G$ has $\rrank$ 0. In addition, if $\alpha(G)=\alpha$ and $V_1, V_2, \dots, V_\alpha$ are cliques partitioning $V$ then the matrix 

\begin{center}
    $P:=\begin{pmatrix}
(\alpha-1)J & -J & \cdots & -J \\
-J & (\alpha-1)J& \cdots & -J \\
\vdots  & \vdots  & \ddots & \vdots  \\
-J & -J & \cdots & (\alpha-1)J 
\end{pmatrix}, $
\end{center}
whose block-structure is  induced by the partition $V=V_1\cup \dots \cup V_\alpha$,
is a $\kzeroc$ for $M_G$.
 In this section we show that the reverse is true for critical graphs and for graphs with $\alpha(G)\le 2$. We also provide an algorithmic method that permits to reduce the characterization of $\rrank$ 0 graphs to the same property for the class of acritical graphs.
 
  Throughout we often set $\alpha:=\alpha(G)$ to simplify notation and we say that a set $S\subseteq V$ is an $\alpha$-stable set if it is a stable set of size $\alpha(G)$. 

\subsection{Characterizing critical graphs with $\rrank$ 0}

The following  result will be repeatedly used.

\begin{lemma}\label{lem-kernel}
Let $G$ be a graph  with $\alpha(G)=\alpha$ and let $S$ be an $\alpha$-stable set. Assume $M_G\in \kzero$ and let $P$ be a $\kzeroc$ for $M_G$. Then, $\chi^S \in \ker(P)$ and $P[S]=\alpha I_{\alpha}-J_{\alpha}$.
\end{lemma}
\begin{proof}
Directly from Lemma \ref{kernel-k0} since $(\chi^S)^T M_G \chi^S =0$ as $\chi^S/|S|$ is a global minimizer of (\ref{motzkin-form}) (recall  (\ref{eq-mini})).
\end{proof}

\begin{proposition}\label{lem-components_critical}
Let $G=(V,E)$ be a graph, let $E_c$ denote the set of critical edges of $G$ and let $G_c=(V,E_c)$ be the corresponding subgraph of $G$. If  $\rrank(G)=0$ then each  connected component of the graph $G_c$ is a  clique of $G$.
\end{proposition}

\begin{proof}
By assumption, $\rrank(G)=0$. Let $P$ be a $\kzeroc$ for $M_G$. Let $V_1, V_2, \dots, V_p$ be the connected components of the graph $G_c$.
We show that each component $V_i$  is a clique in $G$. For this pick two nodes $u\ne v\in V_i$ that are connected in $G_c$.  As the edge $\{u,v\}$ is  critical, there exists a set $I\subseteq V$ such that $I\cup\{u\}$ and $I\cup \{v\}$ are $\alpha$-stable in $G$. 
Then, by Lemma \ref{lem-kernel}, the characteristic vectors $\chi^{I\cup\{u\}}$ and $\chi^{I\cup\{v\}}$ both belong to the kernel of $P$ and thus
$\chi^{\{u\}}-\chi^{\{v\}}\in \ker P$. From this we deduce that the columns of $P$ indexed by the nodes in  $V_i$ are all equal. 
Combining this with  the fact that  the diagonal entries of $P$ are equal to $\alpha -1$ and that $P$ is symmetric we can conclude that, with respect to the partition $V=V_1\cup\ldots \cup V_p$, the matrix $P$ has the following block-form:
        \begin{equation}\label{eqP}
    P=
\begin{pmatrix}
(\alpha -1)J_{|V_1|} & a_{12}J_{|V_1|\times |V_2|} & \cdots & a_{1p}J_{|V_1|\times|V_p| } \\
a_{21}J_{|V_2|\times |V_1|} & (\alpha -1)J_{|V_2|}& \cdots & a_{2p}J_{|V_2|\times |V_p|} \\
\vdots  & \vdots  & \ddots & \vdots  \\
a_{p1}J_{|V_p|\times |V_1|} & a_{p2}J_{|V_p|\times |V_2|} & \cdots & (\alpha-1)J_{|V_p|} 
\end{pmatrix}
\end{equation}
 for some scalars $a_{ij}$ ($1\le i<j\le p$).
We can now show that each $V_i$ is a clique of $G$. For this pick two distinct nodes $u,v\in V_i$. Then we have $P_{uv}=\alpha -1\le (M_G)_{uv}$, which implies that $(M_G)_{uv}=\alpha-1$ and thus $\{u,v\}$ is an edge of $G$. Here we use the fact that the off-diagonal entries of $M_G$ are equal to $\alpha-1$ for positions corresponding to edges and to $-1$ for non-edges. Hence we have shown that  each component $V_i$ is a clique of $G$, which concludes the proof.
\end{proof}

\begin{coro}\label{cor-critical_rank_0}
Assume  $G=(V,E)$ is a critical graph, i.e., all its edges are critical.   Then we have 
 $\rrank(G)=0$ if and only if $G$ is  the disjoint union of $\alpha(G)$ cliques. In particular, $\rrank(G)=0$ if and only if $\overline \chi(G)=\alpha (G)$.
\end{coro}

\begin{proof} The `only if' part follows from  Proposition \ref{lem-components_critical} and the `if part' follows from Eq. (\ref{eqsandwich}). The last claim follows directly.
\end{proof}

\begin{example}\label{rk-odd-compl}
 Let $n\geq 2$. We  saw in Remark \ref{rem-critical} that $\rrank(C_{2n+1})\leq 1$ and $\rrank(\overline{C_{2n+1}})\leq 1$. Here we can show, as an application of Corollary \ref{cor-critical_rank_0},  that their $\rrank$ is equal 1.
\begin{description}
\item [(i)] The graph $C_{2n+1}$ is critical and connected (and not a clique), so by Corollary \ref{cor-critical_rank_0}, $\rrank(C_{2n+1})\geq 1$.
\item [(ii)] The critical edges of the graph $G=\overline{C_{2n+1}}$ are those of the form $\{i,i+2\}$ (for $i\in [2n+1]$, indices taken modulo $2n+1$). Hence the subgraph $G_c$ (of critical edges) is connected (and not a clique) and thus $\rrank(\overline{C_{2n+1}})~\geq~1$. 
\end{description}
\end{example}
 
Next we  give an example of an acritical graph with $\rrank$ 1.
 \begin{example}\label{acritical-rk1}
 Consider the graph $H_9$ from  Figure \ref{figH9}. Note that $\alpha(H_9)=4$ and that $C_9$ is a critical subgraph of $H_9$ with the same stability number. Hence, by Remark \ref{rem-critical}, $\rrank(H_9)\leq \rrank(C_9)=1$. 

 \begin{figure}[H]
 \definecolor{uuuuuu}{rgb}{0.26666666666666666,0.26666666666666666,0.26666666666666666}
\definecolor{uququq}{rgb}{0.25098039215686274,0.25098039215686274,0.25098039215686274}
\begin{center}
\begin{tikzpicture}[line cap=round,line join=round,>=triangle 45,x=.4cm,y=.4cm]
\clip(-5.53,-4.11) rectangle (3.99,5.05);
\draw [line width=1pt] (-2.845262289886178,2.754689839559646)-- (-0.59,3.5755381835412496);
\draw [line width=1pt] (-0.59,3.5755381835412496)-- (1.66526228988618,2.754689839559644);
\draw [line width=1pt] (1.66526228988618,2.754689839559644)-- (2.8652622898861786,0.6762288704769914);
\draw [line width=1pt] (2.8652622898861786,0.6762288704769914)-- (2.4485066634855457,-1.6873097367523064);
\draw [line width=1pt] (2.4485066634855457,-1.6873097367523064)-- (0.61,-3.23);
\draw [line width=1pt] (0.61,-3.23)-- (-1.79,-3.23);
\draw [line width=1pt] (-1.79,-3.23)-- (-3.628506663485547,-1.687309736752305);
\draw [line width=1pt] (-3.628506663485547,-1.687309736752305)-- (-4.045262289886179,0.6762288704769942);
\draw [line width=1pt] (-4.045262289886179,0.6762288704769942)-- (-2.845262289886178,2.754689839559646);
\draw [line width=1pt] (-1.79,-3.23)-- (2.8652622898861786,0.6762288704769914);
\draw [line width=1pt] (2.4485066634855457,-1.6873097367523064)-- (-0.59,3.5755381835412496);
\draw [line width=1pt] (-0.59,3.5755381835412496)-- (-3.628506663485547,-1.687309736752305);
\draw [line width=1pt] (0.61,-3.23)-- (1.66526228988618,2.754689839559644);
\draw [line width=1pt] (-2.845262289886178,2.754689839559646)-- (-1.79,-3.23);
\begin{scriptsize}
\draw [fill=uququq] (-1.79,-3.23) circle (2.5pt);
\draw[color=uququq] (-2.33,-3.22) node {$6$};
\draw [fill=uququq] (0.61,-3.23) circle (2.5pt);
\draw[color=uququq] (1.05,-3.58) node {$5$};
\draw [fill=uuuuuu] (2.4485066634855457,-1.6873097367523064) circle (2.5pt);
\draw[color=uuuuuu] (2.93,-1.46) node {$4$};
\draw [fill=uuuuuu] (2.8652622898861786,0.6762288704769914) circle (2.5pt);
\draw[color=uuuuuu] (3.03,1.3) node {$3$};
\draw [fill=uuuuuu] (1.66526228988618,2.754689839559644) circle (2.5pt);
\draw[color=uuuuuu] (2.03,3.18) node {$2$};
\draw [fill=uuuuuu] (-0.59,3.5755381835412496) circle (2.5pt);
\draw[color=uuuuuu] (-0.23,4) node {$1$};
\draw [fill=uuuuuu] (-2.845262289886178,2.754689839559646) circle (2.5pt);
\draw[color=uuuuuu] (-3.09,3.18) node {$9$};
\draw [fill=uuuuuu] (-4.045262289886179,0.6762288704769942) circle (2.5pt);
\draw[color=uuuuuu] (-4.13,1.24) node {$8$};
\draw [fill=uuuuuu] (-3.628506663485547,-1.687309736752305) circle (2.5pt);
\draw[color=uuuuuu] (-4.03,-1.32) node {$7$};
\end{scriptsize}
\end{tikzpicture}
\end{center}
\caption{Graph $H_9$, acritical}\label{figH9}
\end{figure}
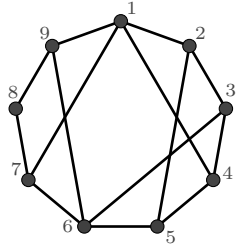

 Now, we show that $\rrank(H_9)\geq 1$. For this assume, for contradiction, that $P$ is a $\kzeroc$ for $M_{H_9}$ and let $C_1,C_2, \dots, C_9$ denote the columns of $P$. Since the sets $\{1,3,5,8\}, \{2,4,7,9\}, \{3,5,7,9\}$ and $\{2,4,6,8\}$ are stable sets of size 4 in $H_9$, by applying Lemma \ref{lem-kernel} we obtain
 \begin{align*}
 \textbf{(1) } \ C_1 + C_3 + C_5 +C_8 = 0, \quad\quad & \textbf{(2) } \ C_2 +C_4 +C_7 +C_9 =0, \\
 \textbf{(3) }\  C_3 +C_5 +C_7 +C_9 = 0 ,\quad\quad & \textbf{(4) }\ C_2 +C_4 +C_6 +C_8 =0 .\end{align*}

 By combining (2) and (4) we get that $C_7 +C_9 = C_6+C_8$. By combining (2) and (3) we get $C_2+C_4 =C_3+C_5$. Using these two identities and (2), we get $C_3 +C_5 +C_6 +C_8=0$. Finally, using (1) and the last identity we obtain  $C_6= C_1$. This implies $P_{16}=P_{11}=3>-1$, which yields a contradiction since 
 $P_{16}\le -1$ as $\{1,6\}$ is a non-edge.
 \end{example}
 
\subsection{Characterizing graphs with $\alpha(G)=2$ and $\rrank(G)=0$}\label{secalpha2}
Here we observe that the  result of Corollary \ref{cor-critical_rank_0}  holds   for all (not necessarily critical) graphs with  $\alpha(G)\le 2$.
In Section \ref{sec-large-alpha} we will show that this 
also  holds  for  acritical graphs with 
 $\alpha(G)\ge |V| -4$
(see Proposition \ref{prop-alpha-large}).

\begin{lemma}\label{lem-alpha-2}
Let $G$ be a graph with $\alpha(G)\le 2$. Then,
$\rrank(G)=0$ if and only if  $\overline \chi(G)=\alpha(G)$.
\end{lemma}

\begin{proof}

It suffices to show the `only if' part. The case $\alpha(G)=1$ is trivial. So assume $\alpha(G)=2$ and $\rrank(G)=0$.  We show that  $G$ is perfect.  For if not then, by the strong perfect graph theorem, $G$ contains  $C_5$ or $\overline{C_{2n+1}}$ ($n\geq 2$) as an induced subgraph. Both of these graphs have $\rrank$ 1 (see Example \ref{rk-odd-compl}). This contradicts Lemma \ref{indu-same-alpha} which claims that for every induced subgraph $H$ with $\alpha(H)=\alpha(G)$ we must have $\rrank(H)\leq \rrank(G)$.
\ignore{ and let $P$ be a $\kzeroc$ for $M_G$, i.e.,  $P\succeq 0$, $M_G\ge P$ and
$P_{ii}=\alpha(G)-1=1$ for all $i\in V$. As $P\succeq 0$ with diagonal entries equal to 1 it follows that $-1 \le P_{ij}\le 1$ for all $i,j\in V$. On the other hand, $P\le M_G$ implies $P_{ij}\le -1$ for all positions corresponding to non-edges. Therefore we have $P_{ij}=-1$ for every non-edge $\{i,j\}$.

As $P\succeq 0$ we may assume that $P$ is the Gram matrix of unit vectors $v_1,\ldots,v_n\in\oR^n$, i.e., $P=(v_i^Tv_j)_{i,j\in V}$.  Then,  for any two non-adjacent vertices $i,j$, we have  $v_i^Tv_j=-1$ and thus  $v_i=-v_j$.
Pick a unit vector $r\in\oR^n$ such that $r^Tv_i\ne 0$ for all $i\in V$ (such a vector exists since the kernel of $P$ is nontrivial by Lemma \ref{lem-kernel}).
 Define the sets $V_1=\{i\in V: r^Tv_i>0\}$ and $V_2=\{i\in V: r^Tv_i<0\}$. Then $V_1$ and $V_2$ are two cliques of $G$ that cover $V$.
 }
\end{proof}

\begin{example}\label{ex-Petersen}
We  give some examples  showing that   the characterization in Corollary \ref{cor-critical_rank_0} and Lemma \ref{lem-alpha-2} of rank 0 graphs as those with $\overline \chi(G)=\alpha(G)$ does not hold if $\alpha(G)\ge 3$ and $G$ has some non-critical edges.

Let $G$ be the Petersen graph. 
Then $G$ has rank 0, since $\vartheta(G)=\vartheta^{(0)}(G)=\alpha(G)\ (=4)$,  but $\overline \chi(G)= 5>\alpha(G)=4$ (see  \cite{Lo79}). 
Note  that the Petersen graph is in fact acritical. The graph $G=\overline{G_{13}}$ considered  in \cite{oddities} provides another  example with  $3=\alpha(G)=\vartheta(G) <\overline \chi(G)=4$ and $\rrank(G)=0$.

A class of counterexamples is provided by the Kneser graphs $G_{n,k}$ when $n\ge 2k+1$ and $k$ does not divide $n$. Recall $G_{n,k}$ has as  vertex set  the collection of all $k$-subsets of $[n]$, where two vertices are adjacent if the corresponding subsets are disjoint.  Note that $G_{5,2}$ is the Petersen graph.
 It has been shown by Lov\'asz \cite{Lo79,LovszKnesersCC} that 
$$\vartheta(G_{n,k})=\alpha(G_{n,k})=\binom{n-1}{k-1} \quad \text{ and } \quad \omega(G_{n,k}) \ (=\alpha(\overline{G_{n,k}})) \ =\floor{\frac{n}{k}}.$$
Therefore $\rrank(G_{n,k})=0$. However,  $\overline \chi(G_{n,k}) \ge {n\choose k}/ \floor {n/k} > {n-1\choose k-1}=\alpha(G_{n,k})$ if $k$ does not divide $n$. 

Note that $G_{n,k}$ is acritical for any $n>2k$. To see this one can use a result of Erd\"os et al. \cite{Erdos} who proved that for $n>2k$ the maximum stable sets of the Kneser graph $G_{n,k}$  are of the form $\mathcal{A}_j:=\{  S\subseteq[n]: j\in S, |S|=k\}$ for  $j\in [n]$. 
To see that $G_{n,k}$ is acritical  assume for contradiction that $\{A, B\}$ is a critical edge.  Then there exists a collection  $\mathcal{I}$ of $k$-subsets of $[n]$ such that $\mathcal{I}\cup\{A\}=\mathcal{A}_i$ and $\mathcal{I}\cup \{B\} = \mathcal{A}_j$ for $i\neq j\in[n]$. Hence, every element of $\mathcal{I}$ contains both $i$ and $j$, so that $|\mathcal{I}|\leq \binom{n-2}{k-2}$. This gives a contradiction as $|\mathcal{I}|+1=|\mathcal{A}_j|=\binom{n-1}{k-1}$.
\end{example}

\subsection{Reduction of $\rrank$ 0 graphs to the class of acritical graphs}

Here we further investigate  the structure of graphs with $\rrank$ 0. We introduce a reduction procedure, which we use  to reduce the task of checking the $\rrank$ 0 property to the same property for the class of acritical graphs. This procedure relies on  the following graph construction, which is motivated by Lemma \ref{lem-components_critical}.

\begin{definition}\label{def-Gamma}
Let $G=(V,E)$ be a graph and 
let $G_c=(V,E_c)$ be the subgraph of $G$, where $E_c$ is the set of critical edges of $G$. Let  $V_1,\dots, V_p$ denote the connected components of $G_c$.
Assume that  each of $V_1,\ldots,V_p$ is a clique in $G$. We  define the graph $\Gamma(G)$ with vertex set $\{1,2,\dots, p\}$,
 where a pair $\{i,j\}\subseteq [p]$ is an edge of $\Gamma(G)$ if $V_i\cup V_j$ is a clique of $G$. 
\end{definition}

We show that  this graph construction preserves the $\rrank$ 0 property and the stability number.

\begin{lemma}\label{lem-alpha-gamma}
Assume  $G$ is a graph with $\rrank(G)=0$ and let $\Gamma(G)$ be the graph as in Definition \ref{def-Gamma}.  Then we have:
$\rrank(\Gamma(G))=0$ and $\alpha(\Gamma(G))=\alpha(G)$.
\end{lemma}

\begin{proof}
Set $\alpha=\alpha(G)$. First, we prove that $\alpha(\Gamma(G)) \geq \alpha$. For this let $S$ be an $\alpha$-stable set in $G$ and, for each $v\in S$,  let $V_v$ denote  the connected component of $G_c$ that contains $v$.  Since each $V_i$  is a clique of $G$ (by Lemma \ref{lem-components_critical}), we have $V_v\neq V_u$ for $u\neq v\in S$ and moreover $V_u\cup V_v$ is not a clique in $G$. Hence, by defininition of the graph $\Gamma(G)$, it follows that the set $\{V_v \text{ : } v\in S\}$ provides a stable set of size $\alpha$ in $\Gamma(G)$.

Next we show that  $\rrank(\Gamma(G))=0$.
By assumption,  $\rrank(G)=0$ and thus  $M_G=P+N$, where  $P\succeq 0$, $N\geq 0$ and $P_{ii}=\alpha-1$ for all $i\in V$. As shown in the proof of Lemma   \ref{lem-components_critical}, 
the matrix $P$ has the block-form (\ref{eqP}) with respect to the partition $V=V_1\cup\ldots\cup V_p$.
Then the following $p\times p$  matrix 
$$
P':=\begin{pmatrix}
\alpha -1 & a_{12} & \cdots & a_{1p} \\
a_{21}& \alpha -1& \cdots & a_{2p}\\
\vdots  & \vdots  & \ddots & \vdots  \\
a_{p1}& a_{p2} & \cdots & \alpha-1 
\end{pmatrix}
$$
is positive semidefinite. We show that $P'\leq M_{\Gamma(G)}$, thus proving that $\Gamma(G)$ has $\rrank$ 0. As $P'\succeq 0$, we have $|a_{ij}|\le \alpha -1 \leq \alpha(\Gamma(G))-1$ for all $i,j\in [p]$.  It suffices to check that $a_{ij}\le -1$ if $\{i,j\}$ is not an edge of $\Gamma(G)$. Indeed, in this case,  $V_i\cup V_j$ is not an clique in $G$ and thus  there exist vertices $u\in V_i$ and $v \in V_j$ such that $\{u,v\}$ is not an edge in $G$, which implies  $a_{ij}=P_{uv} \le (M_G)_{uv}=-1$. 
\ignore{Therefore, we have  $P' \leq \alpha(A_{\Gamma(G)}+I)-J \leq \alpha(\Gamma(G))(A_{\Gamma(G)}+I)-J =M_{\Gamma(G)}$, where we use the fact that $\alpha\le \alpha(\Gamma(G))$.}
 This concludes the proof.

\ignore{
\textcolor{blue}{is a $\kzeroc$ for $M_{\Gamma(G)}$. Indeed, it is positive semidefinite and $|a_{ij}|\le \alpha -1 \leq \alpha(\Gamma(G))-1$ for all $i,j\in [p]$. It suffices to check that $a_{ij}\le -1$ if $\{i,j\}$ is not an edge of $\Gamma(G)$.  In this case,  $V_i\cup V_j$ is not an clique in $G$ and thus  there exist vertices $u\in V_i$ and $v \in V_j$ such that $\{u,v\}$ is not an edge in $G$, which implies  $a_{ij}=P_{uv} \le (M_G)_{uv}=-1$. Therefore, $P'$ is a $\kzeroc$ for $M_{\Gamma(G)}$, which shows that $\Gamma(G)$ has $\rrank$ 0.}\ignore{Therefore, we have  $P' \leq \alpha(A_{\Gamma(G)}+I)-J \leq \alpha(\Gamma(G))(A_{\Gamma(G)}+I)-J =M_{\Gamma(G)}$, where we use the fact that $\alpha\le \alpha(\Gamma(G))$. This shows that $\Gamma(G)$ has $\rrank$ 0.}
}

Finally, we prove $\alpha(\Gamma(G))\leq \alpha$.
 For this let $I\subseteq [p]$ be an $\alpha(\Gamma(G))$-stable set. For any $i\ne j\in I$ the set  $V_i\cup V_j$ is not a clique in $G$ and thus  $a_{ij}\leq -1$ (as observed above). Consider the principal submatrix $P'[I]$ of $P'$ indexed by $I$.
 Then we have
 $$0\le e^TP'[I]e \leq (\alpha-1)|I|-|I|(|I|-1),$$
 which implies $|I|\le \alpha$ and thus $\alpha(\Gamma(G))\le \alpha$, concluding the proof.
 \end{proof}

\begin{lemma}\label{lem-gamma_cliques}
Assume $\rrank (G)=0$. Then we have $\overline{\chi}(\Gamma(G))\geq \overline{\chi}(G)$. In particular, if $\Gamma(G$) is covered by $\alpha(\Gamma(G))$ cliques, then $G$ is covered by $\alpha(G)$ cliques.
\end{lemma}

\begin{proof}
If $C\subseteq [p]$ is a clique of $\Gamma(G)$, then $\bigcup_{i\in C} C_i$ is a clique in $G$. Therefore, if we can cover $V(\Gamma(G))=[p]$ by $k$ cliques of $\Gamma(G)$, then we can  cover $V(G)$ by $k$ cliques of $G$.
The last claim follows from the fact that $\alpha(\Gamma(G))=\alpha(G)$ (Lemma \ref{lem-alpha-gamma}).
\end{proof}

Now we  provide a partial converse to the result of Lemma \ref{lem-alpha-gamma}.

\begin{lemma}\label{lem-necessary}
Let $G=(V,E)$ be a graph and let $G_c=(V,E_c)$ be its subgraph of critical edges. Assume that the connected components  $V_1,\ldots,V_p$ of $G_c$ are cliques in  $G$ and let $\Gamma(G)$ be as in Definition \ref{def-Gamma}.  If   $\rrank(\Gamma(G))=0$ and $\alpha(\Gamma(G))\leq \alpha(G)$, then 
we have $\rrank(G)=0$.
\end{lemma}

\begin{proof}
By assumption, $\rrank(\Gamma(G))=0$. Hence  there exists a matrix $P\succeq 0$ such that $M_{\Gamma(G)}\ge  P$ and $P_{ii}=\alpha_\Gamma:=\alpha(\Gamma(G))$ for each $i\in [p]$. Write    $P$ as
$$
P=\begin{pmatrix}
\alpha_\Gamma -1 & a_{1,2} & \cdots & a_{1,p} \\
a_{2,1}& \alpha_\Gamma -1& \cdots & a_{2,p}\\
\vdots  & \vdots  & \ddots & \vdots  \\
a_{p,1}& a_{p,2} & \cdots & \alpha_\Gamma-1 
\end{pmatrix}
$$
and consider the matrix indexed by $V(G)=V_1\cup\ldots\cup V_p$ with the following block-form
        \begin{equation*}
    P'=
\begin{pmatrix}
(\alpha_\Gamma -1)J_{|V_1|} & a_{12}J_{|V_1|\times |V_2|} & \cdots & a_{1p}J_{|V_1|\times|V_p| } \\
a_{21}J_{|V_2|\times |V_1|} & (\alpha_\Gamma -1)J_{|V_2|}& \cdots & a_{2p}J_{|V_2|\times |V_p|} \\
\vdots  & \vdots  & \ddots & \vdots  \\
a_{p1}J_{|V_p|\times |V_1|} & a_{p2}J_{|V_p|\times |V_2|} & \cdots & (\alpha_\Gamma-1)J_{|V_p|} 
\end{pmatrix}.
\end{equation*}
Then, $P'\succeq 0$. We claim that $P'\le M_G$ holds.
This is true for the diagonal entries and  for the positions corresponding to edges of $G$ (since we assume $\alpha_\Gamma\le \alpha(G)$). Consider now a pair $\{u,v\}\subseteq V$ of vertices that are not adjacent in $G$. Say $u\in V_i$, $v\in V_j$. Then, as 
$V_i\cup V_j$ is not a clique in $G$, the two vertices  $i\ne j\in [p]$ are not adjacent in $\Gamma(G)$ and thus $a_{ij}\le -1$ since $P\le M_{\Gamma(G)}$.
\end{proof}

So we have shown that if we apply the $\Gamma$-operator to a graph $G$ with $\rrank$ 0,  then we obtain a new graph $\Gamma(G)$ with $\rrank$ 0, with the same stability number  and  with $|V(\Gamma(G))|\le |V(G)|$, where the inequality is strict if $G$ has  critical edges.  We may iterate this construction  until obtaining a graph without critical edges.

\begin{definition}\label{def-residual}
Let $G$ be a graph with  $\rrank (G)=0$. We define the {\em residual graph} $R(G)$ of G as the graph $\Gamma^{k}(G)$, where $k$ is the smallest integer such that $\Gamma^{k}(G)$ has no critical edge, after setting $\Gamma^{i+1}(G)=\Gamma(\Gamma^{i}(G))$ for any  $i\ge 0$.
\end{definition}

As a direct application of Lemmas \ref{lem-alpha-gamma} and \ref{lem-gamma_cliques} we obtain the following result.

\begin{lemma}\label{lem-residual}
Let $G$ be a graph with $\rrank(G)=0$ and let $R(G)$ be its residual graph as defined in Definition \ref{def-residual}. Then $R(G)$ has no critical edges and we have
$\rrank(R(G))=0$, $\alpha(R(G))=\alpha(G)$, and $\overline{\chi}(R(G))\geq \overline{\chi}(G).$
\end{lemma}

Based on the above  results,  we now present an algorithmic procedure that permits to reduce the task of checking whether a graph has $\rrank$ 0 to the same task restricted to the class of graphs with no critical edges. 

\medskip\noindent
\textbf{Algorithm}: REDUCE-TO-ACRITICAL

\smallskip\noindent
\textbf{Input:} A graph $G=(V,E)$.

\smallskip\noindent
\textbf{Output:} 
Either: $\rrank (G)\ge 1$. 
Or: the graph $R(G)$, which is acritical with $\alpha (R(G))= \alpha(G)$ and such that 
$\rrank (G)=0$ $\Longleftrightarrow$ $\rrank (R(G))=0$.
\begin{enumerate}
    \item Compute the connected components $V_1, V_2, \dots, V_p$ of the graph $G_c=(V,E_c)$, where $E_c$ is the set of critical edges of $G$. 
    \item If $V_i$ is a clique in $G$ for all $i\in [p]$, go to Step 3. Otherwise \textbf{return}:  $\rrank (G)\ge 1$.
    \item Compute the graph $\Gamma(G),$ with  set of vertices $\{1,2,\dots,p\}$ and where $\{i,j\}$ is an edge if $V_i\cup V_j$ is a clique in $G$. If 
    $\alpha(\Gamma(G))= \alpha(G)$ then go to Step 4. Otherwise  \textbf{return}: $\rrank (G) \ge 1$.
    \item If $\Gamma(G)$ is acritical then  \textbf{return}: $\Gamma(G)$.
    Otherwise  set $G=\Gamma(G)$ and go to Step 1. 
\end{enumerate}

We  verify the correctness of the output of the above algorithm. For this let us  assume the algorithm does not output $\rrank (G)\ge 1$.
In view of Definition \ref{def-residual}  the returned graph at step 4 is the residual graph $R(G)$, which is acritical by construction. In addition, in view of Step 3, we have $\alpha(R(G))= \alpha (G)$. Remains to check that $\rrank(G)=0$ if and only if $\rrank(R(G))=0$.  Indeed, the `only if' part follows using iteratively  Lemma \ref{lem-alpha-gamma}, and the `if part' folllows using 
Lemma \ref{lem-necessary}.


Observe that, if we apply the above algorithm to a class of graphs with a fixed stability number, then the algorithm runs in polynomial time, so we have shown the following theorem.

\begin{theorem}
For any fixed integer $\alpha$, the problem of deciding whether a graph with stability number $\alpha$ has $\rrank$ 0 is reducible in polynomial time to the problem of deciding whether a graph with no critical edges and stability number $\alpha$ has $\rrank$ 0.
\end{theorem}

\begin{example}
	We illustrate in Figure \ref{fig-GammaGGc} the construction of the residual graph $R(G)$ when $G$ is  the cycle $C_5$ with a pendant edge. We show the subgraph $G_c$ (consisting of the critical edges of $G$) and the graph $\Gamma(G)$, which is critical, so that $\Gamma(G)=\Gamma(G)_c$. Finally, as $\Gamma^2(G)=\overline{K_3}$ has no critical edge, we have $R(G)=\Gamma^2(G)=\overline{K_3}$. 
		Clearly, $\rrank(R(G))=0$, which shows again $\rrank(G)=0$.
	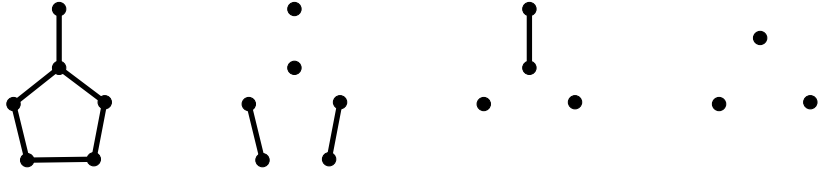
\begin{figure}[H]
	\centering
		\definecolor{ududff}{rgb}{0.30196078431372547,0.30196078431372547,1.}
		\definecolor{ffffff}{rgb}{0.,0.,0.}
		\begin{tikzpicture}[line cap=round,line join=round,>=triangle 45,x=.6cm,y=.6cm]
		\clip(-1.44,1.34) rectangle (3.42,6.28);
		\draw [line width=2.pt] (-0.02,3.18)-- (0.98,3.98);
		\draw [line width=2.pt] (0.98,3.98)-- (1.98,3.22);
		\draw [line width=2.pt] (1.74,1.96)-- (1.98,3.22);
		\draw [line width=2.pt] (0.28,1.94)-- (-0.02,3.18);
		\draw [line width=2.pt] (0.28,1.94)-- (1.74,1.96);
		\draw [line width=2.pt] (0.98,3.98)-- (0.98,5.28);
		\begin{scriptsize}
		\draw [fill=ffffff] (0.98,3.98) circle (2.5pt);
		\draw [fill=ffffff] (-0.02,3.18) circle (2.5pt);
		\draw [fill=ffffff] (0.28,1.94) circle (2.5pt);
		\draw [fill=ffffff] (1.74,1.96) circle (2.5pt);
		\draw [fill=ffffff] (1.98,3.22) circle (2.5pt);
		\draw [fill=ffffff] (0.98,5.28) circle (2.5pt);
		\end{scriptsize}
		\end{tikzpicture}
		\definecolor{ududff}{rgb}{0.30196078431372547,0.30196078431372547,1.}
		\begin{tikzpicture}[line cap=round,line join=round,>=triangle 45,x=.6cm,y=.6cm]
		\clip(-1.44,1.34) rectangle (3.42,6.28);
		\draw [line width=2.pt] (1.74,1.96)-- (1.98,3.22);
		\draw [line width=2.pt] (0.28,1.94)-- (-0.02,3.18);
		\begin{scriptsize}
		\draw [fill=ffffff] (0.98,3.98) circle (2.5pt);
		\draw [fill=ffffff] (-0.02,3.18) circle (2.5pt);
		\draw [fill=ffffff] (0.28,1.94) circle (2.5pt);
		\draw [fill=ffffff] (1.74,1.96) circle (2.5pt);
		\draw [fill=ffffff] (1.98,3.22) circle (2.5pt);
		\draw [fill=ffffff] (0.98,5.28) circle (2.5pt);
		\end{scriptsize}
		\end{tikzpicture}
		\definecolor{ududff}{rgb}{0.30196078431372547,0.30196078431372547,1.}
		\begin{tikzpicture}[line cap=round,line join=round,>=triangle 45,x=.6cm,y=.6cm]
		\clip(-1.44,1.34) rectangle (3.42,6.28);
		\draw [line width=2.pt] (0.98,3.98)-- (0.98,5.28);
		\begin{scriptsize}
		\draw [fill=ffffff] (0.98,3.98) circle (2.5pt);
		\draw [fill=ffffff] (-0.02,3.18) circle (2.5pt);
		\draw [fill=ffffff] (1.98,3.22) circle (2.5pt);
		\draw [fill=ffffff] (0.98,5.28) circle (2.5pt);
		\end{scriptsize}
		\end{tikzpicture}
		\definecolor{ududff}{rgb}{0.30196078431372547,0.30196078431372547,1.}
		\begin{tikzpicture}[line cap=round,line join=round,>=triangle 45,x=.6cm,y=.6cm]
		\clip(-1.44,1.34) rectangle (3.42,6.28);
		\begin{scriptsize}
		\draw [fill=ffffff] (-0.02,3.18) circle (2.5pt);
		\draw [fill=ffffff] (1.98,3.22) circle (2.5pt);
		\draw [fill=ffffff] (0.88,4.64) circle (2.5pt);
		\end{scriptsize}
		\end{tikzpicture}
		\caption{From right to left, the graphs  $G$,   $G_c$ (consisting of the critical edges of $G$),  $\Gamma(G)$, $R(G)=\Gamma^2(G)$} 
		\label{fig-GammaGGc}
	\end{figure}
\end{example}
\begin{remark}The results from this section can be 
adapted to the Lov\'asz parameter $\vartheta(G)$ instead of $\vartheta^{(0)}(G)$.
Recall form \cite{Lo79} that $\vartheta(G)=\alpha(G)$ if and only if there exists a positive semidefinite matrix $P$ such that 
$P_{ii}=\alpha(G)-1$ for $i\in V$  and $P_{ij}=-1$ for $\{i,j\}\in E$;   call such a $P$ a Lov\'asz-exactness certificate  for $G$. Then one can restate all results from this section by replacing the notion `$\rrank(G)=0$' by `$\vartheta(G)=\alpha(G)$'  and the notion of `$\kzeroc$' by `Lov\'asz-exactness certificate'. As a consequence,  we obtain the following analogous result: For any fixed integer $\alpha$ and for graphs with $\alpha(G)=\alpha$, the problem of deciding whether $\vartheta(G)=\alpha$ is reducible in polynomial time to the same problem for graphs with no critical edges.  
\end{remark}
\subsection{Acritical graphs with large stability number and $\rrank$ 0}\label{sec-large-alpha}

Motivated by the reduction to acritical graphs from  the previous section, we now consider   acritical graphs with large stability number. We show that if $G=(V,E)$ is acritical with  $\alpha(G)\ge |V|-4$,  then $V$ can be covered by $\alpha(G)$ cliques and thus $G$ has $\rrank$ 0. 
\begin{proposition}\label{prop-alpha-large}
Let $G=(V,E)$ be a graph and assume $\alpha(G)\ge |V|-4$. \begin{description}
\item[(i)] If $\alpha(G)\ge |V|-2$ then $\overline \chi(G)=\alpha(G)$ and thus $\rrank (G)=0$.
\item[(ii)] If $\alpha(G)=|V|-3$ then $\overline \chi(G)=\alpha(G)$ and thus $\rrank(G)=0$, unless $G$ is the disjoint union of $C_5$ and isolated nodes in which case $\rrank (G)\ge 1$ and $G$ is critical.
\item[(iii)] If $\alpha(G)=|V|-4$ and $G$ is acritical then $\overline \chi(G)=\alpha(G)$ and thus $\rrank(G)=0$. 
\end{description}
\end{proposition}

\begin{proof}
Throughout we set $\alpha=\alpha(G)$. We will use the fact that perfect graphs satisfy $\chi(\overline G)=\alpha(G)$ and their characterization via the strong perfect graph theorem. We distinguish several cases depending on the value of $n=|V|$. 

\noindent \textbf{ Case 1:} $\alpha(G)\geq |V|-2$.\\
We claim that $G$ is perfect. For, if not, then $G$ contains  an induced subgraph  $H=C_{2k+1}$ or $H=\overline{C_{2k+1}}$ ($k\ge 2$); as  every stable set of $G$ should exclude at least 3 vertices of $H$ this implies $\alpha(G)\leq |V|-3$, yielding a contradiction.

 \ignore{{\bf Case 1:} 
$\alpha(G)=|V|$: Then $G$ consists of isolated nodes and the result is clear. 

\smallskip\noindent
{\bf Case 2:}   $\alpha(G)=|V|-1$: Let $S$ be an $\alpha$-stable set and set $V\setminus S=\{x\}$. 
Then $x$ is adjacent to some node $y\in S$ and thus $V$ is covered by the  clique $\{x,y\}$ and the $\alpha-1$ singletons $\{v\}$ for $v\in S\setminus \{y\}$.

\smallskip\noindent
{\bf Case 3:}
$\alpha(G)=|V|-2$: Let $S$ be an $\alpha$-stable set and set $V\setminus{S}=\{x,y\}$. Then $x$ and $y$ are adjacent to some node in $S$.
We have two cases:\\
$\bullet$ There exist $u\ne w\in S$ such that $\{x,u\}\in E$ and  $\{y,w\}\in E$: Then $G$ is covered by the two edges $\{x,u\}$, $\{y,w\}$ and the $\alpha -2$ singletons $\{v\}$ for $v\in S\setminus \{w,z\}$.\\
$\bullet$  The two nodes $x,y$ are adjacent to a single node $u$  in $S$: $N_S(x)=N_S(y)=\{u\}$. As the set $V\setminus \{u\}\cup\{x,y\}$ is not stable since its size is $\alpha+1$, $x,y$ must be adjacent  and thus $\{x,y,z\}$ is a clique.  Hence $G$ is covered by  the clique $\{x,y,z\}$ and the $\alpha -1$ singletons $\{v\}$ for $v\in S\setminus \{z\}$.
}

\smallskip\noindent
\textbf{Case 2:} $\alpha(G)=|V|-3$.
\\
Let $S$ be an $\alpha$-stable set and set $V\setminus{S}=\{x,y,z\}$. Assume $G$ is not covered by $\alpha$ cliques, we show that $G$ is the disjoint union of $C_5$ and $n-5$ isolated vertices. As $\overline \chi(G)\ne \alpha(G)$ the graph $G$  is not perfect and thus it contains an induced subgraph $H$ which is an odd cycle $C_{2k+1}$ or its complement $\overline {C_{2k+1}}$ with $k\ge 2$. As $|V(H)\cap S|\ge 2k-2$ it follows that $\alpha(H)\ge 2k-2$. 
 If $H=C_{2k+1}$ then $\alpha(H)=k\ge 2k-2$ implies $k\le 2$ and, if $H=\overline {C_{2k+1}}$, then $\alpha(H)=2\ge 2k-2$ again implies $k\le 2$. Hence $k=2$, $H=C_5$, and $H$ contains two nodes of $S$ and the three nodes $x,y,z$. Say $H$ is the cycle $(x,u,y,w,z)$ with $u,w\in S$. If there exists a node $u_0\in S\setminus \{u,w\}$ that is adjacent to a node in $\{x,y,z\}$ then one can cover the nodes in $\{u,w,u_0,x,y,z\}$ with three edges and thus $V$ with $\alpha$ cliques, which we had excluded. Therefore, one must have $N_S(\{x,y,z\})=\{u,w\}$, which implies that  $G$ is $C_5$ together with $n-5$ isolated nodes. 
 
\smallskip\noindent
  {\bf Case 3}: $\alpha(G)=|V|-4$ and $G$ acritical.
 \\ Let $S$ be an $\alpha$-stable set and set $T=\{x,y,z,w\}=V\setminus S$. Note that every vertex of $T$ has at least two neighbors in $S$, otherwise the edge between that vertex and $S$ would be a critical edge of $G$. In addition, if there is a matching between $T$ and $S$ that covers all the nodes in $T$, then $V$ is covered by $\alpha$ cliques (the four edges of the matching and the remaining $\alpha-4$ vertices in $S$) and we are done. Hence we may now assume that there is no matching between $S$ and $T$ that covers $T$. By Hall's theorem (see \cite{Hall}), there exists $W\subseteq T$ such that $|N_S(W)|\leq |W|-1$. Then $|W|\ge 3$ since $|N_S(W)|\ge 2$. We distinguish two cases.
  
{\bf Case 3a:} First assume $|W|=3$, say $W=\{x,y,z\}$. Then $|N_S(W)|=2$, say $N_S(W)=\{u,v\}$.
     So   $N_S(x)=N_S(y)=N_S(z)=\{u,v\}$.
Since $(S\setminus \{u,v\})\cup \{x,y,z\}$ is not stable, there is an edge between the vertices $x,y,z$, say $\{x,y\}\in E$. If $w$ has a neighbor in $S$ different from $u$ and $v$, say $\{w,t\}\in E$ for $t\in S\setminus\{u,v\}$, then $V$ is covered by the cliques $\{x,y,u\}$, $\{z,v\}$, $\{w,t\}$ and the $\alpha -3$ singleton nodes in $S\setminus \{u,v,t\}$, showing $\overline\chi(G)=\alpha(G)$. So we now  assume that $N_S(w)=\{u,v\}$. Note that $\overline \chi(G)=\alpha(G)$ holds in each of the following two cases: (i) when $T$ contains a clique of size 3 (say, $\{x,y,z\}$) and (ii) when $T$ contains two disjoint edges (say, $\{x,y\},\{z,w\}\in E$) since then 
$G$ is covered by the cliques $\{x,y,z,u\}$, $\{v,w\}$ in case (i), or $\{x,y,u\},\{z,w,v\}$ in case (ii), and the $\alpha-2$ singletons in $S\setminus \{u,v\}$. So we may now assume that $T$ does not contain a triangle nor two disjoint edges. But then we reach a contradiction with the fact that each of the two sets $S\setminus\{u,v\}\cup\{x,z,w\}$ and $S\setminus\{u,v\}\cup\{y,z,w\}$ is not a stable set and thus contains an edge.

{\bf Case 3b:} Assume now $W=T=\{x,y,z,w\}$ and $|N_S(W)|=2,3$. If $|N_S(W)|=2$ then we are in the situation $N_S(x)=N_S(y)=N_S(z)=N_S(w)=\{u,v\}\subseteq S$, already considered in the previous case. So we now assume
        $|N_S(W)|=3$, say $N_S(W)=\{u,v,t\}\subseteq S$. 
        We may also assume that $G$ is not perfect (else we are done), so $G$ contains an induced subgraph $H$ which is $C_{2k+1}$ or $\overline {C_{2k+1}}$ with $k\ge 2$. As $V(H)\subseteq W\cup N_S(W)$ we have $2k+1\le 7$, so $H$ is $C_5$, $C_7$ or $\overline {C_7}$. Note $H$ cannot be $\overline {C_7}$ since $\alpha(\overline {C_7})=2$ while the set $\{u,v,t\}$ is stable. If $H=C_7$ then $G$ is $C_7$ together with $n-7$ isolated nodes, but then we contradict the assumption that $G$ is acritical. So assume now $H=C_5$. Then $|V(H)\cap S| =1$ or $2$. We distinguish these two cases:\\
$\bullet$   Assume $|V(H)\cap S|=1$, say $V(H)\cap S=\{u\}$ and $H$ is the 5-cycle
                $(x,y,z,w,u)$.
As $H$ is an induced subgraph of $G$ it follows that $\{y,u\},\{z,u\}\not\in E$. As   each of  the vertices $y$ and $z$ has at least two neighbors is $S$,  they are both  are adjacent to both $v$ and $t$ and thus $\{y,z,v\}$ and $\{y,z,t\}$ are cliques. Node $w$ is adjacent to at least two nodes in $S$ and thus $w$ is adjacent to $v$ or $t$. If $w$ is adjacent to $v$ (resp., to $t$), then $G$ is covered by the cliques $\{x,u\}$, $\{y,z,t\}$, $\{w,v\}$  (resp., $\{y,z,v\}$, $\{w,t\}$)  and the $\alpha-3$ singletons in $S\setminus \{u,v,t\}$.\\
$\bullet$ Assume $|V(H)\cap S|=2$, say $V(H)\cap S=\{u,v\}$ and $H$ is the 5-cycle $(x,y,v,z,u)$. As $x,y$ must have at least two neighbors in $S$ this implies $\{x,t\},\{y,t\}\in E$ and thus $\{x,y,t\}$ is a clique. As $w$ has at least two neighbors in $S$ it follows that $w$ is adjacent to $u$ or $v$. Say, $w$ is adjacent to $u$. Then $G$ is covered by the cliques $\{x,y,t\}$, $\{w,u\}$,
                $\{z,v\}$ and the $\alpha-3$ singletons in $S\setminus \{u,v,t\}$.               
This concludes the proof.\end{proof}

\begin{figure}[H]\label{graph_G_9}
	\centering
	\definecolor{uuuuuu}{rgb}{0.26666666666666666,0.26666666666666666,0.26666666666666666}
	\definecolor{uuuuuu}{rgb}{0.26666666666666666,0.26666666666666666,0.26666666666666666}
	\definecolor{ffffff}{rgb}{1.,1.,1.}
	\begin{tikzpicture}[line cap=round,line join=round,>=triangle 45,x=.4cm,y=.4cm]
	\fill[line width=2.pt,dash pattern=on 4pt off 4pt,color=ffffff,fill=ffffff,fill opacity=0.10000000149011612] (9.00298539751217,3.5542158406275184) -- (11.34298539751217,3.674215840627518) -- (13.058394881248194,5.270264180468297) -- (13.346554686627346,7.595552103836895) -- (12.072631638173213,9.562051548692482) -- (9.832708488335108,10.249615569580238) -- (7.674870172895163,9.336525319703865) -- (6.608793221349295,7.25002587484828) -- (7.133306887451376,4.966413514119744) -- cycle;
	\draw [line width=1.pt,dash pattern=on 4pt off 4pt,color=ffffff] (9.00298539751217,3.5542158406275184)-- (11.34298539751217,3.674215840627518);
	\draw [line width=1.pt,dash pattern=on 4pt off 4pt,color=ffffff] (11.34298539751217,3.674215840627518)-- (13.058394881248194,5.270264180468297);
	\draw [line width=1.pt,dash pattern=on 4pt off 4pt,color=ffffff] (13.058394881248194,5.270264180468297)-- (13.346554686627346,7.595552103836895);
	\draw [line width=1.pt,dash pattern=on 4pt off 4pt,color=ffffff] (13.346554686627346,7.595552103836895)-- (12.072631638173213,9.562051548692482);
	\draw [line width=1.pt,dash pattern=on 4pt off 4pt,color=ffffff] (12.072631638173213,9.562051548692482)-- (9.832708488335108,10.249615569580238);
	\draw [line width=1.pt,dash pattern=on 4pt off 4pt,color=ffffff] (9.832708488335108,10.249615569580238)-- (7.674870172895163,9.336525319703865);
	\draw [line width=1.pt,dash pattern=on 4pt off 4pt,color=ffffff] (7.674870172895163,9.336525319703865)-- (6.608793221349295,7.25002587484828);
	\draw [line width=1.pt,dash pattern=on 4pt off 4pt,color=ffffff] (6.608793221349295,7.25002587484828)-- (7.133306887451376,4.966413514119744);
	\draw [line width=1.pt,dash pattern=on 4pt off 4pt,color=ffffff] (7.133306887451376,4.966413514119744)-- (9.00298539751217,3.5542158406275184);
	\draw [line width=1.pt] (7.674870172895163,9.336525319703865)-- (9.832708488335108,10.249615569580238);
	\draw [line width=1.pt] (9.832708488335108,10.249615569580238)-- (12.072631638173213,9.562051548692482);
	\draw [line width=1.pt] (12.072631638173213,9.562051548692482)-- (13.346554686627346,7.595552103836895);
	\draw [line width=1.pt] (13.346554686627346,7.595552103836895)-- (13.058394881248194,5.270264180468297);
	\draw [line width=1.pt] (13.058394881248194,5.270264180468297)-- (11.34298539751217,3.674215840627518);
	\draw [line width=1.pt] (11.34298539751217,3.674215840627518)-- (9.00298539751217,3.5542158406275184);
	\draw [line width=1.pt] (9.00298539751217,3.5542158406275184)-- (7.133306887451376,4.966413514119744);
	\draw [line width=1.pt] (7.133306887451376,4.966413514119744)-- (6.608793221349295,7.25002587484828);
	\draw [line width=1.pt] (6.608793221349295,7.25002587484828)-- (7.674870172895163,9.336525319703865);
	\draw [line width=1.pt] (9.832708488335108,10.249615569580238)-- (7.133306887451376,4.966413514119744);
	\draw [line width=1.pt] (9.832708488335108,10.249615569580238)-- (13.058394881248194,5.270264180468297);
	\draw [line width=1.pt] (7.674870172895163,9.336525319703865)-- (9.00298539751217,3.5542158406275184);
	\draw [line width=1.pt] (11.34298539751217,3.674215840627518)-- (12.072631638173213,9.562051548692482);
	\draw [line width=1.pt,] (6.608793221349295,7.25002587484828)-- (11.34298539751217,3.674215840627518);
	\draw [line width=1.pt,] (6.608793221349295,7.25002587484828)-- (13.346554686627346,7.595552103836895);
	\draw [line width=1.pt] (9.00298539751217,3.5542158406275184)-- (13.346554686627346,7.595552103836895);
	\begin{scriptsize}
	\draw [fill=black] (9.00298539751217,3.5542158406275184) circle (2.5pt);
	\draw [fill=black] (11.34298539751217,3.674215840627518) circle (2.5pt);
	\draw[color=ffffff] (11.055409935273346,10.279677925127816) node {};
	\draw [fill=uuuuuu] (13.058394881248194,5.270264180468297) circle (2.5pt);
	\draw [fill=uuuuuu] (13.346554686627346,7.595552103836895) circle (2.5pt);
	\draw [fill=uuuuuu] (12.072631638173213,9.562051548692482) circle (2.5pt);
	\draw [fill=uuuuuu] (9.832708488335108,10.249615569580238) circle (2.5pt);
	\draw [fill=uuuuuu] (7.674870172895163,9.336525319703865) circle (2.5pt);
	\draw [fill=uuuuuu] (6.608793221349295,7.25002587484828) circle (2.5pt);
	\draw [fill=uuuuuu] (7.133306887451376,4.966413514119744) circle (2.5pt);
	\draw[color=black] (9.263581330571034,5.82398650778164) node {$f$};
	\draw[color=black] (9.952746178533463,7.768054923664748) node {$e$};
	\draw[color=black] (10.8563178680842,5.87713425177329) node {$g$};
	\end{scriptsize}
	\end{tikzpicture}
	\caption{Graph $G_9$ has  $\alpha(G_9)=4$, $\vartheta(G_9)=\vartheta^{(0)}(G_9)= 4.155$, $\overline\chi(G_9)=5$}\label{fig-G9}
\end{figure}
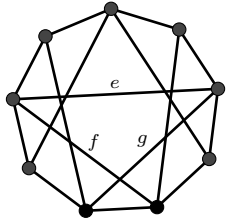

	\begin{remark}
	\begin{description}
	\item [(i)] As we just saw in Proposition \ref{prop-alpha-large} (ii),  the only graphs $G$ with  $\alpha(G)=|V|-3$  that do not have $\rrank$~0 are  of the form $G=C_5\oplus \overline{K_{n-5}}$, the disjoint union of  $C_5$ and $n-5$ isolated nodes.  In fact, we will show that $\rrank(C_5\oplus \overline {K_{n-5}})=1$  if and only if $n\le 13$ (see Corollary \ref{cor-cycle-isolated} in Section \ref{sec-isolated}).

\item [(ii)] Proposition \ref{prop-alpha-large}
shows that any acritical graph with $\alpha(G)\ge |V|-4$ satisfies $\overline \chi(G)=\alpha(G)$ and thus has $\rrank$~0. The same holds for graphs with $\alpha(G)=2$ (Lemma \ref{lem-alpha-2}). The next natural case to consider are  graphs with $\alpha(G)=3$ and $n\ge 8$ nodes.
	Polak \cite{Polak} verified (using computer) that if $G$ is an acritical graph on $8$ nodes with $\alpha(G)=3$  then $\overline \chi(G)=\alpha(G)$ holds (and thus $\rrank(G)=0$). In addition, if $G$ is acritical on 9 nodes with $\alpha(G)=3$ then $\rrank(G)=0$ holds as well (but sometimes with $\overline \chi(G)>\alpha(G)$). On the other hand there exist acritical graphs on $n=10$ nodes with $\alpha(G)=3$ that do not have $\rrank$ 0. 
	
	 \item [(iii)] There are acritical graphs $G$ with $4\le \alpha(G)\le |V|-5$ that cannot be covered by $\alpha(G)$ cliques. As a first example consider the graph $G_9$  in Figure \ref{fig-G9}, which is acritical,  with $|V|=9$, $\alpha(G_9)=4$, ${\overline{\chi}}(G_9)=5,$ and $ \vartheta(G_9)=\vartheta^{(0)}(G_9)=4.155$, and thus $\rrank(G_9)\ge 1$. Moreover,  with $e,f,g$ being the three labeled edges in $G_9$, each of the three graphs $G_9\setminus{e}, G_9\setminus{\{f,g\}}$ and $G_9\setminus{\{e,f\}}$ is acritical and satisfies $\vartheta^{(0)}(G)=\vartheta(G)>\alpha(G)$. This gives four non-isomorphic acritical graphs on 9 vertices that 
		have $\rrank$ at least 1  (and thus cannot be covered by $\alpha(G)$ cliques).  Polak \cite{Polak}  verified (using computer) that these  are  the only non-isomorphic acritical graphs on 9 vertices that do not have $\rrank$ 0.
	
\item[(iv)] Finally we use the graph $H_9$ from Example \ref{acritical-rk1} to construct a class of acritical graphs with $\chi(\overline G)>\alpha(G)$ and $\rrank(G)\ge 1$.	For any pair $(n,\alpha)$ with $4\le \alpha\le n-5$, we construct an acritical  graph $G$ on $n$ nodes with $\alpha(G)=\alpha$ and $\overline \chi(G)>\alpha(G)$. For this we let the vertex set of $G$ be partitioned as  $V=V_0\cup V_1\cup V_2$, where  $|V_0|=9$, $|V_1|= n-5-\alpha$ and $|V_2|=\alpha-4$,  and we select  the following edges: on $V_0$ we put a copy of $H_9$,  on $V_1$ we put a clique, we let  every node of $V_1$ be adjacent to every node of $V_0$, and we let $V_2$ consist of isolated nodes. Then it is easy to see that $\alpha(G)=\alpha$, $G$ is acritical and $\overline \chi(G)>\alpha(G)$. 
One can show that $\rrank(G)=\rrank(H_9\oplus \overline{K_{\alpha-4}})$. This follows from the following (easy-to-check) property: If $\{i,j\}$ is an edge and $N(i)\subseteq N(j)$ then $\rrank(G\setminus j)=\rrank(G)$. 
  Since $\rrank(H_9)=1$ one can now deduce that  $\rrank(G)\ge 1$. 
 \end{description}
\end{remark}

\section{On the impact of isolated nodes on the $\rrank$}\label{sec-k1}

As mentioned in Proposition \ref{adding-iso-rk}, if the $\rrank$ does not increase under the simple graph operation of adding an isolated node then Conjecture \ref{conj1}  holds. In \cite{GL2007} it was conjectured that adding isolated nodes indeed does not increase the $\rrank$. In this section we investigate this question and in fact disprove the latter conjecture, already for graphs with $\rrank$ 1.
For this we first  observe that critical edges provide a lot of structure on the matrices $P(i)$ ($i\in V$) appearing in $\konec$s, which can be exploited for verifying whether a graph has $\rrank$ 1. Then we investigate the impact of adding isolated nodes to certain classes of graphs $H$ with $\rrank$ 1.  First, when the subgraph of critical edges of $H$ is connected, we give an upper bound on the number of isolated nodes that can be added to $H$ while preserving the $\rrank$ 1 property (Theorem \ref{theo-ineq}). Second, we show that adding this number of isolated nodes indeed produces a graph with $\rrank$ 1 when $H$ satisfies the property $\rrank(H\setminus i^\perp)=0$ for all its nodes (Theorem \ref{adding_iso}).
 As an application 
 we are able to determine the exact number of isolated nodes that can can be added to an odd cycle $C_{2n+1}$ ($n\ge 2$) or its complement while preserving the $\rrank$ 1 property (see Corollary \ref{cor-cycle-isolated}). 
As a byproduct  we obtain that  adding an isolated node to a graph with $\rrank$ 1 can produce  a graph with  $\rrank \geq $ 2. For instance, $C_5\oplus \overline {K_8}$ has $\rrank 1$ but $C_5\oplus \overline{K_{9}}$ has $\rrank$ 2.

\subsection{Properties of the kernel of $\konec$s }
 The following results are based on the kernel property observed in Lemma \ref{kernel-k1}, which is applied to the matrices $M_G$ and permits to exploit the structure of the graph $G$.

\begin{lemma}\label{structure-P}
Let $G=(V=[n],E)$ be a graph with  $\rrank(G)=1$. Let $\{P(i): i\in V\}$ be a $\konec$ for $M_G$, let $i\in V$  and let $C_1, C_2, \dots, C_n$ denote the columns of the matrix $P(i)$. Then the following holds.
	\begin{description}
		\item[(i)]   If $S$ is a stable set of size $\alpha(G)$ and $i\in S$,  then we have $\sum_{j\in S}C_j=0$.
		\item[(ii)] If $\{i,j\}\in E$ is a critical edge of $G$ then we have $C_i=C_j$. 
		\item[(iii)] If $\alpha(G\setminus i^{\perp})=\alpha(G)-1$ and $\{l,m\}\in E$ is a critical edge of $G\setminus i^{\perp}$, then we have $C_l=C_m$.
	\end{description} 
In particular, if  $G$ is critical and $G\setminus i^{\perp}$ is critical and connected then the matrix $P(i)$ takes the form 
\begin{equation}\label{P-form}
P(i)=\left(
\begin{array}{c|c}
(\alpha-1)J_{|i^\perp|}& -1\\
\hline
-1 & \frac{1}{\alpha-1}J_{|V\setminus i^\perp|}
\end{array}
\right),
\end{equation}
where the blocks are indexed by $i^{\perp}$ and $V\setminus i^\perp$, respectively. 
\end{lemma}

\begin{proof} Set $\alpha:=\alpha(G)$ for short.
Part (i) follows directly from Lemma \ref{kernel-k1} (i), which claims $P(i)x=0$ as  $x^T M_Gx=0$ for $x=\chi^S$.

(ii)  Since the edge $\{i,j\}$ is critical  in $G$ there exists $I\subseteq V$ such that $I\cup \{i\}$ and $I\cup\{j\}$ are $\alpha$-stable sets in $G$; then,  using part (i), we get  $C_i=-\sum_{k\in I}C_k$ . Now,  observe that  the vector $y= \frac{1}{2\alpha}(\chi^{I\cup\{i\}} + \chi^{I\cup\{j\}})$ satisfies  $y^TMy=0$ (recall Eq. (\ref{eq-mini}) and Theorem \ref{minimizers-LV}). Using Lemma \ref{kernel-k1} (i), we obtain $P(i)y=0$ and thus $\frac{C_i}{2}+\frac{C_j}{2}+\sum_{k\in I}C_k=0$. Combining the two equations we get $C_i=C_j$.
\\
(iii) If $\alpha(G\setminus i^\perp)=\alpha-1$ and $\{l,m\}$ is critical in $G\setminus i^\perp$ then there exists $I\subseteq V$ with $i\in I$ such that $I\cup\{l\}$ and $I\cup\{m\}$ are stable of size $\alpha$ in $G$. Then, using again part (i), we get $C_l=-\sum_{k\in I}C_k = C_m$.   
\\
Finally, assume $G$ is critical and $G\setminus i^\perp$ is critical and connected. Since $G$ is critical, by part (ii), we have $C_i=C_j$ for all $j\in i^\perp$. Moreover,  as $G$ is critical, $i$ belongs to an $\alpha$-stable set and thus  $\alpha(G\setminus i^\perp)=\alpha-1$. Then, part (iii) can be applied and using the connectivity and criticality of $G\setminus i^\perp$ we obtain that $C_l=C_m$ for all $l,m\in V \setminus i^\perp$. Therefore, $P(i)$ takes a block structure indexed by $i^\perp$ and $V\setminus i^\perp$. Using an $\alpha$-stable set of the form $\{i\}\cup I$ (with $I\subseteq V\setminus i^\perp$) we have 
$C_i+\sum_{k\in I}C_k=0$ which, combined with the fact that $P(i)_{ii}=\alpha-1$, implies the desired structure for the matrix $P(i)$.
\end{proof}

Using Lemma \ref{structure-P} we can show that for some $\rrank$ 1 graphs the construction of the matrices $P(i)$ in a $\konec$ is in fact unique. 
We already saw that this is the case for the 5-cycle in Example \ref{example-5-cycle}, we now extend this to any critical graph with $\alpha(G)=2$ and to the graph $C_5\oplus i_0$. We show in Figure \ref{Fig4} an example of critical graph with stability number $\alpha(G)=2$; of course  $C_5$ is another such  example.

\begin{figure}[H]
\centering
\definecolor{uuuuuu}{rgb}{0.26666666666666666,0.26666666666666666,0.26666666666666666}
\definecolor{uququq}{rgb}{0.25098039215686274,0.25098039215686274,0.25098039215686274}
\begin{tikzpicture}[line cap=round,line join=round,>=triangle 45,x=.3cm,y=.3cm]
\draw [line width=1pt] (-12.75,9.853998519614898)-- (-9.63,9.853998519614898);
\draw [line width=1pt] (-9.63,9.853998519614898)-- (-8.07,7.151999259807448);
\draw [line width=1pt] (-8.07,7.151999259807448)-- (-9.63,4.45);
\draw [line width=1pt] (-9.63,4.45)-- (-12.75,4.45);
\draw [line width=1pt] (-12.75,4.45)-- (-14.31,7.15199925980745);
\draw [line width=1pt] (-14.31,7.15199925980745)-- (-12.75,9.853998519614898);
\draw [line width=1pt] (-14.31,7.15199925980745)-- (-9.63,9.853998519614898);
\draw [line width=1pt] (-12.75,9.853998519614898)-- (-8.07,7.151999259807448);
\begin{scriptsize}
\draw [fill=uququq] (-12.75,4.45) circle (2.5pt);
\draw [fill=uququq] (-9.63,4.45) circle (2.5pt);
\draw [fill=uuuuuu] (-8.07,7.151999259807448) circle (2.5pt);
\draw [fill=uuuuuu] (-9.63,9.853998519614898) circle (2.5pt);
\draw [fill=uuuuuu] (-12.75,9.853998519614898) circle (2.5pt);
\draw [fill=uuuuuu] (-14.31,7.15199925980745) circle (2.5pt);
\end{scriptsize}
\end{tikzpicture}
\caption{A critical graph with stability number $2$}\label{Fig4}

\ignore{
\begin{minipage}[t]{0.5\linewidth}
\centering

\definecolor{uuuuuu}{rgb}{0.26666666666666666,0.26666666666666666,0.26666666666666666}
\definecolor{uququq}{rgb}{0.25098039215686274,0.25098039215686274,0.25098039215686274}
\begin{tikzpicture}[line cap=round,line join=round,>=triangle 45,x=.5cm,y=.5cm]
\clip(-14.726363636362857,4.301074380165286) rectangle (-7.156115702477376,11.044876033057852);
\draw [line width=1pt] (-13.578509963461387,8.326106157908189)-- (-11.19999999999867,10.054194799648329);
\draw [line width=1pt] (-11.19999999999867,10.054194799648329)-- (-8.821490036535954,8.326106157908187);
\draw [line width=1pt] (-8.821490036535954,8.326106157908187)-- (-9.72999999999844,5.53);
\draw [line width=1pt] (-12.6699999999989,5.53)-- (-9.72999999999844,5.53);
\draw [line width=1pt] (-12.6699999999989,5.53)-- (-13.578509963461387,8.326106157908189);
\begin{scriptsize}
\draw [fill=uququq] (-12.6699999999989,5.53) circle (2.5pt);
\draw [fill=uququq] (-9.72999999999844,5.53) circle (2.5pt);
\draw [fill=uuuuuu] (-8.821490036535954,8.326106157908187) circle (2.5pt);
\draw [fill=uuuuuu] (-11.19999999999867,10.054194799648329) circle (2.5pt);
\draw [fill=uuuuuu] (-13.578509963461387,8.326106157908189) circle (2.5pt);
\end{scriptsize}
\end{tikzpicture}
\caption{Critical, $\alpha=2$}\label{FigC5}
\end{minipage}
}
\end{figure}

\begin{example}
Let $G=(V,E)$ be a critical graph with $\alpha(G)=2$.  
Then $M_G\in \MK^{(1)}$ (recall Theorem \ref{theoGLstrong}).
Let $\{P(i):  i\in V\}$ be a $\konec$ for $M_G$. We show that the matrices $P(i)$ are uniquely determined using Lemma~\ref{P-form}. Indeed,   as  $\alpha(G)=2$, for any $i\in V$ the graph $G\setminus i^\perp$ is a clique  and thus it is critical and connected with $\alpha(G\setminus i^\perp)=1=\alpha(G)-1$. Hence Lemma \ref{structure-P} can be applied  and we obtain that for every $i\in V$ the matrix $P(i)$ takes the form (\ref{P-form}).
\end{example}

\begin{example}
Let $G=C_5\oplus i_0 =([5]\cup\{i_0\},E)$,  so that $G\setminus i_0^\perp=C_5$. As $\alpha(G\setminus i_0^\perp)=\alpha(G)-1=2$ and $G\setminus i_0^\perp$ is critical and connected, by Lemma \ref{structure-P} we conclude that the matrix $P(i_0)$ takes the form (\ref{P-form}) (also displayed below). In particular we have 
$P(i_0)_{ij}=1/2$ and $P(i_0)_{i_0i}=-1$ for all $i,j\in [5]$. We now  show that for any $i\in [5]$ also the matrices $P(i)$ are uniquely determined; by symmetry it suffices to show this for  matrix $P(1)$.

Since $G$ is critical,  by Lemma \ref{structure-P} (ii) (applied to the edges $\{1,2\}$ and $\{1,5\}$),  the columns of $P(1)$ indexed by nodes 1, 2, and 5  are identical. As the edge $\{3,4\}$ is critical in the graph $G\setminus 1^\perp$, by Lemma \ref{structure-P} (iii), also  the two columns of $P(1)$  indexed by 3 and 4  are identical. This implies that the matrix $P(1)$ takes a block structure indexed by the partition of its index set into $\{1,2,5\}$, $\{3,4\}$ and $\{i_0\}$.  By Lemma \ref{lemK1} we have $P(1)_{11}=\alpha-1=2$, $2P(1)_{1,i_0}+P(i_0)_{1,1}=\alpha-3=0$ and $P(1)_{i_0,i_0}+2P(i_0)_{1,i_0}=\alpha-3=0$ .  Combining with  the fact that $P(i_0)_{11}=\frac{1}{2}$ and $P(i_0)_{1,i_0}=-1$ we obtain that  $P(1)_{1,i_0}=-\frac{1}{4}$ and $P(1)_{i_0,i_0}=2$. Finally, since $\{1,3,i_0\}$ is stable, using Lemma~\ref{structure-P}(i) we obtain that the columns indexed by 1,3 and $i_0$ sum up to 0, which enables to complete  the rest of the matrix $P(1)$, whose shape is shown below.
   $$
P(i_0)=\bordermatrix{& i_0 & [5] \cr
i_0 & 2 & -1 \cr
[5]& -1 & 1/2
},
\quad \quad P(1)
=\bordermatrix{
& i_0 & \{3,4\} & \{1,2,5\} \cr
i_0& 2 & -7/4 & -1/4 \cr
\{3,4\}& -7/4  & 7/2 & -7/4 \cr
\{1,2,5\} & -1/4 & -7/4& 2
}.
\ignore{=\left(
\begin{array}{c|c|c}
2 & -\frac{7}{4} & -\frac{1}{4} \\
-\frac{7}{4} &  \frac{7}{2} &
      -\frac{7}{4}\\
    -\frac{1}{4} & -\frac{7}{4} & 2
\end{array}\right) .
}
$$
\end{example}


\begin{lemma}\label{equality-support}
Let $G=(V, E)$ be a graph with $M_G\in \MK_n^{(1)}$ and let $P(1), P(2), \dots, P(n)$ be a $\konec$ for $M_G$. Assume that for $S\subseteq V$ the induced subgraph $G[S]$ is the disjoint union of $\alpha(G)$ cliques.. Then, for any $\{i,j,k\}\subseteq S$, we have 
$$P(i)_{jk}+P(j)_{ik} + P(k)_{ij}= (M_G)_{ij} + (M_G)_{jk} + (M_G)_{ik} = \alpha(G) \ |E(\{i,j,k\})|-3.$$
\end{lemma}

	\begin{proof}
	By Theorem \ref{minimizers-LV} there exists $x\in \Delta_n$ such that $x^TM_Gx=0$ and $\supp(x)=S$. Then Lemma \ref{kernel-k1} (ii) gives  the desired result.
	\end{proof}

\begin{example}
Consider the graph $G_8$ shown in Figure \ref{G_8}, which is critical with $\alpha(G_8)=3$.
We show that $\rrank(G_8)\geq 2$ (which was verified numerically in \cite{PVZ2007}).  Assume for contradiction  that $M_G\in \MK_8^{(1)}$ and let $P(1), \dots, P(8)$ be a $\konec$ for $M_G$. Notice that for $i=1,2,3,4$ the graph $G\setminus i^\perp =C_5$ is critical and connected. Hence, by Lemma \ref{structure-P}, the matrices $P(1), P(2), P(3)$ and $P(4)$ take the form (\ref{P-form}) and thus we have 
$P(1)_{23} + P(2)_{13} + P(3)_{12}=-1-1+{1\over 2}=-\frac{3}{2}.$
However, as the graph induced by $\{1,2,3, 6\}$ is the disjoint union of $\alpha(G)$ cliques, in view of Lemma \ref{equality-support} one should have 
$P(1)_{23} + P(2)_{13} + P(3)_{12}=3 \times 1-3=0,$
so we reach a contradiction.

\begin{figure}[H]
\begin{minipage}{0.5\linewidth}
\begin{center}
\definecolor{uuuuuu}{rgb}{0.26666666666666666,0.26666666666666666,0.26666666666666666}
\definecolor{uququq}{rgb}{0.25098039215686274,0.25098039215686274,0.25098039215686274}
\begin{tikzpicture}[line cap=round,line join=round,>=triangle 45,x=1cm,y=1.1cm]
\draw [line width=1pt] (-4.646286937826003,4.055390483334499)-- (-3.175176360218357,4.055390483334498);
\draw [line width=1pt] (-3.175176360218357,4.055390483334498)-- (-1.7040657826107104,4.055390483334498);
\draw [line width=1pt] (-1.7040657826107104,4.055390483334498)-- (-0.23295520500306388,4.055390483334498);
\draw [line width=1pt] (-0.23295520500306388,4.055390483334498)-- (-0.23295520500306421,2.5842799057268513);
\draw [line width=1pt] (-0.23295520500306421,2.5842799057268513)-- (-1.7040657826107104,2.5842799057268513);
\draw [line width=1pt] (-1.7040657826107104,2.5842799057268513)-- (-3.175176360218357,2.584279905726851);
\draw [line width=1pt] (-3.175176360218357,2.584279905726851)-- (-4.646286937826004,2.5842799057268513);
\draw [line width=1pt] (-4.646286937826004,2.5842799057268513)-- (-4.646286937826003,4.055390483334499);
\draw [line width=1pt] (-3.175176360218357,2.584279905726851)-- (-1.7040657826107104,4.055390483334498);
\draw [line width=1pt] (-1.7040657826107104,2.5842799057268513)-- (-3.175176360218357,4.055390483334498);
\begin{scriptsize}
\draw [fill=uququq] (-3.175176360218357,2.584279905726851) circle (2.5pt);
\draw[color=uququq] (-3.2236459534925346,2.856830868682590387) node {$7$};
\draw [fill=uququq] (-3.175176360218357,4.055390483334498) circle (2.5pt);
\draw[color=uququq] (-3.139416152837559,4.30150458658616776) node {$5$};
\draw [fill=uuuuuu] (-4.646286937826003,4.055390483334499) circle (2.5pt);
\draw[color=uuuuuu] (-4.511221090598189,4.30150458658616776) node {$1$};
\draw [fill=uuuuuu] (-4.646286937826004,2.5842799057268513) circle (2.5pt);
\draw[color=uuuuuu] (-4.505890553535837215,2.856830868682590387) node {$2$};
\draw [fill=uuuuuu] (-1.7040657826107104,2.5842799057268513) circle (2.5pt);
\draw[color=uuuuuu] (-1.667611215076929,2.85678653720856145) node {$8$};
\draw [fill=uuuuuu] (-1.7040657826107104,4.055390483334498) circle (2.5pt);
\draw[color=uuuuuu] (-1.667611215076929,4.30150458658616776) node {$6$};
\draw [fill=uuuuuu] (-0.23295520500306421,2.5842799057268513) circle (2.5pt);
\draw[color=uuuuuu] (-0.0919580627731629918,2.85678653720856145) node {$3$};
\draw [fill=uuuuuu] (-0.23295520500306388,4.055390483334498) circle (2.5pt);
\draw[color=uuuuuu] (-0.0919580627731629918,4.30150458658616776) node {$4$};
\end{scriptsize}
\end{tikzpicture}
\caption{The graph $G_8$ (critical, $\alpha(G_8)=3$)}\label{G_8}
\end{center}
\end{minipage}
\begin{minipage}{0.5\linewidth}
\begin{center}
\definecolor{uuuuuu}{rgb}{0.26666666666666666,0.26666666666666666,0.26666666666666666}
\definecolor{uququq}{rgb}{0.25098039215686274,0.25098039215686274,0.25098039215686274}
\begin{tikzpicture}[line cap=round,line join=round,>=triangle 45,x=.65cm,y=.65cm]
\draw [line width=1pt] (-4.425993445928173,4.861386383616873)-- (-3.552562930533183,6.087775325088295);
\draw [line width=1pt] (-3.552562930533183,6.087775325088295)-- (-2.1162928635717506,5.636068566519027);
\draw [line width=1pt] (-2.1162928635717506,5.636068566519027)-- (-0.6888193998400561,6.114848486031521);
\draw [line width=1pt] (-0.6888193998400561,6.114848486031521)-- (0.20764092179593363,4.9051916782059095);
\draw [line width=1pt] (-0.665789593599056,3.678802736734488)-- (0.20764092179593363,4.9051916782059095);
\draw [line width=1pt] (-0.665789593599056,3.678802736734488)-- (-2.1020596605604878,4.130509495303755);
\draw [line width=1pt] (-3.529533124292182,3.651729575791261)-- (-2.1020596605604878,4.130509495303755);
\draw [line width=1pt] (-3.529533124292182,3.651729575791261)-- (-4.425993445928173,4.861386383616873);
\draw [line width=1pt] (-2.1020596605604878,4.130509495303755)-- (-2.1162928635717506,5.636068566519027);
\draw [line width=1pt] (-3.552562930533183,6.087775325088295)-- (-0.6888193998400561,6.114848486031521);
\begin{scriptsize}
\draw [fill=uququq] (-3.529533124292182,3.651729575791261) circle (2.5pt);
\draw[color=uququq] (-3.4940679450690344,4.11197042244953469) node {$7$};
\draw [fill=uququq] (-2.1020596605604878,4.130509495303755) circle (2.5pt);
\draw[color=uququq] (-1.86659448133734,4.4258221644659635) node {$6$};
\draw [fill=uuuuuu] (-2.1162928635717506,5.636068566519027) circle (2.5pt);
\draw[color=uuuuuu] (-2.0798939235460203,5.933092281449744) node {$8$};
\draw [fill=uuuuuu] (-3.552562930533183,6.087775325088295) circle (2.5pt);
\draw[color=uuuuuu] (-3.516233682083502,6.485273316544877) node {$2$};
\draw [fill=uuuuuu] (-4.425993445928173,4.861386383616873) circle (2.5pt);
\draw[color=uuuuuu] (-4.389563720453514,5.2572914859433865) node {$1$};
\draw [fill=uuuuuu] (-0.665789593599056,3.678802736734488) circle (2.5pt);
\draw[color=uuuuuu] (-0.6302547227998587,4.111764112937083) node {$5$};
\draw [fill=uuuuuu] (0.20764092179593363,4.9051916782059095) circle (2.5pt);
\draw[color=uuuuuu] (0.24307531557015308,5.3016229599723205) node {$4$};
\draw [fill=uuuuuu] (-0.6888193998400561,6.114848486031521) circle (2.5pt);
\draw[color=uuuuuu] (-0.652420459814326,6.511872200962238) node {$3$};
\end{scriptsize}
\end{tikzpicture}\caption{The graph $H_8$ (critical, $\alpha(H_8)=3$)}
\end{center}
\end{minipage}
\end{figure}
It can also be shown that $\rrank(H_8)\geq 2$, the arguments are similar but technical so we omit them.
So we have $\rrank(G_8)=\rrank(H_8)=2$.   In fact,  $G_8$ and $H_8$ are the {\em only} critical graphs on  8 nodes with $\rrank=2$. 
To see this one can use the list of critical graphs on 8 nodes from \cite{Small} and verify that all of them have $\rrank$ at most 1 except $G_8$ and $H_8$. Note also that,  as observed in \cite{PVZ2007}, any  graph with at most 7 nodes has $\rrank$ at most 1. 
\end{example}

\subsection{Adding isolated nodes to graphs with $\rrank$ 1}\label{sec-isolated}

As we saw in  Section \ref{pre-iso-critical}, it is crucial to understand the role of isolated nodes for  the $\rrank$ of a graph (recall Proposition \ref{adding-iso-rk}).
Here we investigate  how  many isolated nodes can be added to a  graph $H$  with $\rrank$ 1 (and satisfying certain properties) without increasing its $\rrank$. As an application we show that adding an isolated node to some $\rrank$ 1 graphs may produce a graph with $\rrank \ge 2$. 

Throughout this section we  consider  a graph of the form $G=H\oplus \overline{K_{\alpha-k}}$, where $H=(V,E)$ has  $\alpha(H)=:k$, so that $\alpha(G)=\alpha$.  Here $\alpha$ and $k$ are integers such that $\alpha\ge k\ge 2$. Note that, if $k=1$, then $H$ is a clique and thus $G$ has $\rrank$ 0 for any $\alpha$. We let $W$ denote the set of  isolated nodes that are added to $H$, so that $|W|=\alpha-k$ and $G=(V\cup W, E)$. We also consider  the subgraph $H_c=(V,E_c)$ of $H$, where $E_c$ is the set of critical edges of $H$. 

\subsubsection{Upper bound on the number of isolated nodes}
 First, we investigate some necessary conditions about the parameters $\alpha$ and $k$ that must hold if $\rrank(G)= 1$. 

\begin{theorem}\label{theo-ineq}
Given integers $\alpha>k\ge 2$, let $H=(V,E)$ be a graph  with $\alpha(H)=k$ and let $G=H\oplus \overline{K_{\alpha-k}}$. Assume the graph $H_c=(V,E_c)$ is connected and  $\rrank(G)=1$. Then 
we have
\begin{equation}\label{ineq}
    \alpha \leq \frac{k(k+3)}{k-1} = k + 4+\frac{4}{k-1}.
\end{equation}

\end{theorem}

The rest of the section is devoted to the proof of Theorem \ref{theo-ineq}.
Throughout we assume that $G$ and $H$ are as defined in Theorem \ref{theo-ineq}, so $M_G=\alpha(A_G+I)-J\in \mathcal{K}_n^{(1)}$. 
We will use the following result of  Dobre and Vera \cite{DV}, which shows  the existence of a  $\konec$ for $M_G$, which inherits some  symmetry properties of $M_G$.
 
\begin{proposition}[\cite{DV}]\label{k_1_sym}
Assume that $M\in \mathcal{K}_n^{(1)}$.   Then $M$ has a $\konec$ $P(1),\dots,P(n)$ satisfying the following symmetry property: $\sigma(P(i))=P(\sigma(i))$ for all $\sigma\in \text{Sym}(n)$ such that  $\sigma(M)=M$.
\end{proposition}

So let $\{P(i): i\in V\}$ be a  $\konec$ for $M_G$ satisfying the symmetry property from Proposition \ref{k_1_sym}.   In particular, since any permutation $\sigma\in \text{Sym}(W)$ of the isolated nodes leaves the graph $G$ invariant it follows that
\begin{equation}\label{eqsym}
\sigma(P(i))=P(\sigma(i)), \text{ i.e., } \ P(i)_{\sigma(j)\sigma(k)}=P(\sigma(i))_{jk} \  \text{ for all } \sigma\in \text{Sym}(W) \text{ and } j,k \in V\cup W.
\end{equation}
We will use this symmetry property repeatedly in the proof.  We  mention a simple identity that follows as a direct application of Lemma \ref{equality-support}, which we will  also repeatedly use in the rest of the section:

 \begin{equation}\label{-3}
 P(i)_{jk} + P(j)_{ik} + P(k)_{ij} = -3 \hspace{.5cm} \text{if  $\{i,j,k\}$ is contained in a stable set of $G$ with size $\alpha(G).$}
 \end{equation}
 
 Now we prove some preliminary lemmas and we end with Lemma \ref{matrix_P_v-psd}, which will directly imply Theorem \ref{theo-ineq}. We start with a general property about the structure of the submatrices $P(i)[W]$ when  $i\in W$ is an isolated node.

\begin{lemma}\label{P[W]}
There exists a scalar $b\in \mathbb{R}$ such that  
 the following holds:
\begin{description}
    \item [(i)] $P(i)_{ij}=b$ \ for all distinct $i,j\in W$, 
    \item  [(ii)] $P(i)_{jj}=\alpha-2b-3$ \ for all distinct $i,j\in W$, 
    \item [(iii)] $P(i)_{jk}=-1$ \ for all distinct $i,j,k\in W$. 
\end{description}
\end{lemma}

\begin{proof}
Let $i,j,k \in W$ be  distinct (isolated) nodes and set $b:=P(i)_{ij}$. 
First we show that $b$ does not depend on the choice of $i,j\in W$. For this we use the symmetry property from (\ref{eqsym}), which claims 
$P(i)_{\sigma(i)\sigma(j)}=P(\sigma(i))_{ij}$ for any $\sigma\in \text{Sym}(W)$. Using the permutation  $\sigma=(j,k)$ we get $P(i)_{ij}=P(i)_{ik}=b$, and using  $\sigma=(i,j)$ we get $P(i)_{ij}=P(j)_{ij}=b$,
thus showing (i). Now, by Lemma \ref{lemK1}, we have $P(i)_{jj}+2P(j)_{ij}=\alpha-3$, which implies  $P(i)_{jj}=\alpha-2b-3$ and thus  (ii) holds.
Using again (\ref{eqsym}) with $\sigma=(i,k)$ we obtain $P(i)_{\sigma(i),\sigma(j)} =P(\sigma(i))_{i,j}$, and thus $P(i)_{jk}=P(k)_{ij}$. Similarly, using $\sigma=(i,j)$ we get $P(i)_{\sigma(i)\sigma(k)}=P(\sigma(i))_{ik}$ and thus 
$P(i)_{jk}=P(j)_{ik}$. By using Eq. (\ref{-3}) for the nodes $i,j,k$ we obtain $P(i)_{jk}=P(j)_{ik}=P(k)_{ij}=-1$, thus showing (iii).
\end{proof}

So we know the structure of the submatrix $P(i)[W]$ when $i\in W$ is an isolated node. When the graph $H_c$ (consisting of the critical edges of $H$) is connected we can also derive the structure of the rest of the matrix $P(i)$.

\begin{lemma}\label{struc-1}
Assume the graph $H_c$ is connected. 
Then the matrix $P(i)$  takes the form
    $$P(i) =\bordermatrix{& i & W\setminus{i}& V \cr
i & & & d\dots d\cr
& & & \beta J\cr
W\setminus{i} & & & \cr
& d & & \cr
V&\vdots &\beta J& \gamma J\cr
& d& &  & } \quad \text{ for  all }  i\in W,
$$
where the blocks are indexed by $\{i\}, W\setminus \{i\}$ and $V$, respectively, and the scalars $d,\beta,\gamma$ are given by
$$d=\dfrac{b(k+1)+1-\alpha-b\alpha}{k},
\quad \beta= \dfrac{b+1-k}{k},
\quad \gamma=\dfrac{\alpha-k}{k}.
$$
\end{lemma}

\begin{proof}
Fix an isolated node $i\in W$.
Let $\{l,m\}\in E_c$ be a critical edge of $H$. By Lemma \ref{structure-P}(iii) we get that the two columns of $P(i)$ indexed by $l$ and $m$ are identical. Since $H_c$ is connected it follows that the columns of $P(i)$ indexed by $V$ are all identical. From this follows that 
$P(i)[V]$ (the submatrix of $P(i)$ indexed by $V$) is of the form $\gamma_i J$ for some scalar $\gamma_i$ and there exists a vector $b_i\in \oR^{W}$ such that 
$P(i)_{j h}=(b_i)_j$ for all $j\in W, h\in V$.

Let $j\ne k\in W\setminus \{i\}$ and $v\in V$. By applying Eq. (\ref{eqsym}) to the permutation  $\sigma=(j,k)$, we obtain  
$P(i)_{\sigma(k)\sigma(v)}=P(\sigma(i))_{kv}$, and thus $P(i)_{jv}=P(i)_{kv}$. Therefore, the entries of $b_i$ indexed by $W\setminus \{i\}$ are all equal, say to a scalar $\beta_i$. We set $d_i:= (b_i)_i$. 
Finally we show that the scalars $\beta_i, \gamma_i, d_i$ in fact do not depend on the choice of $i\in W$ and take the values claimed in the lemma.

For this consider an $\alpha$-stable set $S$ of $G$. Then $i\in S$ and thus, 
 by Lemma \ref{structure-P}(i), the  columns of $P(i)$ indexed by $S$ sum up to zero. Using the identities of Lemma \ref{P[W]} combined with the above facts on the remaining entries of $P(i)$,  we obtain 
\begin{align*}
   (\alpha-1)+(\alpha-k-1)b+kd_i=0 &\ \Longrightarrow \ d_i=\dfrac{b(k+1)+1-\alpha-b\alpha}{k},\\
   b-(\alpha-k-2)+(\alpha-2b-3)+k\beta_i=0 &\ \Longrightarrow\  \beta_i=\dfrac{b+1-k}{k},\\
    d_i+(\alpha-k-1)\beta_i+k\gamma_i=0 &\ \Longrightarrow\  \gamma_i=\dfrac{\alpha-k}{k}.
\end{align*}
This concludes the proof. \end{proof}

We now are able to conclude some properties on the structure of the matrices $P(j)$ for $j\in V$.

\begin{lemma}\label{P_v}
Assume $H_c$ is connected. For any $v\in V$  the submatrix $P(v)[W\cup\{v\}]$ takes the form
\begin{equation}\label{matrix_P_v}
        P(v)[W\cup\{v\}]=
\left(
\begin{array}{c|c}
M_b & \frac{\alpha}{2}-\frac{\alpha}{2k}-1 \\
\hline
\frac{\alpha}{2}-\frac{\alpha}{2k}-1 & \alpha-1
\end{array}
\right),
\end{equation}
where the blocks are indexed by $W$ and $\{v\}$, respectively.
Here,  $b\in \mathbb{R}$ is the constant from  Lemma \ref{P[W]} and the matrix $M_b$ is indexed by $V$ and takes the form
\begin{equation}\label{eqMb}
M_b=\begin{pmatrix}
 a & c &\cdots & c \\
c & a& \cdots & c \\
\vdots  & \vdots  & \ddots & \vdots  \\
c & c & \cdots & a 
\end{pmatrix},
\quad \text{ with }  \quad a= \alpha -3 - {2\over k}\Big(b(k+1)+1-\alpha -b\alpha\Big),\quad  c=-1-\frac{2}{k}(b+1).
\end{equation}
\end{lemma}

\begin{proof}
Consider an isolated node $i\in W$. By Lemma \ref{lemK1} we have  $P(v)_{ii}+2P(i)_{iv}=\alpha-3$. This implies 
$P(v)_{ii}=\alpha-3-2d$ and thus $P(v)_{ii} =\alpha-3-\frac{2}{k}(b(k+1)+1-\alpha-b\alpha)$, which shows the claimed value of $a$. 

Consider $i\ne j\in W$. As $H_c$ is connected, $v$ belongs to a critical edge and thus  there exists an $\alpha$-stable set of $G$ that contains $i,j,v$. Then,  by  (\ref{-3}), we have 
 $P(i)_{vj}+P(j)_{iv}+P(v)_{ij}=-3$. This implies  $P(v)_{ij} =-3-2\beta$ and thus  $P(v)_{ij}=-1-\frac{2(b+1)}{k}$, which shows the claimed value of $c$. 

Let $i\in W$. Using again  Lemma \ref{lemK1} we get $2P(v)_{iv}+P(i)_{vv}=\alpha-3$. Hence $P(v)_{iv}=\frac{\alpha-3-\gamma}{2}$, which implies   $P(v)_{iv}=\frac{\alpha}{2}-\frac{\alpha}{2k}-1$. This completes the proof.
\end{proof}

The following lemma gives necessary and sufficient conditions for the matrix in Eq. (\ref{matrix_P_v}) to be positive semidefinite. In particular, the part (ii) of the lemma shows Theorem \ref{theo-ineq}.

\begin{lemma}\label{matrix_P_v-psd}
The matrix in Eq. (\ref{matrix_P_v}) is positive semidefinite if and only if the following two conditions hold:
\begin{description}
\item[(i)]$a\geq c$,
\item[(ii)]$\alpha\leq k+4+\frac{4}{k-1}$.
\end{description}
\end{lemma}

\begin{proof}
By taking the Schur complement of the matrix $P(v)[W\cup\{v\}]$ in (\ref{matrix_P_v}) with respect to its  $(v,v)$-entry  we obtain that $P(v)[W\cup\{v\}]\succeq 0$ if and only if 
\begin{center}
    $(a-c)I_{\alpha-k}+(c-\frac{1}{\alpha-1}(\frac{\alpha}{2}-\frac{\alpha}{2k}-1)^2)J_{\alpha-k}\succeq 0$.
\end{center}
This  happens if and only $a\geq c$ and the following inequality holds:
    $$a-c+ (\alpha-k)\Big(c-\frac{1}{\alpha-1}\big(\frac{\alpha}{2}-\frac{\alpha}{2k}-1\big)^2\Big)\geq 0.$$
We show that this last inequality holds if and only if (ii) holds. First, notice that $a+(\alpha-k-1)c=k$. Indeed, if we see this expression as a polynomial in $b$ then 
 the coefficient of $b$ is 
$$ -\frac{2}{k}(k-\alpha+1)- \frac{2}{k}(\alpha-k-1)=0$$
and the constant coefficient is 
    $$\alpha-3-\frac{2(1-\alpha)}{k}+(\alpha-k-1)(-1-\frac{2}{k})=k.$$
Therefore, the inequality $a-c+(\alpha-k)(c-\frac{1}{\alpha-1}(\frac{\alpha}{2}-\frac{\alpha}{2k}-1)^2)\geq 0$ is equivalent to 
    $$k(\alpha-1)\geq (\alpha-k)\Big(\frac{\alpha}{2}-\frac{\alpha}{2k}-1\Big)^2.$$
Multiplying both sides by $4k^2$, this  is equivalent to
    $$4k^3(\alpha-1)\geq (\alpha-k)(\alpha(k-1)-2k)^2$$
    $$ \Longleftrightarrow \ \ 4 k^3\alpha-4k^3 \geq (\alpha-k)(\alpha^2(k-1)^2-4k(k-1)\alpha+4k^2)
       $$
    $$\Longleftrightarrow \ \ 4k^3\alpha -4k^3\geq \alpha^3(k-1)^2  -\alpha^2k(k-1)^2        -4\alpha^2k(k-1)       +4\alpha k^3  -4k^3$$
after cancelling terms in the right hand side. Cancelling terms at both sides and dividing by $\alpha^2(k-1)$ (as $k\geq2$) we obtain 
    $\alpha(k-1)-4 k- k(k-1)\leq 0$
and thus the desired inequality (ii).
\end{proof}

\subsubsection{Lower bound on the number of isolated nodes  }

In Theorem \ref{theo-ineq}  we saw that if the subgraph $H_c$ of critical edges of $H$ is connected and the graph $G=H\oplus \overline{K_{\alpha-k}}$, obtained by adding $\alpha-k$ isolated nodes to a graph $H$ with $\alpha(H)=k$,  has $\rrank$ 1, then the parameters $\alpha$ and $k$ must satisfy the inequality  (\ref{ineq}).  So this gives the upper bound $\alpha-k\le 4 +4/(k-1)$ on the number of isolated nodes that can be added while preserving the $\rrank$ 1 property.

Here we provide some classes of graphs $H$ for which it is indeed possible to add this  maximum number of isolated nodes and preserve the $\rrank$ 1 property.  Hence, for these graphs, we  characterize the exact number of isolated nodes that can be added while preserving the $\rrank$ 1 property.

We begin with a preliminary lemma which we will use for our main result below.

\begin{lemma}\label{matrix-psd}
Assume  $\alpha\geq k\geq 2$ satisfy the inequality (\ref{ineq}), and let $M:=\alpha I_{\alpha-k}-J_{\alpha-k}$. Then
    $$
\left(
\begin{array}{c|c}
M & \frac{\alpha}{2}-\frac{\alpha}{2k}-1 \\
\hline
\frac{\alpha}{2}-\frac{\alpha}{2k}-1 & \alpha-1
\end{array}
\right)\succeq 0. $$
\end{lemma}

\begin{proof}
The above matrix corresponds to the matrix   in Eq. (\ref{matrix_P_v}) with $b=-1$, which gives $a=\alpha-1$ and $c=-1$, so that  $M=M_b=M_{-1}$. As $a\ge c$, using   Lemma \ref{matrix_P_v-psd}, we get the desired result.
\end{proof}

\begin{theorem}\label{adding_iso}
Given integers $\alpha\ge k\ge 2$, let $H=(V,E)$ be a graph with $\alpha(H)=k$ and let $G=H\oplus \overline{K_{\alpha-k}}$.
Assume that  $\rrank(H\setminus i^\perp)=0$ for all $i\in V$ and $\rrank(H)=1$. 
In addition assume that $\alpha,k$ satisfy the inequality (\ref{ineq}). Then we have $\rrank(G)= 1$.
\end{theorem}

\begin{proof}
We construct a $\mathcal K^{(1)}$-certificate for the matrix $M_G$. That is, we  construct matrices  $P(i)$ (for $i\in W\cup V$) that satisfy the properties of Lemma \ref{lemK1}. Recall  Remark \ref{remK0} where we observed that it will suffice to show that the matrices $P(i)$ belong to the cone $\mathcal K^{(0)}$.
For this consider the following construction (inspired from \cite{GL2007}), where we set $M:=\alpha I_{\alpha-k}-J_{\alpha-k}$.
\\
$\bullet$ 
For $i\in V$, we set
    $$
P(i)=
\left(
\begin{array}{c|c|c}
M & \frac{\alpha}{2}-\frac{\alpha}{2k}-1 & -\frac{\alpha}{2k}-1 \\
\hline
\frac{\alpha}{2}-\frac{\alpha}{2k}-1 &  \alpha-1 &
      \frac{\alpha}{2}-1-\frac{\alpha^2}{2k} \\
      \hline
    -\frac{\alpha}{2k}-1 & \frac{\alpha}{2}-1-\frac{\alpha^2}{2k} &
    \left\{ \begin{array}{ll}
      \frac{\alpha^2}{k}-1 & \text{ if }i\simeq j \\
      -1 & \text{ else}
      \end{array} \right.
\end{array}
\right),
$$
where the blocks are indexed by $W$, $i^\perp$ and  $V \setminus i^\perp$, respectively. Here the notation $i\simeq j$ means that the nodes $i$ and $j$ are equal or adjacent in $G$.
\\
$\bullet$ For $i\in W$, we set
$$
   P(i)=\left( \begin{array}{c|c}
M & -1\\
\hline
-1 &  \frac{\alpha-k}{k}J 
\end{array} \right),
$$
where the blocks are indexed by $W$ and $V$, respectively. 

First we show that the matrix $P(i)$ is positive semidefinite for all $i\in  W$.  Indeed, deleting repeated rows and columns and taking the Schur complement with respect to the lower right corner we obtain that $P(i)\succeq 0$ if and only if $0\preceq M -\frac{k}{\alpha-k}J_{\alpha-k}=\alpha I_{\alpha-k}-\frac{\alpha}{\alpha-k}J_{\alpha-k}$,  which is indeed true.

Next we show that 
$P(i)\in \mathcal K^{(0)}$ for all $i\in V$. For this, let $i\in V$ and observe that  we can decompose $P(i)$ as $P(i)=Q(i)+\frac{\alpha^2}{k(k-1)}R(i)$, where
$$
Q(i)=\left(
\begin{array}{c|c|c}
M & \frac{\alpha}{2}-\frac{\alpha}{2k}-1 & -\frac{\alpha}{2k}-1 \\
\hline
\frac{\alpha}{2}-\frac{\alpha}{2k}-1 &  \alpha-1 &
      \frac{\alpha}{2}-1-\frac{\alpha^2}{2k} \\
      \hline
    -\frac{\alpha}{2k}-1 & \frac{\alpha}{2}-1-\frac{\alpha^2}{2k} & \frac{\alpha^2}{k(k-1)}-1
\end{array} 
\right) \ \text{ and } \ R(i)=\left(
\begin{array}{c|c|c}
0 & 0 & 0 \\
\hline
0 &  0 &
      0 \\
      \hline
    0 & 0 & 
 \left\{   \begin{array}{ll}
      k-2 & \text{ if }i\simeq j \\
      -1 & \text{ else}
     \end{array} \right.
\end{array} 
\right),$$
whose blocks are indexed by $W$, $i^\perp$ and $V\setminus i^\perp$, respectively.
We prove that 
$Q(i)\succeq 0$ and $R(i)\in \mathcal K^{(0)}$.

 First, we show that $Q(i)$ is positive semidefinite. By Lemma \ref{matrix-psd} we know
that the submatrix $Q(i)[W\cup i^\perp]$ is positive semidefinite. 
We will now show that any column $C_v$ of $Q(i)$ indexed by a node $v\in V\setminus i^\perp$ (in the third block) can be expressed as a linear combination of the columns $C_u$ indexed by $u\in W\cup \{i\}$ (in the first two blocks), which directly implies that $Q(i)\succeq 0$. Namely, one can show  $C_v= \frac{1}{1-k}(\sum_{j\in W} C_j +C_i)=:C$ by direct inspection of the entries:\\
- for the entries  indexed by $u\in I$ we have:
$$
C_u= {1\over 1-k} \Big(\alpha-1-(\alpha-k-1)+{\alpha\over 2}-{\alpha\over 2k}-1\Big) = -1-{\alpha\over 2k}= (C_v)_u,
$$
- for the entries indexed by $u\in i^\perp$ we have:
$$C_u= {1\over 1-k}\Big( (\alpha-k)\Big({\alpha\over 2}-{\alpha\over 2k}-1\Big)+\alpha -1\Big) = -1+{\alpha\over 2}-{\alpha^2\over 2k},$$
- for the entries indexed by $u\in V\setminus i^\perp$ we have:
$$C_u= {1\over 1-k}\Big( (\alpha-k)\Big(-{\alpha\over 2k}-1\Big) +{\alpha\over 2}-1-{\alpha^2\over 2k}\Big) = {\alpha^2\over k(k-1)}-1.$$

Now we show that $R(i)\in \MK^{(0)}$. For this note that $\alpha(H\setminus i^\perp)\leq k-1$, which implies the entry-wise inequality 
$$\left(
\begin{array}{c|c}
0 & 0\\
\hline
0 & M_{H\setminus i^{\perp}}
\end{array}
\right) \leq R(i).$$
By hypothesis $M_{H\setminus i^\perp}\in\MK^{(0)}$. Since adding zero row/columns preserve membership  in $\MK^{(0)}$ we get that $R(i)\in \MK^{(0)}$.

To conclude the proof we now need  to check that the linear constraints (ii)-(iv) of Lemma \ref{lemK1} are satisfied by the matrices $P(i)$. This is direct case checking, but we give the details for clarity.
\begin{description}
\item [Identity (ii):] $P(v)_{vv}=\alpha-1=(M_G)_{vv}$ \ for all $v\in V\cup I$.
	\item [Identity (iii):] We check that $P(u)_{vv}+2P(v)_{uv} = (M_G)_{vv}+2(M_G)_{uv}$ for all $u \ne v\in I\cup V$:   
	\begin{description}
		\item [-] for $i,j\in I$, we have   $P(i)_{jj}+2P(j)_{ij}=\alpha-1-2=\alpha-3,$
		\item [-] for $i\in I$, $v\in V$, we have 
			\begin{itemize}
				\item $P(i)_{vv}+2P(v)_{iv}=\frac{\alpha-k}{k}+\alpha-\frac{\alpha}{k}-2=\alpha-3,$
				\item $P(v)_{ii}+2P(i)_{iv}=\alpha-1-2=\alpha-3,$
			\end{itemize}
			
		\item  [-] for $u,v\in V$, we have
			\begin{itemize}
				\item if $\{u,v\}\in E$ then $P(u)_{vv}+2P(v)_{uv}=3\alpha-3,$
				\item if  $\{u,v\} \notin E$ then  $P(u)_{vv}+2P(v)_{uv}=\frac{\alpha^2}{k}-1+2(\frac{\alpha}{2}-1-\frac{\alpha^2}{2k})=\alpha-3.$
			\end{itemize}	
		\end{description}

\item[Inequality (iv):] We check $P(u)_{vw}+P(v)_{uw}+P(w)_{uv}\le (M_G)_{uv}+(M_G)_{vw}+(M_G)_{vw}$ for distinct $u,v,w\in I\cup V$:
	\begin{description}
		\item[-] for  $i,j,k\in I$ we have   $P(i)_{jk}+P(j)_{ik}+P(k)_{ij}=-3,$
		 \item [-] for  $i,j\in I, v\in V$ we have $P(i)_{jv}+P(j)_{iv}+P(v)_{ij}=-3,$
		\item [-] for $i\in I, u,v\in V$ we have
			\begin{itemize}
        				\item if $\{u,v\}\notin E$ then   $P(i)_{uv}+P(u)_{iv}+P(v)_{iu}=\frac{\alpha-k}{k}-2(\frac{\alpha}{2k}+1)=-3,$
        				\item  if $\{u,v\}\in E$ then $P(i)_{uv}+P(u)_{iv}+P(v)_{iu}=\frac{\alpha-k}{k}+2(\frac{\alpha}{2}-\frac{\alpha}{2k}-1)=\alpha-3,$
			\end{itemize}
		\item[-] for $u,v,w\in V$ we have
			\begin{itemize}
        				\item if $\{u,v\},\{v,w\}, \{u,w\}\in E$ then  $P(u)_{vw}+P(v)_{uw}+P(w)_{uv}=3(\alpha-1),$
        				\item if $\{u,v\},\{u,w\}\in E$, $\{v,w\}\notin E$ then $P(u)_{vw}+P(v)_{uw}+P(w)_{uv}=\alpha-1+2(\frac{\alpha}{2}-1-\frac{\alpha^2}{2k})=2\alpha-3-\frac{\alpha^2}{2k}\leq 2\alpha-3,$
        				\item if $\{u,v\}\in E$, $\{u,w\},\{v,w\}\notin E$ then $P(u)_{vw}+P(v)_{uw}+P(w)_{uv}=2(\frac{\alpha}{2}-1-\frac{\alpha^2}{2k})+\frac{\alpha^2}{k}-1=\alpha-3,$
        				\item if $\{u,v\},\{u,w\}, \{v,w\} \notin E$ then $P(u)_{vw}+P(v)_{uw}+P(w)_{uv}=-3.$
    			\end{itemize}
	\end{description} 
\end{description}
This completes the proof.
\end{proof}

We now give  some examples of graphs for which the conditions of Theorem \ref{theo-ineq} and \ref{adding_iso} hold, so that we are able to compute the exact number of isolated nodes that can be added with  the resulting graph still having  $\rrank$ 1.

\begin{coro} \label{cor-cycle-isolated}
For any integer  $n\geq 2$ the following holds: 
\begin{description} 
	\item [(i)] $\rrank(C_{2n+1} \oplus \overline{K_m})=1$ if and only if $m\leq 4+\frac{4}{n-1}$. 
	\item [(ii)] $\rrank(\overline{C_{2n+1} }\oplus \overline{K_m})=1$ if and only if $m\leq 8$.
\end{description}
\end{coro}

\begin{proof}
Consider the graph $H=C_{2n+1}$ or $H=\overline{C_{2n+1}}$.
As pointed out in Example \ref{odd-cycles}, 
 $H$ satisfies  the property: $\rrank (H\setminus i^{\perp})=0$ for all $i\in V$, and thus   the assumption  of Theorem \ref{adding_iso} holds. For $H=C_{2n+1}$ the inequality (\ref{ineq}) reads $m\le 4 + {4\over n-1}$ and, for $H=\overline {C_{2n+1}}$, it reads $m\le 8$. So the `if part' in both (i), (ii) follows as a direct application of 
Theorem \ref{adding_iso}.

The `only if' part in both (i), (ii) follows as a direct application of  Theorem \ref{theo-ineq},  since the graph $C_{2n+1}$ is critical while the subgraph of critical edges of  $\overline{C_{2n+1}}$ is a connected graph.
\end{proof}

\begin{coro}\label{corH2}
Assume  $H$ is a graph with $\overline{\chi}(H)>\alpha(H)=2$.
 Then,  $\rrank(H\oplus \overline{K_m})=1$ if and only if $m\leq 8$.
\end{coro}

\begin{proof}
The `if' part follows directly from Theorem \ref{adding_iso}. Now we prove that $\rrank(H\oplus \overline{K_m})\geq 2$ for $m\geq 9$. Since $H$ is not perfect it contains the graph $H_0=C_5$ or $H_0=\overline{C_{2n+1}}$ ($n\ge 2$) as an induced subgraph. Hence, $H_0\oplus \overline{K_m}$ is an induced subgraph of $H\oplus \overline{K_m}$ with the same stability number. Then, by Lemma \ref{indu-same-alpha}, $\rrank(H\oplus \overline{K_m})\geq \rrank(H_0\oplus \overline{K_m})\geq 2,$ where the last inequality follows from Corollary \ref{cor-cycle-isolated}.
\end{proof}

\begin{coro}
Consider a graph $H$ and a connected component $H_0$ of $H$. Assume $\alpha(H_0) \geq 2$ and the subgraph $(H_0)_c$ of critical edges of $H_0$ is connected. 
Then the following holds:
\begin{description}
\item[(i)]  If $\alpha(H)\geq \alpha(H_0)+9$ then $\rrank(H)\geq 2$.
\item[(ii)]  If $\alpha(H)\leq \alpha(H_0)+8$ then $\rrank(H\oplus \overline{K_s})\geq 2$ for $s\geq 9-\alpha(H)+\alpha(H_0)$.
\end{description}
\end{coro}

\begin{proof}
By Corollary \ref{cor-critical_rank_0} we know   $\rrank(H_0)\ge 1$.
Pick a stable set $W\subseteq V(H\setminus H_0)$ such that $\alpha(H_0\oplus W)=\alpha(H)$, i.e., $|W|=\alpha(H)-\alpha(H_0)$. Then $H_0\oplus W$ is an induced subgraph of $H$ with the same stability number as $H$. Then, by Lemma \ref{indu-same-alpha}, $\rrank(H_0\oplus W \oplus \overline{K_s})\leq \rrank(H\oplus \overline{K_s})$ for any $s\geq 0$. By applying  Corollary \ref{corH2} to the graph $H_0$,  we obtain that $\rrank(H_0 \oplus W\oplus \overline{K_s})\geq 2$ if $s+|W|\geq 9$. From these facts  (i) and (ii) now follow easily.
\end{proof}


\begin{thebibliography}{10}

\bibitem{Bomze-survey}
I.M. Bomze, M. Budinich, P.M. Pardalos and M. Pellilo.
The maximum clique problem.
 In: Du D.Z., Pardalos P.M. (eds), {\em Handbook of Combinatorial Optimization}. Springer, Boston, MA, 1999.



\bibitem{BY2020}
S.~Burer and Y.~Ye.
\newblock Exact semidefinite formulations for a class of (random and
  non-random) nonconvex quadratic programs.
\newblock {\em Mathematical Programming}, 181:1--17, 2020.


\bibitem{SPGT}
M. Chudnovsky, N. Robertson, P. Seymour, and R. Thomas, Robin.
The strong perfect graph theorem.
{\em Annals of Mathematics}, 164(1):51--229,
2006.

\bibitem{dKLP}
E.~de~Klerk, M.~Laurent, and P.~Parrilo.
\newblock {\em On the equivalence of algebraic approaches to the minimization
  of forms on the simplex}.
\newblock Number 312 in LNCIS, pages 121--133. Springer, 2005.

\bibitem{dKP2002}
E.~de~Klerk and D.~Pasechnik.
\newblock Approximation of the stability number of a graph via copositive
  programming.
\newblock {\em SIAM Journal on Optimization}, 12:875--892,  2002.

\bibitem{Diananda}
P.H. Diananda. On non-negative forms in real variables some or all of which are non-negative.
{\em Proceedings of the Cambridge Philosophical Society}, 58:17--25, 1962.

\bibitem{DDGH}
P.J.C. Dickinson, M. D\"ur, L. Gijben and R. Hildebrand.
\newblock Scaling relationship between the copositive cone and Parrilo`s first level approximation.
\newblock{\em Optimization Letters}, 7(8):1669--1679, 2013.

\bibitem{DV}
C. Dobre and J. Vera
\newblock Exploiting symmetry in copositive programs via semidefinite hierarchies.
\newblock{\em Mathematical Programming},151:659--680, 2015.


\bibitem{Erdos}
P. Erd\'os, Chao Ko and R. Rado
\newblock Intersection theorems for systems of finite sets.
\newblock{\em The Quarterly Journal of Mathematics}, 12(1):313--320, 1961.


\bibitem{GHPR}
L.E. Gibbons, D.W. Hearn, P.M. Pardalos and M.V. Ramana.
Continuous characterizations of the maximum clique problem.
{\em Mathematics of Operations Research}, 22(3):754--768, 1997.

\bibitem{GY2021}
Y.G. G\"okmen and E.A. Yildirim.
On standard quadratic programs with exact and inexact doubly nonnegative relaxations.
{\em Mathematical Programming}, 2021. https://doi.org/10.1007/s10107-020-01611-0

\bibitem{GL2007}
N.~Gvozdenovi{\'c} and M.~Laurent.
\newblock Semidefinite bounds for the stability number of a graph via sums of
  squares of polynomials.
\newblock {\em Mathematical Programming}, 110:145--173, 2007.



\bibitem{Hall}
P.~Hall.
\newblock On representatives of subsets.
\newblock {\em Journal of the London Mathematical Society}, s1-10(1):26--30,
  1935.
  
 \bibitem{Hildebrand}
  R. Hildebrand.
  The extreme rays of the $5\times 5$ copositive cone.
  {\em Linear Algebra and its Applications}, 437(7):1538--1547, 2012.
  
\bibitem{Hossain}
  A. Hossain, E. Lopez, S. Halper, D. Cetnar, A. Reis, D. Strickland, E.Klavins, and H. Salis.
Automated design of thousands of nonrepetitive parts for engineering stable genetic systems.
  \newblock{\em Nature Biotechnology}, 38:1466-1475, 2020.
  
  \bibitem{Karp}
R.~Karp.
\newblock {\em Reducibility among combinatorial problems}, pages 85--103.
\newblock Plenum Press, New York, 1972.  
  
\bibitem{Lasserre2001b}
J.B. Lasserre.
\newblock {\em An explicit exact {SDP} relaxation for nonlinear 0-1 programs},
  volume 2081 of {\em Lecture Notes in Computer Science}, pages 293--303.
\newblock 2001.
  
  \bibitem{Laurent2003}
M.~Laurent.
\newblock A comparison of the {S}herali-{A}dams, {L}ov\'asz-{S}chrijver and
  {L}asserre relaxations for 0-1 programming.
\newblock {\em Mathematics of Operations Research}, 28(3):470--496, 2003.
  
  \bibitem{LV2021}
M. Laurent and L.F. Vargas.
\newblock Finite convergence of sum-of-squares hierarchies for the stability number of a graph.
\newblock https://arxiv.org/abs/2103.01574.
  
\bibitem{LovszKnesersCC}
L.~Lov{\'a}sz.
\newblock Kneser's conjecture, chromatic number, and homotopy.
\newblock {\em J. Comb. Theory, Ser. A}, 25:319--324, 1978.  

\bibitem{Lo79}
L.~Lov\'asz.
\newblock On the {S}hannon capacity of a graph.
\newblock {\em IEEE Trans. Inform. Theory}, 25:1--7, 1979.

\bibitem{oddities}
L.~Mancinska and D.~Roberson.
\newblock Oddities of quantum colorings.
\newblock {\em Baltic Journal of Modern Computing}, 4:846--859, 2016.

\bibitem{motzkin}
T.~Motzkin and E.~Straus.
\newblock Maxima for graphs and a new proof of a theorem of {T}ur{\'a}n.
\newblock {\em Canadian Journal of Mathematics}, 17:533--540, 1965.


\bibitem{Parrilo-thesis-2000}
P.A. Parrilo.
\newblock {S}tructured {S}emidefinite {P}rograms and {S}emialgebraic {G}eometry
  {M}ethods in {R}obustness and {O}ptimization.
\newblock PhD thesis, California Institute of Technology, 2000.

\bibitem{PVZ2007}
 J.~Pe\~{n}a, J.~Vera, and L.F. Zuluaga.
\newblock Computing the stability number of a graph via linear and semidefinite
  programming.
\newblock {\em SIAM Journal on Optimization}, 18(1):87--105, 2007.

\bibitem{Polak}
S. Polak. Personal communication, 2021.

\bibitem{PorkolabKhachiyan}
L.~Porkolab and L.~Khachiyan.
\newblock On the complexity of semidefinite programs.
\newblock {\em Journal of Global Optimization}, 10: 351-365, 1997.

\bibitem{Ramana}
M.~Ramana.
\newblock An exact duality theory for semidefinite programming and it complexity implications.
\newblock {\em Mathematical Programming}, 77: 129-162, 1997.

\bibitem{Small}
B.J. Small.
\newblock On alpha-critical graphs and their construction.
\newblock PhD thesis, Washington State University, 2015.


\bibitem{Vera}
J. Vera. Personal communication, 2021.

\bibitem{PVZ2007}
J.~Vera, J.~Pena, and L.F. Zuluaga.
\newblock Computing the stability number of a graph via linear and semidefinite
  programming.
\newblock {\em SIAM Journal on Optimization}, 18(1):87--105, 2007.

\bibitem{WK2021}
A.L. Wang and F.~Kilin{\c{c}}-Karzan.
\newblock On the tightness of {SDP} relaxations of {QCQP}s.
\newblock {\em Mathematical Programming}, 2021. https://doi.org/10.1007/s10107-020-01589-9

\bibitem{WuHao}
Q. Wu and J.-K. Hao.
A review on algorithms for maximum clique problems.
{\em European Journal of Operations Research}, 242:693--709, 2015.



\end{thebibliography}
\end{document}